\documentclass[11pt]{amsart}

\numberwithin{equation}{section}
\usepackage[utf8]{inputenc}

\usepackage{amsmath,latexsym,amsfonts,amssymb,amsthm,amscd,bbm}
\usepackage{graphicx,epsf,psfrag}
\usepackage{stmaryrd}
\usepackage{tikz}
\usepackage{xcolor}
\usepackage{hyperref}
\hypersetup{
    colorlinks=true,                         
    linkcolor=blue, 
    citecolor=blue, 
    urlcolor=blue}

\usepackage{enumitem}
\setlist{leftmargin=*,itemsep=3pt}

\usepackage{caption}
\newcommand{\E}{\mathbf{E}}
\newcommand{\bE}{\mathbf{E}}
\newcommand{\un}{\mathbbm{1}}
\newcommand{\ind}{\mathbbm{1}}
\newcommand{\PP}{\mathbf{P}}
\newcommand{\bP}{\mathbf{P}}
\newcommand\ent[1]{\left\lfloor #1 \right \rfloor}

\newcommand{\dd}{\mathrm{d}}
\newcommand{\f}{\textsc{f}}
\newcommand{\I}{\mathrm{I}}
\renewcommand{\tilde}{\widetilde}

\newcommand{\bbN}{\mathbb{N}}

\newcommand{\ga}{\alpha}
\newcommand{\gb}{\beta}

\newcommand{\gep}{\varepsilon}       
\newcommand{\gp}{\varphi}

\DeclareMathOperator{\argmin}{{argmin}}
\DeclareMathOperator{\argmax}{{argmax}}
\newcommand{\suptwo}[2]{\sup_{\substack{#1 \\ #2}}} 
\newcommand{\sumtwo}[2]{\sum_{\substack{#1 \\ #2}}} 

\newtheorem{theorem}{Theorem}[section]
\newtheorem{proposition}[theorem]{Proposition}
\newtheorem{remark}[theorem]{Remark}
\newtheorem{example}[theorem]{Example}
\newtheorem{assumption}{Assumption}
\newtheorem{lemma}[theorem]{Lemma}
\newtheorem{corollary}[theorem]{Corollary}

\title[Wetting on a wall or in a well]{Wetting on a wall and wetting in a well:\\
Overview of equilibrium properties}

\author[Q. Berger]{Quentin Berger}
\address{LPSM, Sorbonne Universit\'e, 75005 Paris, France}
\email{quentin.berger@sorbonne-universite.fr}

\address{DMA, \'Ecole Normale Sup\'erieure, Universit\'e PSL, 75005 Paris, France}
\email{qberger@dma.ens.fr}

\author[B. Massouli\'e]{Brune Massouli\'e}
\address{CEREMADE, Universit\'e Paris-Dauphine, UMR 7534, CEREMADE, PSL University,
75016 Paris, France}
\email{brune.massoulie@dauphine.eu}

\begin{document}

\begin{abstract}
We study the wetting model, which considers a random walk constrained to remain above a hard wall, but with additional pinning potential for each contact with the wall.
This model is known to exhibit a \textit{wetting} phase transition, from a localized phase (with trajectories pinned to the wall) to a delocalized phase (with unpinned trajectories).
As a preamble, we take the opportunity to present an overview of the model, collecting and complementing well-known and other folklore results.
Then, we investigate a version with elevated boundary conditions, which has been studied in various contexts both in the physics and the mathematics literature; it can alternatively be seen as a wetting model in a square well.
We complement here existing results, focusing on the equilibrium properties of the model, for a general underlying random walk (in the domain of attraction of a stable law).
First, we compute the free energy and give some properties of the phase diagram; interestingly, we find that, in addition to the wetting transition, a so-called \textit{saturation} phase transition may occur.
Then, in the so-called Cram\'er's region, we find an exact asymptotic equivalent of the partition function, together with a (local) central limit theorem for the fluctuations of the left-most and right-most pinned points, jointly with the number of contacts at the bottom of the well.
\\[3pt]
{\it MSC2020 AMS classification:} Primary: 82B41, 60K35 ; Secondary: 82D60, 60F05.
\end{abstract}

\maketitle

\vspace{-1.7\baselineskip}
\begin{figure}[htbp]
    \centering
    \includegraphics[scale=0.68]{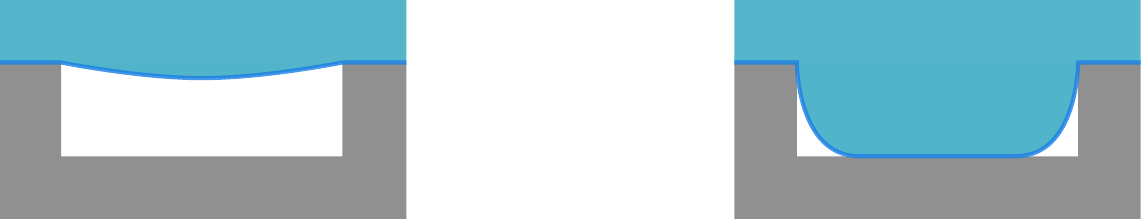}
    \caption{\small Wetting of a square well by a droplet. On the left, the droplet does not reach the bottom of the well (dry phase); on the right the droplet reaches the bottom of the well (wet phase).}
    \label{fig:wettingwell}
 \vspace{-1.6\baselineskip}
\end{figure}

\setcounter{tocdepth}{1}
\tableofcontents

\section{Introduction}

The goal of the present article is to study the random walk wetting model, which has been introduced following the work of \cite{Abr80,Abr81} as an effective model for interfaces in the 2D Ising model, and quickly drew a lot of interest from the physics community, see for instance the seminal paper~\cite{Fish84}.
The model is based on a random walk trajectory which is constrained to remain above a hard wall, with a reward (or penalty) for every point of contact with the wall.
It can be used as an effective model for a one-dimensional interface interacting with a substrate, or a model for the adsorption of polymers on a substrate.
Over the last decades, the wetting model has been widely studied, both in the physics and in the mathematical literature, either in its homogeneous or disordered version.
We refer to~\cite{Giac07,Giac10} for a general overview of the model and its relation to the pinning model.

Here, we study a version of the (homogeneous) wetting model with elevated boundary conditions, in analogy with~\cite{dCD87,Pat98} (for the Solid-On-Solid model), \cite{AdCD89,PV99} (for interfaces in the 2D Ising model) or \cite{BFO09} (for $d$-dimensional Gaussian random walks pinned to a subspace). 
This can also be seen as a model for the wetting of a (square) well or cavity, see~Figure~\ref{fig:wettingwell}.
One motivation of the latter interpretation is the description of a wetting transition for droplets on a grooved surface, see \cite{dCDH11}.
In \cite{dCDH11,LT15}, the authors study this model from a dynamical point of view: in particular, they show that for some region of parameters in the phase diagram, the model exhibits a metastable transition (the dry phase might be stable whereas the equilibrium measure is in the wet phase, or vice versa).
In the present paper, we investigate further the properties of the equilibrium measure.

We consider a general setting where the underlying random walk is in the domain of attraction of some $\alpha$-stable law, with $\alpha\in (0,2]$. 
To our knowledge, the literature mostly focuses on specific cases, such as simple or Gaussian random walks in \cite{LT15}, resp.~\cite{BFO09}, integer or real valued Solid-On-Solid (SOS) models in \cite{Pat98}, resp.~\cite{dCD87}.
This also seems to generalize the setting usually considered for the standard wetting model, see e.g.~\cite[\S1.3]{Giac07} and references therein (authors consider random walks in the domain of attraction of the normal law, \textit{i.e.}\ $\alpha=2$).
The wetting model is known to undergo a phase transition between a delocalized (or unpinned) phase where trajectories wander away from the wall, and a localized (or pinned) phase where trajectories stick the wall.
The remarkable fact noticed by Fisher~\cite{Fish84} (for the simple random walk) is that the critical point and the critical behavior of this model can be understood precisely.
For convenience (and because our setting is a bit more general than in the literature), we provide in Section~\ref{sec:wetting} a complete introduction of the model.

\subsubsection*{Overview of the results}
Our first main result consists in providing an expression for the free energy of the wetting model with elevated boundary condition: it is based on an optimization between the free energy of the wetting model and the large deviation rate functions of the underlying random walk (in analogy with~\cite[Prop.~2.1]{LT15}).
As a byproduct of the proof, we also obtain the location of the left-most and right-most point of contact of the interface with the bottom of the well.
We also provide some properties of the phase diagram. 
We show in particular that if the rate functions are non-trivial (which occurs only in the case where the random walk admits a finite variance, \textit{i.e.}\ $\alpha=2$), then the wetting phase transition is always of first order, in contrast with what happens in the standard wetting model where the phase transition is of second order.
Interestingly, depending on the underlying random walk, another (or two other) phase transition may occur: it corresponds to a \emph{saturation} transition, when the bottom corners of the well become wet.

Our second set of results deals with the so-called Cram\'er regime of the model, when the optimization problem that defines the free energy attains its maximum at points where the rate functions are strictly convex.
In that case, we are able to obtain an exact asymptotic for the partition function, together with sharp trajectory estimates.
In particular, we are able to prove a (local) central limit theorem for the left and right-most points of contact, jointly with the number of contacts of the walk with the bottom of the well.

\subsubsection*{Overview of the article}
Let us now briefly present the organization of the paper.

In Section~\ref{sec:wetting}, we introduce the wetting model in a general setting; we consider either discrete or continuous random walks, in the domain of attraction of a stable law.
Since this setting is not usually the one presented in the literature (as far as we know), we give a complete overview of the model: we explain its relation to the pinning model, give an implicit expression for the free energy and provide the critical behavior of the model.
This section mostly collects well-known facts on the model and could serve as a rather complete overview of the homogeneous model; to complete the overview, we also include some (folklore) \emph{integrable} models where the free energy admits an explicit expression, in Appendix~\ref{sec:examples}.
Readers familiar with the subject may skip this section entirely.

In Section~\ref{sec:wettingwell}, we turn our attention to the wetting model with elevated boundary conditions or \textit{wetting in a (square) well}.
We present our main results on the free energy  (see Theorem~\ref{th:freeenergy}), together with some properties of the phase diagram, in particular regarding the critical curve.
We also make the maximizer(s) of the variational problem that defines the free energy explicit (see Lemma~\ref{lem:psi}, there are cases where the maximizer is not unique!). We then give some consequences on the behavior of path trajectories.
Finally, we state our sharper results in Cram\'er's region (see Theorem~\ref{th:Gaussian}) and we also give some properties of the phase transition.

In Section~\ref{sec:comments}, we discuss several natural questions one is led to consider.
The rest of the paper is devoted to the proofs of the different results.

\section{A review of the (standard) wetting model}
\label{sec:wetting}

\subsection{Wetting of a random walk on a hard wall}

Let $(X_i)_{i \ge 1}$ be i.i.d.\ real valued random variables and let $(S_n)_{n\geq 0}$ be the associated random walk, that is $S_n = \sum_{i=1}^{n} X_i$ for $n\geq 0$ (with $S_0=0$ by convention); we denote by $\PP$ its law.
For $N\in \mathbb N$, the directed random walk $(n,S_n)_{0\leq n\leq N}$ is used to describe an effective interface between two solvents or alternatively the trajectory of a polymer of length $N$ (with $(n,S_n)$ representing the position of the $n$-th monomer), see Figure~\ref{fig:wetting}.

We focus on two specific cases: either the random walk is $\mathbb Z$-valued and aperiodic (discrete case) or the $X_i$'s are have a density with respect to the Lebesgue measure (continuous case).
We denote by~$f(\cdot)$ the density of the law of $X_i$ with respect to $\mu$, where $\mu$ is the counting measure on $\mathbb Z$ in the discrete case and $\mu$ is the Lebesgue measure in the continuous case: $\PP(X_1\in \dd x) = f(x) \mu(\dd x)$.

\smallskip

For $\lambda>0$, we introduce the following family of Gibbs measures $\PP_{N,\lambda}$ on the space of trajectories of $(S_0,\ldots, S_N)$, known as the ($\delta$-pinning) wetting model:
\begin{equation}
\label{eq:PNB}
\dd \PP_{N,\lambda} (s_1, \ldots, s_n) = \frac{1}{Z_{N,\lambda}}  \prod_{i=1}^{N-1} \big( \ind_{\{s_i>0\}} \mu(\dd s_i) + \lambda \delta_0(\dd s_i) \big) 
 \prod_{i=1}^N f(s_i-s_{i-1}) \,\lambda  \delta_0(\dd s_N)
 \,,
\end{equation}
with $\delta_0$ denoting the Dirac mass at $0$ and $s_0=0$ by convention. The quantity $Z_{N,\lambda}$ i the partition function of the model, that normalizes $\PP_{N,\lambda}$ to a probability. 
Let us stress that the measure $\PP_{N,\lambda}$ puts weights only on trajectories that stay above the wall, and gives a weight $\lambda$ to each contact with the wall.
We refer to Figure~\ref{fig:wetting} for an illustration.

\begin{figure}[ht]
    \centering
    \includegraphics[scale=1]{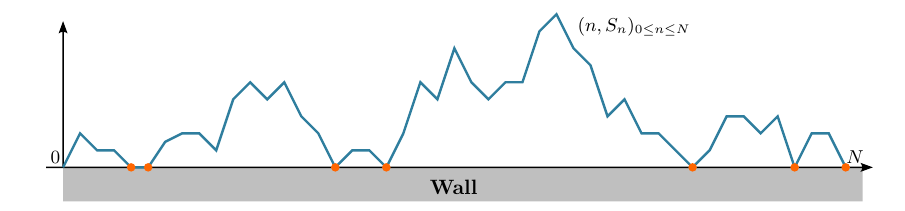}
    \caption{Illustration of the wetting of a random walk on a hard wall --- the random walk trajectory $(n,S_n)_{0\leq n\leq N}$ may represent a directed polymer or some effective interface. The trajectory is subject to a hard wall constraint (\textit{i.e.}\ one cannot have $S_i<0$) and the measure $\PP_{N,\lambda}$ additionally gives weight~$\lambda$ to each contact with the hard wall (represented by dots in the figure).}
    \label{fig:wetting}
\end{figure}

We will see in Section~\ref{sec:renewal} below that the wetting model~\eqref{eq:PNB} reduces to a renewal pinning model, somehow forgetting the underlying random walk path and considering only the instants of return to $0$.
However, it is important in our setting to keep the random walk interpretation, in particular if we want to consider a model with elevated boundary conditions.

Let us observe that, because of the term $\delta_0(\dd s_n)$ in~\eqref{eq:PNB}, the endpoint of the walk is pinned to the wall: this is the \emph{constrained} version of the model.
It is also possible to consider the \emph{free} version of the model, removing the last constraint $\delta_0(\dd s_n)$ in~\eqref{eq:PNB}; results would then be very similar and we refer to~\cite[Ch.~1\&2]{Giac07} for more discussion on the free case (that we do not discuss further).

\begin{remark}
In the discrete setting, which is the one considered originally in~\cite{Fish84} (and in most of the literature), the Gibbs measure $\PP_{N,\lambda}$ can be written as
\begin{equation}
\label{eq:PNB2}
\dd \PP_{N,\lambda}(S) = \frac{1}{Z_{N,\lambda}} \lambda^{H_N(S)}  \un_{\Omega_N^+(S)} \un_{\{S_N = 0\}} \dd \PP(s_1, \ldots, s_N) \,, \quad \text{ with } H_N(S) := \sum_{n=1}^N \un_{\{S_n = 0\}} \,,
\end{equation}
where we have set $\Omega_N^+(S) := \{S_1\geq 0, \ldots, S_N\geq 0\}$; this corresponds to the Gibbs measure considered in~\cite{LT15}.
The assumption that $(S_n)_{n\geq 0}$ is aperiodic is needed here so that the constraint $\{S_N=0\}$ can be verified, at least for $N$ large; if the random walk were periodic of period $d$, then one would simply need to restrict to lengths $N \in d\mathbb N$.
\end{remark}

Our setting is similar to that of~\cite{DGZ05,CGZ06}, where the density $f(\cdot)$ of the $X_i$'s is put in the form $f(x)=e^{-V(x)}$ for some potential $V$; the authors assume there that the increments $X_i$ are centered and in the domain of attraction of the normal law.
Our main assumption is the following (it is similar for instance to \cite[Hyp.~2.1]{CC13}).

\begin{assumption}
\label{hyp:1}
The random walk $(S_n)_{n\geq 0}$ is in the domain of attraction of a strictly stable law with index $\alpha\in (0,2]$ and positivity parameter $\varrho \in (0,1)$.
More precisely, there exists a sequence $(a_n)_{n\geq 1}$ which is regularly varying with index $1/\alpha$ such that $a_n^{-1} S_n$ converges weakly to some (strictly) $\alpha$-stable random variable $\mathcal{Z}$ with positivity parameter $\varrho \in (0,1)$.

In the discrete case, we assume that $(S_n)_{n\geq 0}$ is aperiodic; in the continuous case, we assume that the density $f_n$ of $S_n$ verifies $f_n(0)<+\infty$ for all $n\geq 1$ and is essentially bounded for some~$n$.
\end{assumption}

\noindent
We denote by $f_{\alpha}$ the density of the limiting $\alpha$-stable random variable $\mathcal{Z}$. We stress that $f_{\alpha}(0)>0$ thanks to the fact that $\varrho\in (0,1)$.

\subsection{Renewal representation of the model}
\label{sec:renewal}

As noticed in~\cite{CGZ06}, the wetting model~\eqref{eq:PNB} can be rewritten using a renewal process. 
Let us define, for $n\geq 1$,
\begin{equation}
\label{def:kappan}
 f_n^+(0) := \int_{s_1>0, \ldots, s_{n-1}>0}  \prod_{i=1}^{n} f(s_i-s_{i-1}) \prod_{i=1}^{n-1} \mu(\dd s_i) \,,\quad \text{ with }\  s_0:=s_{n} :=0 \,.
\end{equation}
This corresponds to the density $f_n^+(x)=\frac{1}{\mu(\dd x)}\PP( S_1 >0, \ldots, S_{n-1} >0, S_n \in \dd x)$ at $x=0$.
In particular, in the discrete case, we have $f_n^+(0)=\PP(S_1>0, \ldots, S_{n-1}>0, S_n=0)$.

Let us denote $\kappa:= \sum_{n=1}^{\infty}f_n^+(0)$ and observe that $\kappa \in (0,+\infty)$ holds under Assumption~\ref{hyp:1}, thanks to Lemma~\ref{lem:kappa} below (and formula~\eqref{rel:Kconti} to see that $f_n^+(0)<+\infty$ for all $n$ in the continuous case). Then we define the following probability density on $\mathbb{N}$:
\begin{equation}
\label{def:K1}
K(n) := \kappa^{-1} f_n^+(0) \qquad\text{ for all } \ n\geq 1 \,.
\end{equation}
Now, if we write explicitly the partition function
\[
Z_{N,\lambda} = \int_{(\mathbb{R}_+)^N} \prod_{i=1}^{N-1} \big( \ind_{\{s_i>0\}} \mu(\dd s_i) + \lambda \delta_0(\dd s_i) \big) 
 \prod_{i=1}^N f(s_i-s_{i-1}) \,\lambda  \delta_0(\dd s_N) \,,
\]
then decomposing it with respect to the number and positions of contact with the wall, \textit{i.e.}\ decomposing the integral over the possible sets $I = \{ i\in\{1,\ldots, N\}, s_i=0 \}$, we obtain
\[
Z_{N,\lambda} =  \sum_{k=1}^N \sum_{0=:i_0 < i_1 <\cdots < i_k =N} \lambda^k \prod_{j=1}^k f^+_{i_j-i_{j-1}}(0)  = \sum_{k=1}^N \sum_{0=:i_0 < i_1 <\cdots < i_k =N} (\kappa\lambda)^k \prod_{j=1}^k K(i_j-i_{j-1})  \,.
\]
Introducing a (recurrent) renewal process $\tau =(\tau_i)_{i\geq 0}$ with inter-arrival distribution $K(\cdot)$, \textit{i.e.}\ letting $\tau_0:=0$ and $(\tau_i-\tau_{i-1})_{i\geq 1}$ be i.i.d.\ $\mathbb{N}$-valued random variables with law $\bP(\tau_1=n)=K(n)$, the above partition function can be rewritten as
\[
Z_{N,\lambda} = \bE\Big[ (\kappa\lambda)^{H_N(\tau)} \ind_{\{N\in \tau\}}\Big] \,,\qquad \text{with } H_N(\tau) := \sum_{n=1}^N \ind_{\{n\in \tau\}} \,.
\]
(With some abuse of notation, we also interpret $\tau = \{\tau_i\}_{i\geq 0}$ as a subset of $\mathbb N \cup\{0\}$.)
In this context we can rewrite the Gibbs measure~\eqref{eq:PNB} as
\begin{equation}
\label{def:PNB2}
\frac{\dd \bP_{N,\lambda}}{\dd \bP} (\tau) = \frac{1}{Z_{N,\lambda}} (\kappa\lambda)^{H_N(\tau)} \ind_{\{N\in \tau\}} \,,
\qquad \text{with } H_N(\tau) = \sum_{n=1}^N \ind_{\{n\in \tau\}} \,.
\end{equation}
This is the usual formulation of the homogeneous pinning model, see~\cite[Ch.~2]{Giac07} (write $e^{\beta} =\kappa\lambda$).

Notice here that the wall constraint has been absorbed in the definition of $\tau$ (see the definition~\eqref{def:kappan} of $f_n^+(0)$) and we have reduced to a recurrent renewal $\tau$ at the cost of a change of parameter $\lambda \rightsquigarrow \kappa\lambda$.
This way, one can interpret $\tau$ as the return times of the random walk $(S_n)_{0\leq n\leq N}$ to $0$ conditioned on staying non-negative (due to the definition of $f_n^+(0)$) and having finite excursions away from $0$ (due to the normalization by $\frac1\kappa$ in~\eqref{def:K1}).
Let us also stress that the Gibbs law~\eqref{def:PNB2} describes only the distribution of the renewal process~$\tau$, but this is enough to describe completely the measure~\eqref{eq:PNB} on random walk trajectories: indeed, conditionally on the return times to $0$, the law of the excursions of $(S_n)_{n\geq 0}$ between $\tau_{k-1}$ and $\tau_k$ is unchanged by~\eqref{eq:PNB}.

Analogously to what is observed in~\cite[App.~A]{CGZ06}, under Assumption~\ref{hyp:1} we are able to obtain information on the law of the renewal process $\tau$.
In particular, we have the following results. Lemma~\ref{lem:ladder} is a direct consequence of~\cite{AD99}. Lemma~\ref{lem:kappa} is given in \cite[Prop.~4.1-(4.5)]{CC13} for the discrete case and \cite[Thm.~5.1-(5.2)]{CC13} for the continuous case. For the sake of completeness, we provide (simple) proofs of these results in Appendix~\ref{app:renewal}.

\begin{lemma}
\label{lem:ladder}
Let $\bar T_1 := \min\{ n\geq 1 , S_n \leq 0\}$ and $\bar H_1 = -S_{\bar T_1}$ be the first (weak) ladder epoch and ladder height of the random walk. Then, we have
\[
\kappa:=\sum_{n=1}^{\infty} f_n^+(0) = 
\begin{cases}
\bP(\bar H_1=0) & \ \ \text{ in the discrete case}, \\
 f_{\bar H_1} (0) & \ \ \text{ in the continuous case} ,
\end{cases}
\]
where $f_{\bar H_1}$ is the density of $\bar H_1$ w.r.t.\ the Lebesgue measure in the continuous case.
\end{lemma}

\begin{lemma}
\label{lem:kappa}
Suppose that Assumption~\ref{hyp:1} holds and let $f^+_n(0)$ be defined as in~\eqref{def:kappan}.
Then, 
\[
f_n^+(0) \sim  \frac{c_0}{n a_n} \qquad \text{ as }\ n\to\infty \,,
\] 
with $c_0 := f_{\alpha}(0) \bP(\bar H_1 >0)$, where we recall that $f_{\alpha}$ is the density of the limiting $\alpha$-stable random variable $\mathcal Z$;
observe also that $\bP(\bar H_1 >0)=1$ in the continuous case.

\noindent
Note that if $\sigma^2:=\bE[X_1^2]<+\infty$ ($\bE[X_1]=0$), then $f_n^+(0)\sim c_1 n^{-3/2}$ with $c_1 := \frac{1}{\sigma\sqrt{2\pi}} \bP(\bar H_1 >0)$.
\end{lemma}

\noindent

Let us now highlight one consequence of Lemma~\ref{lem:kappa} on the inter-arrival probability distribution of $\tau$.
Since the normalizing sequence $a_n$ in Assumption~\ref{hyp:1} is regularly varying with index $1/\alpha$, we find that there exists a  function $L(\cdot)$ slowly varying at infinity\footnote{A function $L(\cdot)$ is said to be slowly varying at infinity if for any $a>0$, $\lim_{x\to\infty} L(ax)/L(x) =1$.} such that
\begin{equation}
\label{def:K}
K(n) = \bP(\tau_1=n) = L(n) n^{-(1+\frac{1}{\alpha})}\,.
\end{equation}

To summarize, we have rewritten the Gibbs measure~\eqref{eq:PNB} in terms of the standard homogeneous pinning model, see~\eqref{def:PNB2}, with underlying (recurrent) renewal $\tau$ whose inter-arrival distribution~$K(\cdot)$.
The relation~\eqref{def:K} is often the underlying assumption when studying the pinning model (with a wider range for the parameter $\frac{1}{\alpha}$, which is here restricted to $[\frac12,+\infty)$).
We refer to~\cite[Ch.~2]{Giac07} for a complete overview of the (homogeneous) pinning model, but we collect (and complement) below some of the results.

\subsection{Free energy and phase transition}

An important physical quantity of the wetting (or pinning) model is the \textit{free energy}, defined by
\[
\f(\lambda) := \lim_{N\to\infty} \frac1N \log Z_{N,\lambda} \,.
\]
The fact that the limit exist follows easily once one realizes that the sequence $(\log Z_{N,\lambda})_{N\ge 1}$ is super-additive.
It is standard to see that the free energy $\f(\lambda)$ verifies:

(i) $\f(\lambda)\geq 0$, since we have $Z_{N,\lambda} \geq \kappa\lambda \bP(\tau_1=N)$ and~\eqref{def:K}, 
and $\f(\lambda) =0$ for all $\kappa\lambda \leq 1$;

(ii) $\lambda \mapsto\f(\lambda)$ is non-decreasing, since $\lambda \mapsto \log Z_{N,\lambda}$ is non-decreasing for any $N\geq 1$.

(iii) $\beta \mapsto \f(e^{\beta})$ is convex, since $\beta \mapsto \log Z_{N,e^{\beta}}$ is convex for any $N\geq 1$.

\noindent
Hence, we can define a critical point
\[
\lambda_c  = \sup \{ \lambda, \f(\lambda) =0\}= \inf\{\lambda, \f(\lambda) >0 \} \,.
\]
Let us stress that from convexity arguments, we obtain that whenever $\f'(\lambda)$ exists (from Theorem~\ref{prop:pinning} below, this is the case for all $\lambda >0$ except possibly at $\lambda=\lambda_c$), we have
\[
\f'(\lambda) = \lim_{N\to\infty} \frac{\partial}{\partial \lambda}
 \frac1N \log Z_{N,\lambda} = \lambda^{-1} \lim_{N\to\infty}  \bE_{N,\lambda} \Big[ \frac1N \sum_{i=1}^N \ind_{\{n\in \tau\}}\Big] \,.
\]
Hence, $\f'(\lambda)$ is ($\lambda^{-1}$ times) the limiting density of contacts under the Gibbs measure~\eqref{def:PNB2}.
This shows that $\lambda_c$ marks the transition between a \emph{delocalized} phase for $\lambda <\lambda_c$ (with $\f'(\lambda)=0$, zero density of contact) and a \emph{localized} phase (with $\f'(\lambda)>0$, positive density of contacts)

\smallskip
We now collect a number of results on the free energy: they show that the wetting model (or the pinning model when considering the definition~\eqref{def:PNB2}) is solvable.
We state the results with our notation, but we refer to~\cite[Ch.~2]{Giac07} for the general context of the pinning model.
Define the Laplace transform of~$\tau_1$ (recall that $\sum_{n=1}^{\infty}K(n)=1$): for any $\vartheta\geq 0$,
\begin{equation}
\label{def:LaplaceK}
\mathcal{K}(\vartheta) := \sum_{n=1}^{\infty} K(n) e^{- \vartheta n} = \bE\big[ e^{-\vartheta \tau_1} \big] \,.
\end{equation}
Let us note that $\mathcal{K}:(0,+\infty) \to (0,1)$ is decreasing and analytic.

\begin{theorem}[\cite{Giac07}, Thm.~2.1]
\label{prop:pinning}
The free energy is characterized by the following relation: 
\begin{equation}
\label{def:freeenergy}
\f(\lambda) \textit{ is the solution of }  \mathcal{K}(\vartheta) = (\kappa \lambda)^{-1} \text{ if a solution exists, and } \f(\lambda) =0 \text{ otherwise.}
\end{equation}
In particular, the critical point is given by $\lambda_c = 1/\kappa$, and the implicit function theorem shows that $\lambda\mapsto\f(\lambda)$ is analytic on $(\lambda_c,+\infty)$.
Additionally, we have
\begin{equation}
\label{eq:criticF}
\f(\lambda_c+u) \sim \tilde L(u) u^{ \min(1,\alpha)} \qquad \text{ as } u\downarrow 0\,,
\end{equation}
where $\tilde L$ is some slowly varying function (which depends explicitly on $L,\alpha$ in~\eqref{def:K}).
\end{theorem}

\smallskip
In the context of the wetting model, where $K(n) := \frac1\kappa f_n^+(0)$ with some explicit expressions for~$\kappa$ and $f_n^+(0)$ from Lemma~\ref{lem:ladder}-\ref{lem:kappa}, one is able to describe the critical behavior of $\f$ more explicitly than in~\eqref{eq:criticF}.
The proof is identical to that of~\cite[Thm.~2.1]{Giac07}, making some inversion formulas explicit;  we refer to Appendix~\ref{app:criticF} for details.

\begin{proposition}
\label{prop:criticF}
Suppose Assumption~\ref{hyp:1} holds and choose the normalization $(a_n)_{n\geq 1}$ as follows:  

--- if $\alpha \in (0,2)$, let $(a_n)_{n\geq 1}$ be such that $\bP(|X_1|>a_n)\sim \frac1n$ as $n\to\infty$;

--- if $\alpha=2$, let $(a_n)_{n\geq 1}$ be such that $\sigma^2(a_n) a_n^{-2} \sim \frac1n$, where $\sigma^2(x) := \bE[ (X_1)^2 \ind_{\{|X_1|\leq x\}}]$.

\noindent
Then, we have the following asymptotic behaviors, as $u\downarrow 0$:
\begin{itemize}
\item If $\sum_{n=1}^{\infty} \frac{1}{a_n} <+\infty$, \textit{i.e.}\ if the random walk $(S_n)_{n\geq 0}$ is transient (in particular if $\alpha<1$), then 
\begin{equation}
\label{eq:criticF2}
\f(\lambda_c+u) \sim c_1 u \qquad \text{ with } c_1 := \kappa^2 \Big/ \sum_{n=1}^{\infty} n f_n^+(0) \,.
\end{equation}

\item If $\alpha=1$ and $\sum_{n=1}^{\infty} \frac{1}{a_n} =+\infty$, letting $c_0:=f_{\alpha}(0) \bP(\bar H_1 >0)$ as in Lemma~\ref{lem:kappa}, we have
\begin{equation}
\label{eq:criticF3}
\f(\lambda_c+u) \sim \frac{\kappa^2}{c_0 } \, \frac{u}{\mu(1/u)}  \qquad \text{ with } \mu(x) := \int_0^x \frac{\dd s}{s^2 \bP(|X_1|>s)}  \,.
\end{equation}

\item If $\alpha \in (1,2)$, then 
\begin{equation}
\label{eq:criticF4}
\f(\lambda_c+u) \sim \bP\big(|X_1|>c_2/u \big) \qquad  \text{ with } c_2 :=  \alpha c_0 \Gamma (\tfrac{\alpha-1}{\alpha}) /\kappa^2 \,.
\end{equation}

\item If $\alpha =2$, then 
\begin{equation}
\label{eq:criticF5}
\f(\lambda_c+u) \sim c_3  u^2 \sigma^{2}(1/u) \qquad  \text{ with }   c_3:= \tfrac12 \kappa^4 /\bP(\bar H_1>0)^{2}\,.
\end{equation}
\end{itemize}
\end{proposition}


\subsection{Sharp asymptotic of the partition function and some path properties}

Let us stress that the sharp asymptotic of the partition function are known, in the delocalized ($\lambda<\lambda_c$), critical ($\lambda=\lambda_c$) and localized ($\lambda>\lambda_c$) regimes, as can be found in Theorem~2.2 of~\cite{Giac07}.
All these asymptotic behaviors are derived from the following representation:
\begin{equation}
\label{eq:identityZNlambda}
Z_{N,\lambda} = e^{N\f(\lambda)} \tilde \bP_{\lambda} \big(N\in \tilde \tau \big) \,,
\end{equation}
where $\tilde \bP_{\lambda}$ is the law of a  renewal $\tilde \tau$ with inter-arrival distribution given by 
\begin{equation}
\label{def:tildeP}
\tilde \bP_{\lambda}( \tilde \tau_1= n) =  \kappa\lambda  K(n) e^{-\f(\lambda) n} \qquad \text{ for } n\geq 1 \,.
\end{equation}
Indeed, writing $Z_{N,\lambda}(A) := \bE\big[ (\kappa \lambda )^{H_N(\tau)} \ind_{\{N\in \tau\}} \ind_A \big]$ for any event~$A$, we have for $0=:i_0<i_1<\cdots <i_k =N$,
\[
\begin{split}
Z_{N,\lambda}\big( \tau\cap[0,N] = \{i_1,\ldots, i_k\} \big) 
& = (\kappa \lambda)^k \prod_{j=1}^k
K(i_j-i_{j-1})  = e^{N\f(\lambda)}\prod_{j=1}^k \tilde \bP_{\lambda}(\tilde \tau_1=i_j-i_{j-1}) \\
&= e^{\f(\lambda)N} \tilde \bP_{\lambda}\big( \tau\cap[0,N] = \{i_1,\ldots, i_k\} \big) \,.
\end{split}
\]
Summing over $k\geq 1$ and $i_1,\ldots, i_{k-1}$, we obtain~\eqref{eq:identityZNlambda}.
Let us stress that in general, since we have $\bP_{N,\lambda}(A) = \frac{1}{Z_{N,\lambda}}Z_{N,\lambda}(A)$, we obtain the following representation for the pinning (or wetting) measure $\bP_{N,\lambda}$ (see also \cite[Rem.~2.8]{Giac07}): 
\begin{equation}
\label{pinningrepresentation}
\bP_{N,\lambda}\big( \tau\cap [0,N] \in \cdot \big)
= \tilde \bP_{\lambda}\big( \tilde \tau\cap [0,N] \in \cdot \, \big| \, N\in \tilde \tau\big) \,.
\end{equation}
Note that, in view of~\eqref{def:freeenergy}, $\tilde \tau$ defined by~\eqref{def:tildeP} is recurrent if $\lambda \geq \lambda_c$ and transient if $\lambda<\lambda_c$.
In the super-critical case $\lambda>\lambda_c$, we have that $\tilde \bP_{\lambda}( \tilde \tau_1= j)$ decays exponentially fast (with exponential decay rate $\f(\lambda)>0$), so $\tilde \tau_1$ is positive recurrent.
Therefore, using the renewal theorem, we get from~\eqref{eq:identityZNlambda} that
\begin{equation}
\label{eq:asympZ}
Z_{N,\lambda} \sim \frac{1}{m_{\lambda}}\, e^{\f(\lambda)N} \quad \text{ as } N\to\infty \,,
\quad \text{ with } m_{\lambda} := \tilde \bE_{\lambda}[\tilde \tau_1] = \kappa\lambda \sum_{n=1}^{\infty} n K(n) e^{-n \f(\lambda)}\,.
\end{equation}
We refer to~\cite[Thm.~2.2]{Giac07} for the corresponding results in the critical ($\lambda=\lambda_c$) and sub-critical ($\lambda<\lambda_c$) cases: we have
\[
\begin{split}
\text{ if $\lambda<\lambda_c$} \qquad &
Z_{N,\lambda}  \sim \frac{\kappa \lambda}{(1-\kappa\lambda)^2} K(N) \,, \\
\text{ if $\lambda=\lambda_c$} \qquad & Z_{N,\lambda}  \sim \begin{cases}
\bE[ \min(\tau_1,N)]^{-1} & \text{ if } \alpha \in (0,1] \,,  \\
\frac{\alpha}{\pi} \sin(\pi \alpha) N^{\alpha-1}/L(N)  & \text{ if } \alpha \in (1,2] \,.
\end{cases}
\end{split}
\]

To complement the above results, let us also give an application of~\eqref{pinningrepresentation} in the super-critical case.
Indeed, it gives that for any $k\in \{0,\ldots, N\}$, 
\begin{equation}
\label{pinningcontact}
\bP_{N,\lambda}\big(H_N (\tau)=k \big) =  \tilde \bP_{\lambda}\big( \big|\tilde \tau\cap [0,n] \big|  =k \, \big| \, n\in \tilde \tau\big) = \frac{1}{\tilde \bP_{\lambda}(n\in \tau)} \tilde \bP_{\lambda}\big(\tilde \tau_k= n \big) \,.
\end{equation}
Indeed, using the local central limit theorem for $\tilde \tau_k$, see e.g.~\cite[Ch.~9, \S50]{GK54}, and the renewal theorem, one directly obtains the following result.
\begin{proposition}
\label{prop:Hlocal}
For any $\lambda>\lambda_c$, we have the following local limit theorem for $H_N=H_N(\tau)$ under $\bP_{N,\lambda}$:
\[
\sup_{k\geq 0} \Big| \sqrt{N}\, \frac{1}{m_{\lambda}}\bP_{N,\lambda}\big(H_N=k \big) - \phi_{\sigma_{\lambda}}\Big( \frac{N-k m_{\lambda}}{\sqrt{N}}\Big) \Big| \xrightarrow{N\to\infty} 0\,.
\]
where $\phi_{\sigma}(t) = \frac{1}{\sqrt{2\pi} \sigma} e^{-\frac{t^2}{2\sigma^2}} $ is the density of $\mathcal{N}(0,\sigma^2)$, and $m_{\lambda}:= \tilde \bE_{\lambda}[\tilde \tau_1]$, $\sigma_{\lambda} = \tilde{\mathbf{V}}\mathbf{ar}_{\lambda}(\tilde \tau_1)$.
\end{proposition}

As a direct consequence of a Riemann sum approximation, we obtain that under $\bP_{N,\lambda}$, the rescaled number of contacts $\frac{1}{\sqrt{N}} (H_N- \frac{1}{m_{\lambda}} N)$ converges in distribution to $\mathcal{N}(0,\sigma_\lambda^2/m_{\lambda}^2)$.


\section{Wetting with elevated boundary conditions}
\label{sec:wettingwell}

\subsection{Wetting in a square well}

We now turn to the wetting model with elevated boundary condition, also seen as a model of wetting of a square well.
We will focus here on the discrete case, so from now on $(S_n)_{n\geq 0}$ is an aperiodic centered random walk on $\mathbb Z$; the continuous case can be treated analogously.
The model consists in lowering the hard wall to a depth~$-\lfloor aN\rfloor$, where $a>0$ is a parameter tuning the depth of the well, see Figure~\ref{fig:ex_puits}.
\begin{equation}
\label{eq:PNwell}
\frac{\dd \PP_{N,\lambda}^a}{\dd\PP} (S) = \frac{1}{Z_{N,\lambda}^a} \lambda^{H_N(S)}  \un_{\Omega_{N,a}^+(S)} \un_{\{S_N = 0 \}} \,, \quad \text{ with }\  H_N^a(S) := \sum_{n=1}^N \un_{\{ S_n = - \lfloor aN \rfloor \}} \,,
\end{equation}
with $Z_{N,\lambda}^a$ the partition function of the model and $\Omega_{N,a}^+(S) = \{S_1 \ge -\ent{aN}, ..., S_N \ge -\ent{aN}\}$.

The definition~\eqref{eq:PNwell} is equivalent to considering the wetting model~\eqref{eq:PNB2} with elevated boundary conditions $S_0 =S_N = \lfloor aN \rfloor$.
One could also consider non-symmetric boundary conditions $S_0 = \lfloor aN \rfloor$, $S_N = \lfloor bN \rfloor$ with $a,b>0$, as done for instance in~\cite{Pat98,PV99,BFO09}. We chose for simplicity to focus on the symmetric case $a=b$; we comment below how some of our results would be modified.

\begin{figure}[t]
    \centering
    \includegraphics[scale=0.9]{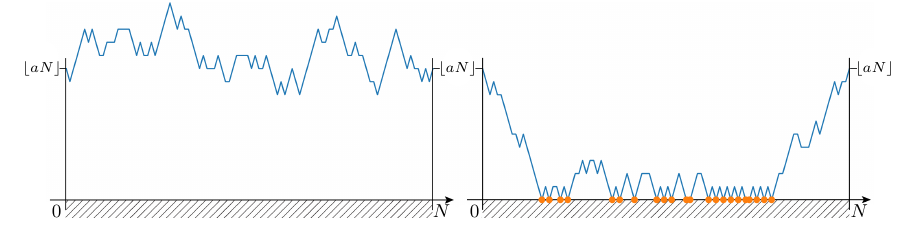}
    \caption{Wetting of a lazy random walk in a rectangular well of depth $\ent{aN}$.
On the left, the system is in the delocalized regime; on the right it is in the localized regime.}
    \label{fig:ex_puits}
\end{figure}

\subsection{First results: free energy, left-most and right-most points of contact}
\label{sec:results}

Here again, one may define the free energy of the wetting model with elevated boundary condition:
\begin{equation}
\label{def:freeenergy2}
\f(\lambda, a) = \lim_{N\to\infty} \frac1N \log Z_{N,\lambda}^a \,.
\end{equation}
For simplicity, we keep the notation $\f(\lambda) = \f(\lambda,0)$ for the free energy of the original wetting model.
Theorem~\ref{th:freeenergy} below shows that the free energy exists and makes it explicit in terms of $\f(\lambda)$ and of the left and right large deviation rate functions for $S_n$.

\subsubsection{Free energy and large deviations}
First, let us define the (left and right) log-moment generating functions of~$X_1$: for $t\geq 0$, let
\begin{equation}
\label{def:Lambda}
\Lambda_+(t) := \log \bE[e^{tX_1}] \,,  
\qquad \Lambda_-(t) := \log \bE[e^{-tX_1}] \,.
\end{equation}
Let us also define the Fenchel--Legendre transforms of $\Lambda_{\pm}$: for $x\geq 0$,
\begin{equation}
\label{def:rateI}
\I_+(x) := \sup_{t\geq 0} \{ tx - \Lambda_+(t)\}\,,
\qquad \I_-(x) := \sup_{t\geq 0} \{ tx - \Lambda_-(t)\}\,.
\end{equation}
Let also $\Lambda(t) := \Lambda_+(t) \ind_{\{t\geq 0\}} + \Lambda_-(-t)\ind_{\{t<0\}}$ and $\I(x) := \I_+(x) \ind_{\{x\geq 0\}} + \I_-(-x)\ind_{\{x<0\}}$.
By Cram\'er's theorem, $\I_+,\I_-$ are the (upward and downward) large deviation rate functions for the random walk $(S_n)_{n\geq 0}$, see e.g.~\cite[Thm.~2.2.3]{DZ09} or Section~\ref{sec:LD} below,~\eqref{eq:CramerLDP}.
Our first result is to obtain the value of the free energy in terms of $\mathrm{I}_+,\mathrm{I}_-$, in analogy with~\cite[Prop.~2.1]{LT15}.

\begin{theorem}
\label{th:freeenergy}
The free energy $\f(\lambda,a)$ defined in~\eqref{def:freeenergy2} exists and $\f(\lambda,a) = \max\{\psi(\lambda, a), 0\}$, where
\begin{equation}
\label{def:psi}
\begin{split}
\psi(\lambda, a) & :=  \sup_{0\leq u\leq v\leq 1} g_{\lambda,a}(u,v),\\
\text{with } \quad g_{\lambda,a}(u,v) & := (v-u) \f(\lambda) - u \,\I_-\Big(\frac{a}{u}\Big) - (1-v)\, \I_+\Big(\frac{a}{1-v}\Big)\,.
\end{split}
\end{equation}
\end{theorem}

\begin{remark}
For the wetting model~\eqref{eq:PNB2} with non-symmetric elevated boundary conditions $S_0 = \lfloor aN \rfloor$, $S_N = \lfloor bN \rfloor$, Theorem~\ref{th:freeenergy} is easily adapted and one obtains that the free energy is $\max\{\psi(\lambda, a,b), 0\}$ with $\psi(\lambda, a,b) :=  \sup_{0\leq u\leq v\leq 1} \big\{ (v-u) \f(\lambda) - u \,\I_-(\frac{a}{u}) - (1-v)\, \I_+(\frac{b}{1-v}) \big\}$.
\end{remark}

Then, for any $a\geq 0$, we can define the critical point 
\[
\lambda_c(a):= \inf\{\lambda, \f(\lambda,a)>0\} = \inf\{\lambda, \psi(\lambda,a)>0\} \,,
\]
and one easily sees that $a\mapsto \lambda_c(a)$ is non-decreasing.
The point $\lambda_c(a)$ marks a \emph{localization} or \emph{wetting} phase transition, from a delocalized phase $\f(\lambda,a)=0$ (with zero density of contact) to a localized phase $\f(\lambda,a)>0$ (positive density of contact).

Note also that if $\mathrm{I}_{+}=\mathrm{I}_-=0$, then $\lambda_c(a) =\lambda_c(0)=\lambda_c$ for all $a\geq 0$, which is in particular the case if Assumption~\ref{hyp:1} holds with $\alpha \in (0,2)$ and $\varrho \in (1-\frac{1}{\alpha},\frac{1}{\alpha})$ (this excludes the totally asymmetric case if $\alpha\in (1,2)$).

\subsubsection{About the free energy and the phase diagram}
\label{sec:transition}

Let us now make more explicit the expressions of the free energy $\f(\lambda,a)$. 
We also give an expression for the critical line separating the localized phase $\mathcal{L}= \{ (\lambda, a), \f(\lambda,a) >0\}$ from the delocalized phase $\mathcal{D}= \{ (\lambda, a), \f(\lambda,a) =0\}$:
\[
a_c(\lambda) = \inf\{a \geq 0 : \f(\lambda,a) = 0\}
\quad \text{and} \quad
\lambda_c(a) = \sup\{\lambda \geq  0 : \f(\lambda,a) = 0\} \,.
\]
We refer to Figure~\ref{fig:PhaseDiag} for an illustration of the phase diagram.

Before we state the results, let us introduce the radii of convergence for $\Lambda_+,\Lambda_-$:
\begin{equation}
\label{t0pm}
t_0^+ := \sup\{t\geq 0, \Lambda_+(t)<+\infty\}
\quad\text{and } \quad
t_0^- := \sup\{t\geq 0, \Lambda_-(t)<+\infty\} \,.
\end{equation}
In the case where $t_0^+=t_0^- =0$, then we trivially have that $\I_+=\I_-\equiv 0$ and $\f(\lambda,a) =\f(\lambda)$ for any $a>0$. 
Our next results therefore have some interest only if $t_0^+>0$ or $t_0^->0$.
We stress that~$\Lambda_{+}$ is increasing and analytic on $[0,t_0^{+})$ so it is in fact invertible with analytic inverse $\Lambda_+^{-1}$ on that interval. 
We extend this definition on $\mathbb{R}_+$ by letting $\Lambda_{+}^{-1}$ be the left-continuous inverse of $\Lambda_{+}$; in particular, $\Lambda_+^{-1}(x) = t_0^+$ for all $x\geq \Lambda_+(t_0^+)$ and $\Lambda_+^{-1}\equiv 0$ if $t_0^+=0$. Similar notation holds for $\Lambda_-$.

\begin{theorem}
 \label{th:formuleG}
For any $\lambda>0$, we have
\begin{equation} \label{eq:formuleG}
    \f(\lambda,a) = \max\{ \psi(\lambda,a),0 \} =
\max\Big\{  \f(\lambda) - a  \Lambda_+^{-1} (\f(\lambda)) -a \Lambda_-^{-1} (\f(\lambda))  , 0 \Big\} 
\end{equation}
In particular, we have the following formula for $a_c(\lambda)$: if $\lambda>\lambda_c$,
\begin{equation}
\label{eq:acritique}
a_c(\lambda) =\frac{\f(\lambda)}{\Lambda_+^{-1}(\textsc{f}(\lambda)) + \Lambda_-^{-1}(\textsc{f}(\lambda))} \,.
\end{equation} 
\end{theorem}

\begin{center}
\begin{figure}[t]
\includegraphics[scale=0.45]{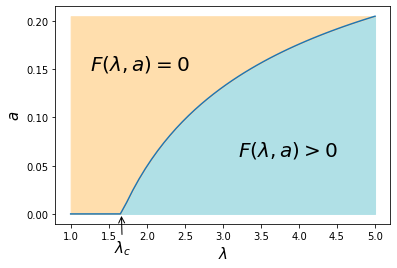}
\qquad
\includegraphics[scale=0.45]{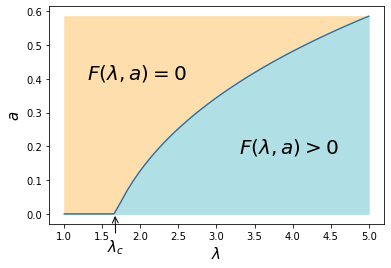}
\vspace{-0.5\baselineskip}
\caption{Phase diagram for the wetting in a square well for two integrable models (see Appendix~\ref{sec:examples}): the symmetric lazy random walk (on the left) and of the symmetric geometric random walk of Section~\ref{ex:geomRW} (on the right). We refer to Sections~\ref{ex:lazyRW} and~\ref{ex:geomRW} for explicit formulas derived from Theorem~\ref{th:formuleG} (with the parameter $\gamma=0.4$ in the figures above).}
\label{fig:PhaseDiag}
\end{figure}
\end{center}

\noindent
Let us now give some properties of $\f(\lambda,a)$.
\begin{itemize}
\item Similarly as for the free energy of the (standard) wetting model, the function $\lambda \mapsto \f(\lambda,a)$ is non-negative and non-decreasing and $\beta \mapsto \f(e^{\gb},a)$ is convex, for any $a\geq 0$.
\item The function $a\mapsto \f(\lambda,a)$ is affine by parts.
\end{itemize}

\noindent
Now, define $\lambda_+,\lambda_->0$ as follows:
\begin{equation}
\label{def:lambda+-}
\begin{cases}
\f(\lambda_+)= \Lambda_+(t_0^+) &\quad \text{ if } t_0^+ \in (0,+\infty)\,, \\
\f(\lambda_-)= \Lambda_-(t_0^-) &\quad \text{ if } t_0^- \in (0,+\infty)\,,
\end{cases}
\end{equation}
and $\lambda_{+}= +\infty$, resp.\ $\lambda_-=+\infty$, if $t_0^{+} =0$ or $t_0^+=+\infty$, resp.\ if $t_0^{-} =0$ or $t_0^-=+\infty$.
Naturally we also have $\lambda_{\pm}=+\infty$ if $\Lambda_{\pm}(t_0^{\pm})=+\infty$.
Then, the function $\lambda \mapsto \Lambda_+^{-1}(\f(\lambda))$ is analytic on $(0,\lambda_+)$ and constant (equal to $t_0^+$) on $[\lambda_+,+\infty)$; and similarly for $\lambda \mapsto \Lambda_-^{-1}(\f(\lambda))$. 
Let us therefore make two observations:
\begin{itemize}
\item If $\lambda > \lambda_+,\lambda_-$, then we have that $\f(\lambda)= \max\{\f(\lambda)-a(t_0^++t_0^-),0\}$ and $a_c(\lambda) = \frac{1}{t_0^++t_0^-} \f(\lambda)$.

\item The functions $\lambda \mapsto \f(\lambda,a)$ is analytic except at $\lambda_c(a)$ and at $\lambda_{+}$, $\lambda_{-}$; the function $\lambda \mapsto a_c(\lambda)$ is analytic on $(\lambda_c,+\infty)$, except at $\lambda_+,\lambda_-$.
\end{itemize}
The last observation shows that in addition to the \emph{localization} phase transition at $\lambda_c(a)$, there might be another phase transition (or two others), at $\lambda = \lambda_{+},\lambda_-$, provided that $\lambda_{+},\lambda_->\lambda_c(a)$.
The phase transition at $\lambda_{\pm}$ is a \emph{saturation} phase transition, in the sense that the left or right-most point of contact becomes degenerate when $\lambda>\lambda_-$ or $\lambda>\lambda_+$, see Section~\ref{sec:leftright} (and Figure~\ref{fig:lempsi}).
This second phase transition appears to be absent from other wetting models considered so far in the literature, and essentially comes from the fact that the rate function $\I_{\pm}$ itself may have a (saturation) phase transition at some $\rho_{\pm}$, see Section~\ref{sec:LD} below.


\begin{example}
\label{ex:almostgeom}
Consider a random walk with symmetric increments whose distribution is given by $\bP(X_1=x) = c_{\theta} (1+|x|)^{-\theta} e^{-|x|}$ for $x\in \mathbb{Z}$, where $\theta \in \mathbb{R}$ is a parameter and~$c_{\theta}$ is a constant that normalizes $\bP$ to a probability.
Then, we clearly have that $t_0^+=t_0^- =1$, and $\Lambda_{\pm}(1) =+\infty$ if $\theta \leq 1$, $\Lambda_{\pm}(1) <+\infty$ if $\theta >1$.
Therefore, the critical point $\lambda_{\pm}$ is finite if and only if $\theta>1$.
\end{example}

\subsubsection{Maximizers of the variational problem~\eqref{def:psi}}
\label{sec:maximizers}

As mentioned above, $\Lambda_+$ is increasing and analytic on $[0,t_0^+)$, so $\Lambda_{+}^{-1}$ and $\Lambda_+'$ are well defined on that interval: we extend their definition on~$\mathbb{R}_+$ by considering respectively the left-continuous inverse and the left-derivative of $\Lambda_{+}$. In the case $t_0^+=0$, we set $\Lambda_+'\circ \Lambda_+^{-1} \equiv +\infty$. Similar notation holds for $\Lambda_-$.
We also denote
\begin{equation}
\label{def:rho}
\rho_+ := \Lambda_+'(t_0^+)=\lim_{t\uparrow t_0^+} \Lambda'_+(t),
\qquad 
\rho_- :=\Lambda_-'(t_0^-)= \lim_{t\uparrow t_0^-} \Lambda'_-(t) \,,
\end{equation}
where by convention we set $\rho_{\pm}=+\infty$ if $t_0^{\pm}=0$.
Then, for any $\lambda\geq \lambda_c$ and $a\geq 0$, we define
\begin{equation}
\label{def:w*}
w_+^* := \frac{a}{\Lambda_+'\circ \Lambda_+^{-1} (\f(\lambda))} \in \Big[ \frac{a}{\rho_+},+\infty \Big)  \,,\qquad
w_-^* := \frac{a}{\Lambda_-'\circ \Lambda_-^{-1} (\f(\lambda))}  \in \Big[ \frac{a}{\rho_-},+\infty \Big)\,.
\end{equation}
Note that $w_+^*> \frac{a}{\rho_+}$ if $\lambda <\lambda_+$ and $w_+^*=\frac{a}{\rho_+}$ if $\lambda\geq \lambda_+$, and similarly for $w_-^*$.
We then have the following result. 

\begin{lemma}
\label{lem:psi}
If $w_+^*+w_-^* \leq 1$, and in particular if $\lambda\geq \lambda_c(a)$, we have
\[
\sup_{0\leq u\leq v\leq 1}\big\{g_{\lambda,a}(u,v) \big\} = \f(\lambda)- a \Lambda_-^{-1}(\f(\lambda)) -a  \Lambda_+^{-1}(\f(\lambda)) \,, 
\]
and the the supremum is attained on $U^*\times V^* = U^*_{\lambda,a} \times V^*_{\lambda,a}$, with 
\[
U^* = \begin{cases}
\{w_-^*\} & \text{ if } \lambda <\lambda_-\,, \\
[0,\frac{a}{\rho_-}] & \text{ if } \lambda =\lambda_- \,, \\
\{0\} & \text{ if } \lambda >\lambda_- \,,
\end{cases}
\qquad\qquad
V^* = \begin{cases}
\{1-w_+^*\} & \text{ if }\lambda <\lambda_+\,, \\
[1-\frac{a}{\rho_+},1] & \text{ if } \lambda=\lambda_+ \,, \\
\{1\} & \text{ if } \lambda >\lambda_+ \,.
\end{cases}
\]
\end{lemma}

\noindent
We show in Lemma~\ref{lem:w+w-} below that $w_+^*+w_-^*<1$ whenever $\lambda_c\geq \lambda_c(a)$; this proves in particular the first claim of the lemma.

\begin{remark}
As far as the wetting model~\eqref{eq:PNB2} with non-symmetric elevated boundary conditions $S_0 = \lfloor aN \rfloor$, $S_N = \lfloor bN \rfloor$ is concerned, one could also adapt Lemma~\ref{lem:psi} (and Theorem~\ref{th:leftrightpoints} below): one simply need to replace $a$ by $b$ in the definition of $w_+^*$.
\end{remark}

\begin{remark}
In the case where $\Lambda_+(t_0^{+}),\Lambda_-(t_0^-)=+\infty$, then $\lambda_+,\lambda_- =+\infty$ and the supremum in~\eqref{def:psi} is uniquely attained.
In fact, the only case when the supremum is not uniquely attained is if $\lambda=\lambda_-$ (and $\rho_-<+\infty$) or $\lambda =\lambda_+$ (and $\rho_+<+\infty$).
Considering Example~\ref{ex:almostgeom} as an illustration, we have $\lambda_{\pm} <+\infty$ if and only if $\theta>1$ and we additionally have $\rho_{\pm}<+\infty$ if and only if $\theta>2$.
\end{remark}

\subsubsection{Left and right-most pinned points}
\label{sec:leftright}
We can extract from Theorem~\ref{th:freeenergy} and Lemma~\ref{lem:psi} some properties of the path, in analogy with~\cite[Thm.~2.2]{LT15}.
Let us define the left-most and right-most points of contact between the walk and the bottom of the well:
\[
\begin{split}
L_N  = L_N^{a} & := \min\{ 0\leq n \leq N, \lfloor aN \rfloor +S_n =0  \}\,,\\
R_N  = R_N^{a} & := \max\{ 0\leq n \leq N, \lfloor aN \rfloor +S_n =0 \}\,.
\end{split}
\]
By convention, we set $L_N,R_N =+\infty$ if $H_N^a(S)=0$, \textit{i.e.}\ if there is no contact with the bottom of the well.
Note that we focus on the left and right-most point of contact, but physically, an important quantity to consider is the contact angles between the interface and the bottom of the well (or substrate); we refer to Remark~\ref{rem:angles} below for some comments.
We prove the following convergence of $(L_N,R_N)$; we let $\mathrm{dist}(z,A) := \inf_{z'\in A} \|z-z'\|$.
\begin{theorem}
\label{th:leftrightpoints}
In the super-critical case, that is for $\lambda >\lambda_c(a)$, we have that for any $\gep>0$,
\begin{equation}
\label{eq:leftright}
\lim_{n\to\infty}\bP_{N,\lambda}^a\bigg( \mathrm{dist}\Big( \tfrac1N ( L_N, R_N ), U^*\times V^* \Big) >\gep  \bigg) =0 \,,
\end{equation}
where $U^*\times V^*$ is defined in Lemma~\ref{lem:psi} (and are the maximizers of the variational problem~\eqref{def:psi}). 

\noindent
In the subcritical case, that is for $\lambda< \lambda_c(a)$, if we assume that $\lim_{n\to\infty} \frac{1}{n}a_n =+\infty$ (so the bottom of the well is much further than the typical fluctuations of the random walk), then we have 
\begin{equation}
\label{eq:nocontact}
\lim_{n\to\infty} \bP_{N,\lambda}^a( H_N^a(S) =0)  =1 \,.
\end{equation}
\end{theorem}

In the case where $\Lambda_+(t_0^+),\Lambda_-(t_0^-)=+\infty$, Lemma~\ref{lem:psi} shows that the supremum in~\eqref{def:psi} is uniquely attained, and from Theorem~\ref{th:leftrightpoints} we get that the left and right-most points of contact are located at $(u^*,v^*) =(w_-^*,1-w_+^*)$ with $\frac{a}{\rho_-}<u^*<v^*<1-\frac{1}{\rho_+}$, for any $a>0$, $\lambda>\lambda_c(a)$.
The most unexpected case occurs when $\Lambda_+(t_0^+)<+\infty$ and/or $\Lambda_-(t_0^-) <+\infty$.
Let us focus on the left-most point for simplicity. Recalling the critical value $\lambda_-$ defined in~\eqref{def:lambda+-}, Lemma~\ref{lem:psi}-Theorem~\ref{th:leftrightpoints} show that the left-most point is located at $u^*=w_-^* > \frac{a}{\rho_+}$ for $\lambda<\lambda_-$ and then drops (or \emph{saturates}) to $u^*=0$ for $\lambda>\lambda_-$, either continuously if $\rho_-=+\infty$ or discontinuously if $\rho_-<+\infty$; at $\lambda=\lambda_-$, the left-most point is located in an interval $[0,\frac{a}{\rho_-}]$.
We refer to Figure~\ref{fig:lempsi} for an illustration.

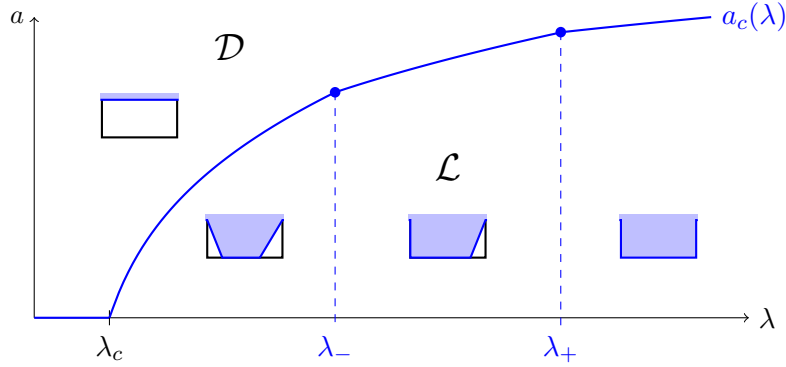
\begin{figure}
\begin{tikzpicture}
  \draw[->] (0, 0) -- (9.5, 0) node[right] {$\lambda$};
  \draw[->] (0, 0) -- (0, 4) node[left] {\small $a$};
  \draw[scale=1, domain=0:1, thick, smooth, variable=\x, blue] plot ({\x}, {0});
  \draw[-] (1,0.1) -- (1,-0.1) node[below] {$\lambda_c$};
  \draw[scale=1, domain=1:4, thick, smooth, variable=\x, blue] plot ({\x}, {3*(\x-1)/(2*sqrt(\x)-1)});
  \node[circle,fill=blue, inner sep=0pt,minimum size=4pt] (b) at (4,3) {};
  \draw[-,dashed, blue] (4,3) -- (4,-0.1) node[below] {$\lambda_-$};
  \draw[scale=1, domain=4:7, thick, smooth, variable=\x, blue] plot ({\x}, {1+2*(\x-1)/(2*sqrt(\x)-1)});
  \node[circle,fill=blue, inner sep=0pt,minimum size=4pt] (b) at (7,3.8) {};
  \draw[-,dashed, blue] (7,3.8) -- (7,-0.1) node[below] {$\lambda_+$};
  \draw[scale=1, domain=7:9, thick, smooth, variable=\x, blue] plot ({\x}, {2.4+(\x-1)/(2*sqrt(\x)-1)}) node[right] {$a_c(\lambda)$};
  \draw[] (2.6,3.6) node[] {\Large $\mathcal{D}$};
  \draw[] (5.5,2) node[] {\Large $\mathcal{L}$};
    \draw[-,thick] (0.9,2.9) -- (0.9,2.4) -- (1.9,2.4) -- (1.9,2.9);
  \fill[blue!25] (0.88,2.9) -- (0.88,2.98) -- (1.92,2.98) --  (1.92,2.9);
  \draw[-,thick,blue] (0.88,2.9) -- (1.92,2.9);
  \draw[-,thick] (2.3,1.3) -- (2.3,0.8) -- (3.3,0.8) -- (3.3,1.3);
  \fill[blue!25] (2.28,1.3) -- (2.28,1.38) -- (3.32,1.38) -- (3.32,1.3);
  \draw[-,thick,blue,fill=blue!25] (2.28,1.3) -- (2.3,1.3) -- (2.5,0.8) -- (3,0.8) -- (3.3,1.3) -- (3.32,1.3);
  \draw[-,thick] (5,1.3) -- (5,0.8) -- (6,0.8) -- (6,1.3);
  \fill[blue!25] (4.98,1.3) -- (4.98,1.38) -- (6.02,1.38) -- (6.02,1.3);
  \draw[-,thick,blue,fill=blue!25] (4.98,1.3) -- (5,1.3) -- (5,0.8) -- (5.8,0.8) -- (6,1.3) -- (6.02,1.3);
  \fill[blue!25] (7.78,1.3) -- (7.78,1.38) -- (8.82,1.38) -- (8.82,1.3);
  \draw[-,thick,blue,fill=blue!25] (7.78,1.3) -- (7.8,1.3) -- (7.8,0.8) -- (8.8,0.8) -- (8.8,1.3) -- (8.82,1.3);
\end{tikzpicture}
\caption{Illustration of the phase diagram in the case where $\lambda_c<\lambda_-<\lambda_+<+\infty$.
The critical curve $a_c(\lambda)$ separates a localized phase $\mathcal{L}=\{(\lambda,a), \f(\lambda,a)>0\}$ and a delocalized phase $\mathcal{D} = \{(\lambda,a), \f(\lambda,a)=0\}$.
We have also represented a typical configuration, depending on the range of parameters: according to Lemma~\ref{lem:psi}, the left-most point of contact is either at $w_-^*>\frac{a}{\rho_-}$ for $\lambda<\lambda_-$, in $[0,\frac{a}{\rho_-}]$ for $\lambda =\lambda_-$ or at $0$ for $\lambda>\lambda_-$ (and similarly for the right-most point of contact).
Also, in view of Theorem~\ref{th:formuleG}, the formula for $\f(\lambda,a)$ is different in the regions $(\lambda_c,\lambda_-)$, $(\lambda_-,\lambda_+)$ and $(\lambda_+,+\infty)$.
}
\label{fig:lempsi}
\end{figure}

Let us mention that, as far as the critical case $\lambda=\lambda_c(a)$ is concerned, one needs some extra assumption to obtain a result analogous to Theorem~\ref{th:leftrightpoints}. 
For instance, if $t_0^+=t_0^-=0$ we have $\lambda_c(a) =0$: then we can easily show that~\eqref{eq:nocontact} holds also at $\lambda=\lambda_c(a)$.
On the other hand, if one have $\lambda_c(a) < \min\{\lambda_+,\lambda_-\}$, \textit{i.e.}\ if one is inside the so-called Cram\'er's region, we prove in Theorem~\ref{th:Gaussian} below that~\eqref{eq:leftright} holds at $\lambda=\lambda_c(a)$, with a unique maximizer $(u^*,v^*)$ verifying $0<u^*<v^*<1$.

\begin{remark}
\label{rem:angles}
One could improve Theorem~\ref{th:leftrightpoints} to obtain in the super-critical case the convergence of the full trajectory $(\frac1N S_{\lfloor tN \rfloor})_{t\in [0,1]}$, as done for instance in~\cite{BFO09}.
If $\lambda>\lambda_c(a)$ and if the maximizers $U^*,V^*$ are reduced to one point $w_-^*,w_+^*$, one can easily show that the trajectory converges under~$\bP_{N,\lambda}^a$ to three line segments joining the points $(0,0)$, $(w_-^*,-a)$, $(w_+^*,-a)$, $(1,0)$; we refer to Figure~\ref{fig:lempsi} for an illustration.

Some interesting quantities considered for instance in~\cite{Pat98,PV99} are the contact angles $\theta_-,\theta_+$ between the bottom of the well and the first and last segment of the trajectory.
Since we have $\tan \theta_- = a/w_-^*$, $\tan \theta_+ = a/w_+^*$, then in view of the formula~\eqref{def:w*} for $w_+^*,w_-^*$, we observe \emph{the contact angles $\theta_-,\theta_+$ do not depend on $a$ but only on $\lambda$} (this remains true for non-symmetric elevated boundary conditions) --- note that in the case of a symmetric random walk one additionally has $\theta_+=\theta_-$.
This was already noticed in~\cite{Pat98,PV99} for the pinning of interfaces in the 2D Ising model, the angle verifying the so-called Herring--Young equation.
\end{remark}

\subsection{Cram\'er's region and Gaussian fluctuations}

We now wish to obtain more precise estimates on the partition function and on the left and right-most pinned point, sharpening in particular Theorem~\ref{th:leftrightpoints}.
For that, we are going to assume that Cram\'er's condition holds, which ensures in particular that the rate functions $\mathrm{I}_+,\mathrm{I}_-$ are non-trivial.

\begin{assumption}[Cram\'er's condition]
\label{hyp:Cramer}
There is some $t_0>0$ such that $\bE[e^{t X_1}] <+\infty$ for $|t|<t_0$.
In other words, $t_0^+>0$ \emph{and} $t_0^->0$.
\end{assumption}

Under that assumption, we call \emph{Cram\'er's region} the interval $(-\rho_-,\rho_+)$, where $\rho_-,\rho_+ \in(0,+\infty]$ are defined in~\eqref{def:rho}.
In particular, one can show that the rate function~$\I$ is strictly convex inside Cram\'er's region $(-\rho_-,\rho_+)$; we refer to Section~\ref{sec:LD} for more comments.
In the super-critical and critical case $\lambda \geq \lambda_c(a)$, Assumption~\ref{hyp:Cramer} allows us to obtain inside Cram\'er's region a sharp estimate on the partition function and on the distribution of $L_N,R_N$ under $\bP_{N,\lambda}^a$.

More precisely, we will assume that $\lambda_c(a)\leq \lambda < \min\{\lambda_+,\lambda_-\}$, which ensures: first that $\lambda \geq \lambda_c(a)$ thanks to Lemma~\ref{lem:w+w-} below; second that  $w_+^*,w_-^*$ defined in~\eqref{def:w*} verify $w_+^*> \frac{a}{\rho_+}$, $w_-^*>\frac{a}{\rho_-}$.
In that case, we set
\begin{equation}
\label{eq:formulauv*}
u^* = u_{\lambda,a}^* = w_-^* = \frac{a}{\Lambda_-'\circ \Lambda_-^{-1} (\f(\lambda))},
\qquad
v^*=v_{\lambda,a}^* =1-w_+^* =1- \frac{a}{\Lambda_+'\circ \Lambda_+^{-1} (\f(\lambda))} \,,
\end{equation}
which verify $\frac{a}{\rho_-}<u^*<v^*<1-\frac{a}{\rho_+}$; hence $\frac{a}{u^*}$ and $\frac{a}{1-v^*}$ are both inside Cram\'ers's region.

\begin{theorem}
\label{th:Gaussian}
Suppose that Assumption~\ref{hyp:Cramer} holds, that $a>0$ and that $\lambda_c(a) \leq \lambda < \min\{\lambda_+,\lambda_-\}$.
Let $(u^*,v^*)$ be defined as in~\eqref{eq:formulauv*}.
\begin{itemize}
\item[(i)] There are some constants $c_0=c_0(a,\lambda)>0$, $c_1=c_1(a,\lambda)>0$ (explicit in the proof)
such that
\begin{equation}
\label{eq:asympZN}
Z_{N,\lambda}^a \sim c_1 e^{\f(\lambda, a) N + c_0 \{aN\} } \qquad \text{ as } N\to+\infty \,,
\end{equation}
where $\{aN\} =  aN- \lfloor aN \rfloor$ is the fractional part of $aN$.

\item[(ii)] There are some constants $\sigma_1=\sigma_1(a,\lambda)>0$, $\sigma_2=\sigma_2(a,\lambda)>0$ (explicit in the proof, see~\eqref{def:sigmas}) such that uniformly for $\delta_{\ell}:=|\frac1N \ell- u^*|, \delta_r:=|\frac1N r-v^*| \leq \gep_N$ with $\gep_N\to 0$, we have as $N\to\infty$
\begin{equation}
\label{eq:TCLlocalZ}
\bP_{N,\lambda}^a\big( L_N = \ell, R_N =r \big) = \frac{1+o(1)}{2\pi N \sigma_1 \sigma_2} e^{- \frac{(\ell-u^*N)^2}{2\sigma_1^2 N} (1+ \delta_{\ell} O(1) )} e^{- \frac{(r-v^*N)^2}{2\sigma_2^2 N } (1+ \delta_r O(1) )}  \,.
\end{equation}
As a consequence, as $N\to\infty$, under $\bP_{N,\lambda}^a$ we have that $(\frac{L_N-Nu^*}{\sqrt{N}}, \frac{R_N-Nv^*}{\sqrt{N}})$ converges weakly to $(Z_1,Z_2)$, where  $Z_1\sim \mathcal{N}(0,\sigma_1^2)$ and $Z_2\sim \mathcal{N}(0,\sigma_2^2)$ are independent.
\end{itemize}
\end{theorem}

This result is analogous to Theorem~1.5 in~\cite{BFO09}, which has been obtained in the context of the pinning or wetting of a $d$-dimensional \emph{Gaussian} random walk on a subspace $M$; see also Remark~4.1 in~\cite{FO10}, which considers the pinning version of the model, under the assumption that the random walk has all finite exponential moments (\textit{i.e.}\ $t_0^+=t_0^- =+\infty$). Here we provide a local limit theorem for $L_N,R_N$ (and the number of contacts $H_N(S)$), under much milder conditions.

\smallskip
Notice that under $\bP_{N,\lambda}^a$, conditionally on $L_N=\ell,R_N=r$, the number of contacts $H_N = H_N(S)$ with the bottom of the well as the same distribution as under $\bP_{r-\ell,\lambda}$.
In particular, combining Theorem~\ref{th:Gaussian} with Proposition~\ref{prop:Hlocal}, we prove the following corollary.

\begin{corollary}
\label{cor:HN}
Suppose that Assumption~\ref{hyp:Cramer} holds, that $a>0$ and that $\lambda_c(a) \leq \lambda < \min\{\lambda_+,\lambda_-\}$.
Then, as $N\to\infty$, we have the following joint convergence in distribution under $\bP_{N,\lambda}^a$:
\begin{equation}
\label{eq:convLRH}
\Big(\frac{L_N-u^*N}{\sqrt{N}}, \frac{R_N-v^*N}{\sqrt{N}} ,\frac{H_N- m_{\lambda}^{-1}(v^*-u^*) N}{\sqrt{N}}  \Big)
 \xrightarrow[N\to\infty]{\bP_{N,\lambda}^a} (Z_1,Z_2,Z_3) \quad \text{ as } N\to\infty\,,
\end{equation}
where $(Z_1,Z_2, Z_3)$ is a Gaussian vector of density
\[
\frac{1}{(2\pi)^{3/2} \sigma_1 \sigma_2 \sigma_{3}} e^{-\frac12 Q(z_1,z_2,z_3)}\quad
\text{ with }\  Q(z_1,z_2,z_3) = \frac{z_1^2}{\sigma_1^2} + \frac{z_2^2}{\sigma_2^2} + \frac{(z_3- m_{\lambda}^{-1}(z_2-z_1))^2}{\sigma_{3}^2}\,,
\] 
for some constants $\sigma_1,\sigma_2,\sigma_{3}$ explicit in the proof.
In particular, $Z_3$ is Gaussian, but not independent of $(Z_1,Z_2)$.
(We actually obtain a local version of this convergence in distribution, see~\eqref{eq:localLRH}.)
\end{corollary}

\noindent
As a consequence, the density of contact $\frac1N H_N(S)$ converges in $\bP_{N,\lambda}^a$-probability to $\frac{1}{m_{\lambda}}(v^*-u^*)$, with $v^*-u^*>0$, even at the critical point $\lambda=\lambda_c(a)$.

\begin{remark}
Obtaining detailed results when $\lambda\geq min\{\lambda_+,\lambda_-\}$, \textit{i.e.}\ at the border or outside Cram\'er's region is left open. 
It would actually require further assumptions on the underlying random walk, in particular to obtain sharp local large deviation estimates, in the spirit of Theorem~\ref{th:limloc} below.
We do not pursue the investigation further, but the situation at $\lambda=\lambda_{\pm}$ when $\rho_{\pm}<+\infty$ (hence the variational problem has no unique maximizer) is particularly intriguing.
\end{remark}

\subsection{About the phase transitions}

In this section, we also focus on the case where Cram\'er's condition is satisfied, \textit{i.e.}\ both $t_0^+,t_0^->0$.
This is only to simplify the statements, but the results could be adapted to the case where only one of $t_0^+,t_0^-$ is non-zero (up to a factor $1/2$ in some constants).
On the other hand, the case $t_0^+=t_0^-=0$ is trivial, since $\f(\lambda,a)=\f(\lambda)$ for all $a\geq 0$.

\subsubsection{About the localization (or wetting) phase transition}

As far as the critical behavior near the \textit{localization} transition (at $\lambda_c(a)$) is concerned, we have the following result.
It shows in particular that, as soon as $a>0$, the phase transition in $\lambda$ is of first order. 
Note we have already observed that $a\mapsto \f(\lambda, a)$ is affine on $(0,a_c(\lambda))$.

\begin{proposition}
\label{th:DL_F}
Suppose that Assumption~\ref{hyp:Cramer} holds.
For any $a>0$ with $\lambda_c(a)<+\infty$, we have 
\[
\f(\lambda_c(a) + u,a) \sim C_a u \qquad \text{ as } u\downarrow 0\,.
\]
for some \emph{explicit} constant $C_a>0$ given in~\eqref{eq:Ca}.
The constant $C_a$ verifies $C_a\sim \sqrt{8 c_3} \, a$ as $a\downarrow 0$, with $c_3$ defined in~\eqref{eq:criticF5}.
\end{proposition}

\subsubsection{About the saturation phase transition}

Let us also give some information on the saturation transition(s) (at $\lambda_{+},\lambda_-$): we show that it is of first order if $\rho_{\pm}<+\infty$ and of second order if $\rho_{\pm}= \infty$.

Assume that $\lambda_c(a)<\lambda_-\leq\lambda_+ <+\infty$; the case $\lambda_+\leq \lambda_-$ is symmetric and the case where $\lambda_+=+\infty$ is simpler. 
We introduce the excess free energies as follows: recalling Theorem~\ref{th:formuleG} (and the fact that $\Lambda_{\pm}^{-1}(\f(\lambda_{\pm})) = t_0^{\pm}$), define for $\lambda_c(a)<\lambda$
\begin{equation}
\label{def:fexcess}
\begin{split}
\f^{\ast}(\lambda,a) & := \f(\lambda,a)- \big(\f(\lambda) - a \Lambda_+^{-1}(\f(\lambda)) -a  t_0^- \big) \,, \\
\f^{\ast\ast}(\lambda,a) & := \f(\lambda,a)- \big(\f(\lambda) - a  t_0^+-at_0^- \big) \,.
\end{split}
\end{equation}
In particular, we have that $\f^{\ast}(\lambda,a) =0$ for $\lambda \in (\lambda_-,\lambda_+)$ and $\f^{\ast\ast}(\lambda,a) =0$ for $\lambda \in (\lambda_+,+\infty)$.
In the case where $\lambda_-=\lambda_+$ then $\f^{\ast}=\f^{\ast\ast}$.
We then have the following result.

\begin{proposition}
\label{prop:saturation}
Suppose that Assumption~\ref{hyp:Cramer} holds.
If $\lambda_-<\lambda_+ <+\infty$, then as $u\downarrow0$ we have
\[
\f^{\ast}(\lambda_{-} -u ,a) = (1+o(1)) \frac{a \f'(\lambda_-)}{\rho_-}  \, u \,,\qquad
\f^{\ast\ast}(\lambda_{+} -u ,a) = (1+o(1))   \frac{a \f'(\lambda_+)}{\rho_+}  \, u  \,.
\]
If $\lambda_-=\lambda_+<+\infty$, then
$\f^{\ast}(\lambda_{\pm} -u ,a) = (1+o(1))  a  \f'(\lambda_{\pm}) \big( \frac{1}{\rho_-}  + \frac{1}{\rho_+}\big) u$ as $u \downarrow 0$.
\end{proposition}

\noindent
In particular, we find that $\f^{\ast}(\lambda_{\pm} -u ,a) = o(u)$ if $\rho_{\pm}=+\infty$.

\subsubsection{About the critical curve}

Let us now provide some information about the critical curve, in particular close to $\lambda_c$ and $\lambda_-, \lambda_+$ (again, assume $\lambda_-\leq \lambda_+ <+\infty$).
To study $a_c(\lambda)$ as~$\lambda$ crosses the values $\lambda_-,\lambda_+$, we define similarly as above (recalling the formulas of  Theorem~\ref{th:formuleG}):
\begin{equation}
\label{def:acexcess}
a_c^{\ast}(\lambda) = a_c(\lambda) - \frac{\f(\lambda)}{ \Lambda_+^{-1} (\f(\lambda)) + t_0^-} \,,\qquad
a_c^{\ast\ast}(\lambda) = a_c(\lambda) - \frac{\f(\lambda)}{t_0^+ + t_0^-}  \,,
\end{equation}
so that $a_c^{\ast}(\lambda) =0$ for $\lambda \in (\lambda_-,\lambda_+)$ and $a_c^{\ast\ast}(\lambda) =0$ for $\lambda \in(\lambda_+,+\infty)$.
Also, $a_c^{\ast}=a_c^{\ast\ast}$ in the case where $\lambda_-=\lambda_+$.

\begin{proposition}
\label{prop:curve}
Suppose that Assumption~\ref{hyp:Cramer} holds.

\smallskip
(i) We have 
\[
a_c(\lambda_c+u)\sim \frac{\sigma^2}{2\sqrt{2}} \sqrt{c_3}  u \qquad \text{ as } u\downarrow 0\,,
\]
where $\sigma^2$ is the variance of $X_1$ and $c_3$ is defined in~\eqref{eq:criticF5}.

\smallskip
(ii) If $\lambda_-<\lambda_+ <+\infty$, then as $u\downarrow 0$ we have
\[
a_c^{\ast}(\lambda_{-} -u)  =(1+o(1))  \frac{ \f(\lambda_-)\f'(\lambda_-)}{(\Lambda_+^{-1}(\f(\lambda_-))+t_0^-)^2\rho_-}  \, u  \,,\quad
a_c^{\ast\ast}(\lambda_{+} -u ) =(1+o(1))   \frac{\f(\lambda_+)\f'(\lambda_+)}{(t_0^-+t_0^+)^2\rho_+}  \, u \,.
\]
If $\lambda_-=\lambda_+ <+\infty$, then
$a_c^{\ast}(\lambda_{\pm} -u ) =(1+o(1))  \frac{ \f(\lambda_{\pm})\f'(\lambda_{\pm})}{(t_0^++t_0^-)^2} \big( \frac{1}{\rho_-}  + \frac{1}{\rho_-} \big) u$ as $ u \downarrow 0$.
\end{proposition}

\noindent
In particular, we find that $\lambda\mapsto a_c(\lambda)$ remains differentiable at $\lambda_{\pm}$ if $\rho_{\pm} =+\infty$; if $\rho_{\pm}<+\infty$, the left and right derivative at $\lambda_{\pm}$ differ.


\section{Some related models and open questions}
\label{sec:comments}


\subsection{Wetting on a tilted wall or in a convex well}

Let $(X_i)_{i\geq 1}$ be i.i.d.\ real random variables, with a density $f(\cdot)$ with respect to the Lebesgue measure.
For a given function $\varphi:[0,1] \to \mathbb R$ we define for $N\geq 1$ the function $\varphi_N: \{0,\ldots N\} \to \mathbb{R}$ by setting $\varphi_N(i) = \frac1N \varphi(\frac{i}{N})$.
Then, we introduce the following Gibbs measure, analogously to~\eqref{eq:PNB}:
\[
 \dd \bP_{N,\lambda}^{\varphi} (s_1,\ldots, s_n)= \frac{1}{Z_{N,\lambda}^{\varphi}} \prod_{i=1}^{N-1} \big( \ind_{\{s_i > \varphi_N(i)\}} \mu(\dd s_i)  + \lambda \delta_{\varphi_N(i)}(\dd s_i) \big) \prod_{i=1}^N f(s_i-s_{i-1})  \lambda\delta_{\varphi_N(N)}(\dd s_N) \,.
\]
This corresponds to a wetting model on a wall of shape $\varphi$; note that we consider only the continuous model so that there is no restriction on the function $\varphi$.
Two natural examples are the following
\begin{itemize}
\item[(i)] If $\varphi(x) =\mu x$ for some $\mu\in \mathbb R$, it corresponds to a wetting model on a wall with slope $\mu$;

\item[(ii)] If $\varphi(0)=\varphi(1)=0$ and $\varphi$ is convex, this corresponds to a wetting model in a convex well.
\end{itemize}
We then define the free energy $\f(\lambda, \varphi) = \lim_{N\to\infty} \frac1N \log Z_{N,\lambda}^{\varphi}$ (one has to show that it exits) and one would like to obtain some explicit expression for it.

\subsubsection*{(i) For a wall of slope $\mu$.}
If $\varphi(x)=\mu x$, let us denote $\f(\lambda,\varphi) =: \f_{\mu}(\lambda)$.
In that case, in the Cram\'er region, that is if $\mu \in (\rho_-,\rho_+)$, then we should have that
\[
\f_{\mu}(\lambda) = \f^{(\mu)}(\lambda) - \I(\mu) \,,
\]
where $\f^{(\mu)}(\lambda)$ is the free energy of the wetting model with a wall of slope $0$ but with underlying random walk with i.i.d.\ increments $\tilde X_i = X_i^{(\mu)}-\mu$, where $X_1^{(\mu)}$ has the tilted law $\bP^{(\mu)}(\dd x) = e^{t_{\mu} x - \Lambda(t_\mu)} \bP(\dd x)$ with $t_{\mu}$ chosen so that $\bE^{(\mu)}[X_1^{(\mu)}] =\mu$.
For instance, in the case of a (standard) Gaussian random walk, since $\I(\mu) = \frac12 \mu^2$ and $\bP^{(\mu)}\sim \mathcal{N}(\mu,1)$, we should get $\f_{\mu}(\lambda) =   \f(\lambda) -\frac12 \mu^2$.

\subsubsection*{(ii) For a convex well.}
In the case where $\varphi(0)=\varphi(1)=0$ and $\varphi$ is convex, then we should get the following.
For $0\leq u\leq v\leq 1$, set
\[
g_{\lambda, \varphi}(u,v) := \int_u^v \f_{\varphi'(t)}(\lambda) \dd t   - u\, \I_-\Big( \frac{|\varphi(u)|}{u}\Big) - (1-v)\, \I_+\Big( \frac{|\varphi(v)|}{1-v}\Big)  \,,
\]
where $\f_{\varphi'(t)}(\lambda)$ is the free energy of the wetting model on a wall of slope $\varphi'(t)$.
Then, setting $\psi(\lambda, \varphi) = \sup_{0\leq u \leq v\leq 1} g_{\lambda, \varphi}(u,v)$, we expect that
\begin{equation*}
\f(\lambda,\varphi) = \max\{\psi(\lambda, \varphi),0\}\,,
\end{equation*}
generalizing the formula found in Theorem~\ref{th:freeenergy}.
For instance, the formulas should slightly simplify for a (standard) Gaussian random walk: we should get $g_{\lambda,\varphi}(u,v) = (v-u) \f(\lambda) - E_{u,v}(\varphi)$, with $E_{u,v}(\varphi) :=  \frac12 \int_u^v \varphi'(t)^2 \dd t + \frac{1}{2u} \varphi(u)^2+ \frac{1}{2(1-v)} \varphi(v)^2$.

\subsection{Wetting on a random walk}

Another model that seems natural to consider but has not been studied in the literature (to our knowledge) is the wetting model in the case where the wall is random (but quenched), given by the realization of another random random walk $(Y_n)_{n \geq 0}$.

More precisely, let $Y=(Y_n)_{n \geq 0}$ be a random walk and consider, for a fixed (\textit{i.e.}\ quenched) realization of $Y$, the Gibbs measure
\[
 \dd \bP_{N,\lambda}^{Y} (s_1,\ldots, s_n)= \frac{1}{Z_{N,\lambda}^{Y}} \prod_{i=1}^{N-1} \big( \ind_{\{s_i > Y_i\}} \mu(\dd s_i)  + \lambda \delta_{Y_i}(\dd s_i) \big) \prod_{i=1}^N f(s_i-s_{i-1})  \lambda \delta_{Y_N}(\dd s_N) \,.
\]
It corresponds to a wetting model of the random walk $(S_n)_{n\geq 0}$ above the (quenched) wall $(Y_n)_{n\geq 0}$.
In the discrete case, this can be rewritten and
\begin{equation}
\label{def:wettingRW}
\frac{\dd \bP_{N,\lambda}^{Y} }{\dd \bP}(S) =  \frac{1}{Z_{N,\lambda}^{Y}} \lambda^{\sum_{n=1}^N \ind_{\{S_n=Y_n\}}} \ind_{\{ S_i\geq Y_i \text{ for all } i \in\{ 1,\ldots, N\} \}} \,.
\end{equation}

By Kingman's super-additive ergodic theorem, one can show that the \emph{quenched} free energy
exists and is a.s.\ constant, \textit{i.e.}\ almost surely does not depend on the specific realization of $Y$: 
\[
\f^{\rm que}(\lambda) := \lim_{N\to\infty} \frac1N \log Z_{N,\lambda}^Y = \lim_{N\to\infty} \frac1N \bE_Y\big[\log Z_{N,\lambda}^Y\big]  \quad  \bP_Y\text{-a.s.} \text{ and in } L^1(\bP_Y)\,.
\]
Additionally, one may define the \emph{quenched} critical point
$\lambda_c^{\rm que} := \inf\{\lambda, \f^{\rm que}(\lambda) >0\}$.

Let us observe that the \emph{annealed} model, with partition function $\bE_Y\big[Z_{N,\lambda}^{Y} \big]$, corresponds to the usual wetting model of Section~\ref{sec:wetting}, with underlying random walk $(X_n-Y_n)_{n\geq 0}$. 
Hence its critical point $\lambda_c^{\rm ann}$ (and the critical behavior of the annealed free energy) is explicit, see Theorem~\ref{th:freeenergy}.
Furthermore, applying Jensen's inequality $\bE_Y[\log Z_{N,\lambda}^Y] \leq \log \bE_Y[Z_{N,\lambda}^Y]$, one obtains that the annealed free energy dominates the quenched one, and in particular $\lambda_c^{\rm que} \geq \lambda_c^{\rm ann}$.
The question of then to know whether one has the strict inequality $\lambda_c^{\rm que}>\lambda_c^{\rm ann}$ or not.

\subsubsection*{Quenched vs.\ annealed critical point}
A variant of~\eqref{def:wettingRW}, where the constraint that $S_i\geq Y_i$ for all~$1\leq i\leq N$ is removed, is known as the Random Walk Pinning Model and has been studied in several instances, see~\cite{BT10,BS10,BS11}.
For this model, the question of the strict inequality between the quenched and annealed critical point has some consequences in several interacting stochastic systems, see~\cite{BGdH11}.
The above references mostly consider the case of the simple symmetric random walk on $\mathbb Z^d$: they show that quenched and annealed critical points are equal (to $1$) in dimension $d=1,2$ and that they differ in dimension $d\geq 3$.
The case of symmetric random walks on $\mathbb Z$ with Assumption~\ref{hyp:1} is also (partially) treated in~\cite{BGdH11}: they prove that the quenched and annealed critical points are equal if $\alpha\geq 1$ and different if $\alpha<\frac12$.

The question of the strict inequality between the quenched and annealed critical point for the wetting on a random walk~\eqref{def:wettingRW} is however slightly different and remains open. 
Abiding by Harris' predictions~\cite{H74} for disorder relevance in physical systems\footnote{The prediction is made for a disorder which is i.i.d., but it has been confirmed in the random walk pinning model, see~\cite{BT10,BS10,BS11} and also~\cite{AB18a}.}, since the annealed model has a critical exponent $\nu =\min(1,\alpha)$, which is smaller than $2$ in the case $\alpha<2$, disorder should then be relevant, meaning that $\lambda_c^{\rm que}>\lambda_c^{\rm ann}$.
The case $\alpha=2$ (of random walks in the Gaussian domain of attraction) is \emph{marginal} for the question of disorder relevance and there is no clear prediction; one might expect that in analogy with the pinning models \cite{BL18,GLT10}, one has $\lambda_c^{\rm que}>\lambda_c^{\rm ann}$ also in that case, but this is a questionable conjecture here.

\subsection{Wetting in higher dimension}
 
Another natural model to consider is the wetting of a square well in dimension $d\geq 2$.
For some compact domain $B \subset \mathbb{R}^d$ (of connected interior), we could consider the wetting of an interface in a well of shape $B$.
Let $B_N := (N B) \cap \mathbb{Z}^d$ and consider the following model of wetting in a well of depth $a N$:
\[
\dd \bP_{B_N,\lambda}^a := \frac{1}{Z_{B_N,\lambda}^a} e^{- H_{B_N}^0(\phi)} \prod_{x\in B_N} \big( \mu(\dd \phi_x) \ind_{\{\phi_x > -\lfloor aN\rfloor \}} + \lambda \delta_{-\lfloor aN \rfloor}(\dd \phi_x) \big)\,,
\]
where $H_{B_N}^0(\phi)$ is a Hamiltonian associated with a random interface $\phi: B_N \to \mathbb{R}^d$ (or $\mathbb Z^d$), with zero boundary condition (that is such that $\phi_x \equiv 0$ on $B_N^c$).

A natural choice for the underlying interface model is a gradient model, \textit{i.e.}\ a Hamiltonian
\[
H_{B_N}^0(\phi) := \sumtwo{ x,y \in B_N}{x\sim y} V(\phi_x-\phi_y) +  \sumtwo{ x \in B_N, y\notin B_N}{x\sim y} V(\phi_x) \,,
\]
for some potential $V(\cdot)$; we refer to~\cite{Fun05} for an overview.
The choice $V(x) = \frac12 x^2$ corresponds to the massless Gaussian Free Field (GFF); the choice $V(x) =|x|$ corresponds to the Solid-On-Solid (SOS) model.
The wetting of the GFF and SOS has been considered in~\cite{CV00}, where it is shown that $\lambda_c >1$ for the GFF in dimension $d=2$ or for the SOS model in any dimension $d\geq 2$; on the other hand, one has that $\lambda_c =1$ for the GFF in dimension $d\geq 3$, see~\cite{BDZ00,GL18}.

A first goal would be to obtain a formula for the free energy of the model, defined as
\[
\f_B(\lambda, a) := \lim_{N\to\infty} \frac{1}{N^d} \log Z_{B_N,\lambda}^a \,.
\]
In the spirit of Theorem~\ref{th:freeenergy}, we expect that one has $\f_B(\lambda, a) = \max\{\psi_B(\lambda,a),0\}$, with
\begin{equation}
\label{def:psiD}
\psi_B(\lambda,a) := \sup_{D \subset B} \big\{ \f(\lambda) |D| - I_B(a,D)  \big\} \,.
\end{equation}
Here, $\f(\lambda) := \lim_{N\to\infty} \frac{1}{|D_N|} \log Z_{D_N,\lambda}^{a=0}$ is the free energy per unit volume of the (usual) wetting model and $I_B(a,D)$ is the rate function for the large deviations of the interface:
\[
I_B(a,D)  := \lim_{N\to\infty} - \frac{1}{N^d} \log \bP_{B_N}\big( \phi_x \leq -aN \ \forall x\in D_N \big) \,,
\]
where $\bP_{B_N}$ is the law of the underlying interface, that is $\dd \bP_{B_N} (\phi) := \frac{1}{Z_{B_N}} e^{ -H_{B_N}^0 (\phi)} \prod_{x\in B_N} \mu(\dd \phi_x)$.
The main difficulty in proving~\eqref{def:psiD} should come from the large number of choices for the possible microscopic ``wetted'' regions; in analogy with~\cite{BI97}, a coarse graining of these regions might be necessary.

We stress that for strictly convex potentials $V(\cdot)$ with bounded second derivative (in particular for the GFF), \cite{DGI00} proves a strong large deviation principle for the rescaled surface $\frac1N \phi$, with speed~$N^d$ and rate function $\Sigma (u) = \int_{B} \sigma(\nabla u(\theta)) \dd \theta$, where $\sigma(\cdot)$ is the surface tension of the model.
In particular, for the GFF one has $\sigma(v) = \frac12 \|v\|^2$, so one obtains that the above rate function is $I_B(a,D) = d a^2 \mathrm{Cap}_B(D)$, where $\mathrm{Cap}_B(D)$ is the Newtonian capacity of $D$ with respect to $B$, namely
\[
\mathrm{Cap}_B(D) := \sup \Big\{ \frac{1}{2d} \int_{D} \|\nabla f(x)\|^2 \dd x \,, f \in \mathcal{C}_c^{\infty}, f(x)\geq 1 \text{ for } x\in D,  f(x)=0 \text{ for } x\in B^c\Big\} \,.
\]
In general, if $V(x) = c_0 \|x\|^p$ for some $p\geq 1$ (in particular for the SOS model if $p=1$), one should also obtain that $I_B(a,D) = 2d c_0 a^p \mathrm{Cap}_B^{(p)}(D)$, where $\mathrm{Cap}_B^{(p)}(D)$ is the $p$-capacity of $D$ with respect to $B$, namely 
\[
\mathrm{Cap}_B^{(p)}(D) := \sup \Big\{ \frac{1}{2d} \int_{D} \|\nabla f(x)\|^p \dd x \,, f \in \mathcal{C}_c^{\infty}, f(x)\geq 1 \text{ for } x\in D,  f(x)=0 \text{ for } x\in B^c\Big\} \,.
\]

\begin{remark}
One could also consider other type of interfaces, for instance with Laplacian interactions (also known as the \emph{membrane} model), where the Hamiltonian is $H_{B_N}^0(\phi) = \sum_{x\in D_N} V(\Delta \phi_x)$, with $\Delta$ the discrete Laplacian.
This is already an interesting question in dimension $d=1$, where the wetting model has been considered in~\cite{CD08} (it is shown that $\lambda_c>1$).
The $\delta$-pinning of the membrane model has also been considered in~\cite{BCK16,Schweiger22} in dimension $d\geq 4$, but many open questions remain, especially in dimension $d=2,3$ where it is not known whether the critical point is strictly positive or not.
We refer to the above references for more details on this model.
\end{remark}

\section{Large deviations: notation and preliminary observations}
\label{sec:LD}

Recall that the rate functions $\mathrm{I}_+,\mathrm{I}_-$ defined in~\eqref{def:rateI} are  (upward and downward) large deviation rate functions for $(S_n)_{n\geq 0}$, see~\cite[Thm.~2.2.3]{DZ09}, namely
\begin{equation}
\label{eq:CramerLDP}
\lim_{n\to\infty} \frac1n \log \bP(S_n \geq x n) = -\I_+(x) \,,
\qquad \lim_{n\to\infty} \frac1n \log \bP(S_n \leq -x n) = -\I_-(x)\,.
\end{equation}
Note that the remark after~\cite[Thm.~2.2.5]{DZ09} (or simply using Chernov's exponential inequality) tells that we have the following useful upper bounds: for any $n\geq 0$ and $x\geq 0$
\begin{equation}
\label{eq:boundLDP}
 \bP(S_n \geq x n) \leq e^{- n \I_+(x)}\,,
 \quad  \bP(S_n \leq - x n) \leq e^{-n \I_-(x)} \,.
\end{equation}

\subsection{Some remarks on the rate functions}
\label{sec:LDrem}

Let us recall that for $t\geq 0$ we have defined  $\Lambda_+(t)= \log \bE[e^{tX_1}]$ and $\Lambda_-(t)= \log \bE[e^{-tX_1}]$ and their respective radius of convergence:
\[
t_0^+ := \sup\{t\geq 0, \Lambda_+(t)<+\infty\} \in [0,+\infty] \,,\quad 
t_0^- := \sup\{t\geq 0, \Lambda_-(t)<+\infty\} \in [0,+\infty] \,.
\]
Some standard properties of the log-moment generating functions $\Lambda_{\pm}$ and of the rate functions~$\I_{\pm}$ are given in \cite[Lem.~2.2.5]{DZ09}.
Let us summarize some of the properties here for~$\Lambda_+$ and $\I_+$ (similar statements hold for $\Lambda_-$ and $\I_-$):
\begin{itemize}
\item $\Lambda_{+}$ is non-decreasing and convex and $\I_{+}$ is a convex non-decreasing rate function;
\item If $t_0^+=0$ then $\I_+(x) =0$ for all $x\geq 0$;
\item If $t_0^+>0$, then $\Lambda_+$ is differentiable and strictly convex on $[0,t_0^+)$ and $\I_+$ is strictly increasing on $\{x\geq 0, \I_+(x)<+\infty\}$.
Additionally, $\Lambda_+'(s)=y$ implies that $\I_+(y) = sy-\Lambda_+(s)$.
\end{itemize}
Lets us recall the definition of $\rho_+,\rho_-$, the limiting slope of $\Lambda$ when approaching the radius of convergence, as in~\eqref{def:rho}:
\[
\rho_+ := \lim_{t\uparrow t_0^+} \Lambda'_+(t) \,, \qquad \rho_- := \lim_{t\uparrow t_0^-} \Lambda'_-(t) \,.
\]
Let us also introduce the extremal points of the support of the random walk increments:
\begin{equation}
\label{def:xpm}
\bar x_+ := \sup\{ x\in\mathbb{N}, \bP(X_1=x)>0\}\,,
\qquad 
\bar x_- := \sup\{x\in \mathbb{N} , \bP(X_1=-x)>0 \} \,.
\end{equation}
There are mostly four cases that we need to consider (we focus on the ``$+$'' case but similar considerations hold for the ``$-$'' case): we summarize them in Table~\ref{table:LDP} below.
%

\begin{center}
\begin{table}[htbp]
\renewcommand{\arraystretch}{1.3}
\begin{tabular}{|c|c|c|c|} 
\hline
\phantom{----}$t_0^+ = +\infty$\phantom{----} &   \phantom{----}$t_0^+ < +\infty$\phantom{----}&   \phantom{----}$t_0^+ < +\infty$\phantom{----} &\phantom{----}$t_0^+ = +\infty$\phantom{----} \\ 
$\bar x_+ =+\infty$ &  $\bar x_+ =+\infty$
& $\bar x_+ = + \infty$ & $\bar x_+ = \rho_+ < +\infty$ \\
\hline
\multicolumn{2}{|c|}{$\rho_+ =+\infty$} 
& \multicolumn{2}{|c|}{$\rho_+ < + \infty$}  \\
\hline
\multicolumn{2}{|c|}{$\mathrm{I}_+$ is strictly convex on $\mathbb R_+$} &  \multicolumn{2}{|c|}{$\I_+$ is strictly convex and finite on $[0,\rho_+]$ and}\\ 
\multicolumn{2}{|c|}{and $\lim_{x\to\infty}\mathrm{I}_+(x)=+\infty$} & is affine on $(\rho_+,+\infty)$ & is infinite on $(\rho_+,+\infty)$ \\
\hline
\multicolumn{2}{|c|}{
\begin{tikzpicture}
  \draw[->] (0, 0) -- (3, 0) node[right] {$x$};
  \draw[->] (0, 0) -- (0, 2) node[left] {\small $\I_+(x)$};
  \draw[scale=1, domain=0:2.5, thick, smooth, variable=\x, blue] plot ({\x}, {\x*\x/3});
  \phantom{ \draw[-] (1.5,-0.1) -- (1.5,0) node[below] {$\rho_+$};}
  \phantom{ \draw[-] (0,0) -- (0,2.5) node[below] {$\rho_+$};}
\end{tikzpicture}
}
& 
\begin{tikzpicture}
  \draw[->] (0, 0) -- (3, 0) node[right] {$x$};
  \draw[->] (0, 0) -- (0, 2) node[left] {\small $\I_+(x)$};
  \draw[scale=1, domain=0:1.5, thick, smooth, variable=\x, blue] plot ({\x}, {\x*\x/3});
  \draw[scale=1, domain=1.5:2.8, thick, smooth, variable=\x, blue] plot ({\x}, {\x-0.75});
  \draw[-] (1.5,-0.1) -- (1.5,0) node[below] {$\rho_+$};
  \node[circle,fill=blue, inner sep=0pt,minimum size=4pt] (b) at (1.5,0.75) {};
\end{tikzpicture}
&
\begin{tikzpicture}
  \draw[->] (0, 0) -- (3, 0) node[right] {$x$};
  \draw[->] (0, 0) -- (0, 2) node[left] {\small $\I_+(x)$};
  \draw[scale=1, domain=0:1.5, thick, smooth, variable=\x, blue] plot ({\x}, {\x*\x/3});
  \draw[-] (1.5,-0.1) -- (1.5,0) node[below] {$\rho_+=\bar x_+$};
  \node[circle,fill=blue, inner sep=0pt,minimum size=4pt] (b) at (1.5,0.75) {};
  \draw[thick, blue] (1.5,2.1) -- (2.7,2.1) node[right]{$+\infty$};
\end{tikzpicture}
\\
\hline
\end{tabular}
\medskip
\caption{Summary of the different possibilities if $t_0^+>0$. In all cases, $\I_+$ is strictly convex on $[0, \rho_+)$. Similar statements hold true for $\I_-$.}
\label{table:LDP}
\end{table}
\vspace{-1\baselineskip}
\end{center}

\noindent
Let us introduce, for $t \in [0,t_0^+)$, the tilted measure $\bP_t$ as
\begin{equation}
\label{def:tiltedPt0}
\frac{\dd \bP_t}{\dd \bP} (x) = \frac{e^{tx}}{\bE[e^{tX_1}]}  = e^{tx-\Lambda_+(x)} \,.
\end{equation}
Notice that $\Lambda_+$ is analytic on $[0,t_0^+)$ and that $\Lambda'_+(t) = \bE_t[X_1]$ is strictly increasing from $[0,t_0^+)$ to $[0,\rho_+)$ (Assumption~\ref{hyp:1} implies that $\bE[X_1]=0$ whenever $X_1$ has a finite expectation) and that $\Lambda''(t)= \mathbf{V}\mathrm{ar}_{t}(X_1)>0$ (otherwise $X_1$ would be a.s.\ constant).

For all $x\in [0,\rho_+)$, the supremum in the definition~\eqref{def:rateI} is attained at $t_x  := (\Lambda_+')^{-1}(x)$, such that $\Lambda_+'(t_x) =x$.
We therefore get, after some classical calculations
\begin{equation}
\label{formulasI+}
\begin{split}
&\I_+(x)   =  x t_x - \Lambda_+\left(t_x\right) = x (\Lambda_+')^{-1}(x) - \Lambda_+\left((\Lambda_+')^{-1}(x)\right) \\
&\I_+'(x) = t_x = (\Lambda_+')^{-1}(x) \,, \qquad \I_+''(x) = \frac{1}{\Lambda''((\Lambda_+')^{-1}(x))} = \frac{1}{\mathbf{V}\mathrm{ar}_{t_x}(X_1)}<+\infty \,.
\end{split}
\end{equation}
Note also that if $\rho_+ < \infty$ and $t_0^+ < \infty $, then $\I_+$ is defined (and affine) on $[\rho_+,+\infty)$, see Table~\ref{table:LDP}: we have $\I_+(x) = x t_0^+ - \Lambda_+(t_0^+)$.
In that case, $\I_+$ is differentiable on $\mathbb{R}_+$, with $\I_+'(x) = t_0^+ = (\Lambda_+')^{-1} (\rho_+)$ for $x\geq \rho_+$,

To summarize, we have that $\I_+$ is differentiable on $[0,\bar x_+)$ and left-differentiable at $\bar x_+$ with
\begin{equation}
\label{derivI}
\I_+'(x) = 
\begin{cases}
(\Lambda_+')^{-1}( x) & \quad \text{ if } x \in (0,\rho_+) \,, \\
 t_0^+  & \quad \text{ if } x \geq \rho_+ \,,
 \end{cases}
\end{equation}
which is continuous in the case where $\rho_+<+\infty$.
As another useful formula that derives from the above, we have
\begin{equation}
\label{I-xI'}
\I_+(x)-x \I_+'(x) = -
\begin{cases}
\Lambda_+\circ (\Lambda_+')^{-1}( x) & \quad \text{ if } x \in (0,\rho_+) \,, \\
\Lambda_+(t_0^+)  & \quad \text{ if } x \geq \rho_+ \,.
 \end{cases}
\end{equation}

\subsection{A local large deviations result}

We complement here Cram\'er's large deviations~\eqref{eq:CramerLDP} with a local version; we were not able to find a reference for it, so we prove it in Appendix~\ref{sec:LDPequal}.

\begin{lemma}
\label{lem:LDPequal}
Suppose that Assumption~\ref{hyp:1} holds.
For any any sequence of non-negative integers $(x_n)_{n\geq 1}$ such that $\lim_{n\to\infty} \frac1n x_n =x <\bar x_+$, we have the following local large deviation behavior:
\[
\lim_{n\to\infty} \frac1n \log \bP(S_n = x_n ) = -\I_+(x) \qquad \text{ and }\qquad
\lim_{n\to\infty} \frac1n \log \bP(S_n = - x_n) = -\I_-(x) \,.
\]
\end{lemma}

\noindent
Note that if $x=\bar x_+ <+\infty$, then we have $\bP(S_n = \bar x_+ n) =  \bP(X_1=\bar x_+)^n$, which also gives $\lim_{n\to\infty} \frac1n \log \bP(S_n = x_n ) = -\I_+(x)$ since one can easily check that $\I_+(\bar x_+)= -\log \bP(X_1= \bar x_+)$.
However, the result cannot hold for any sequence $x_n$ such that $\lim_{n\to\infty} \frac1n x_n =\bar x_+$; for instance $\bP(S_n = x_n) =0$ for any $x_n>\bar x_+ n$.


\section{Optimizers of $\psi(\lambda,a)$, formula for the free energy and phase transitions}
\label{sec:transi}

In this section, we prove Lemma~\ref{lem:psi} that gives the location of the maximizers of $\psi(\lambda,a)$. We then derive from it the formulas for the free energy $\f(\lambda)$ and the critical line $a_c(\lambda)$ given in Theorem~\ref{th:formuleG}, and we deduce the critical behavior of $\f(\lambda,a)$ and $a_c(\lambda)$  of Proposition~\ref{th:DL_F}.

Before that, let us make one simple observation on the limitation of the depth of the well.

\begin{lemma}
\label{lem:depth}
Recall the definition~\eqref{def:xpm} of $\bar x_+,\bar x_-$; notice also that $ \bar x_{\pm} =\sup\{ x\,, \mathrm{I}_{\pm}(x) < +\infty\}$. Let us define 
\begin{equation}
\label{def:abar}
\bar a := \Big(\frac{1}{\bar x_-} +\frac{1}{\bar x_+} \Big)^{-1}  = \frac{\bar x_+ \bar x_-}{\bar x_+ + \bar x_-} \in (0,+\infty]\,,
\end{equation}
with the convention $\frac{1}{\infty} =0$  and $\frac{1}{0} = +\infty$ in cases where $\bar x_+$ and/or $\bar x_-$ are infinite.
Then, if $a>\bar a$, the walk cannot reach the bottom of the well, \textit{i.e.}\ $\bP( H_N^a(S) \geq 1) =0$.
\end{lemma}

\begin{proof}
If the walk has a contact with the bottom of the well, \textit{i.e.} $H_N^a(S)\geq 1$, then by definition of $\bar x_+$, $\bar x_-$, the left-most contact point $\bP$-a.s.\ verifies $L_N \geq \frac{\lfloor aN \rfloor}{\bar x_-}$ and the right-most one $N-R_N \geq \frac{\lfloor aN \rfloor}{\bar x_+}$.
Since we must have $L_N + N-R_N \leq N$, this implies that $a (\frac{1}{\bar x_-} + \frac{1}{\bar x_+}) \leq 1$, that is $a\leq \bar a$.
\end{proof}

\begin{remark}
\label{rem:abar}
In the case $a=\bar a$, since $H_N(S) \leq R_N-L_N+1$, the same reasoning gives that
$H_N(S) \leq N- \lfloor \bar{a} N \rfloor (\frac{1}{\bar x_-} + \frac{1}{\bar x_+})  +1 \leq \frac{1}{\bar x_-} + \frac{1}{\bar x_+}+1$, where we have used that $\lfloor \bar{a} N \rfloor \geq \bar{a} N -1$ with $\bar a (\frac{1}{\bar x_-} + \frac{1}{\bar x_+})=1$.
This shows that in the case $a=\bar a$, the number of contacts is bounded by $\bar a^{-1} +1$.
\end{remark}

\noindent
\begin{remark}
\label{rem:psi-abar}
The reasoning of Lemma~\ref{lem:depth} also translates into the fact that we have $\psi(\lambda, a) = -\infty$ for any $a>\bar a$.
Indeed, having $u\mathrm{I}_-(\frac{a}{u})+(1-v)\mathrm{I}_+(\frac{a}{1-v}) <+\infty$ with $u\leq v$ means that $\frac{a}{u} \leq \bar x_-$ and $\frac{a}{1-u} \leq \frac{a}{1-v} \leq \bar x_+$, since $\I_{\pm}(x)=+\infty$ if $x>\bar x_{\pm}$.
We therefore end up with $\frac{1}{a} = \frac{u}{a} + \frac{1-u}{a} \geq \frac{1}{\bar x_-} + \frac{1}{\bar x_+}$, which implies that $a\leq \bar a$.

In the case $a=\bar a <+\infty$, the only non-zero term in the supremum~\eqref{def:psi} is at $u_0 = \frac{\bar x_-}{\bar a}$ and $v_0=1-\frac{\bar x_+}{\bar a} =u_0$, so that
\begin{equation}
\label{eq:caseabar}
\psi(\lambda,\bar a) =g_{\lambda,\bar a} (u_0,v_0) = - \frac{\bar a}{\bar x_-} \I_-(\bar x_-) - \frac{\bar a}{\bar x_+} \I_+(\bar x_+)  \in (-\infty,0) \,.
\end{equation}
\end{remark}

\subsection{Optimizers of $\psi(\lambda,a)$: proof of Lemma~\ref{lem:psi}}

Let us fix $\lambda, a\leq \bar a$ and recall the definition of $g_{\lambda,a}$ in~\eqref{def:psi}: for $(u,v)\in \mathcal D=\{(u,v)\,,\, 0\leq u \leq v \leq 1\}$,
\begin{equation}
\label{def:g}
g_{\lambda, a} (u,v) :=  (v-u) \f(\lambda) - u \,\I_-\Big(\frac{a}{u}\Big) - (1-v)\, \I_+\Big(\frac{a}{1-v}\Big) \,,
\end{equation}
so that $\psi(\lambda, a):= \sup_{(u,v)\in \mathcal D} g_{\lambda,a}(u,v)$.
Notice that since $\mathrm{I}_{\pm}(x) = +\infty$ if $x >\bar x_{\pm}$, then if there is a maximum for $g_{\lambda,a}$ it must be reached for $(u,v)\in \mathcal{D}_a = \{\frac{a}{\bar{x}_-} \le u \le v \le 1 - \frac{a}{\bar{x}_+} \}$. 
 
Let us first deal with the degenerate cases.
Recall the definition~\eqref{t0pm} of $t_0^{+},t_0^-$.

\medskip
\noindent
\textit{Case $t_0^+=t_0^-=0$.~} 
In that case, we have $\I_+=\I_-\equiv 0$, so $g_{\lambda,a}(u,v)= (v-u) \f(\lambda)$, which is clearly maximized for $(u^*,v^*)=(0,1)$.
Note that the value of $(u^*,v^*)$ matches the formula for $U^*\times V^*$ since $w_-^*=w_+^*=0$ in that case.

\medskip
\noindent
\textit{Case $a=\bar a <+\infty$.~}
Since $\frac{\bar{a}}{\bar{x}_-} = 1-\frac{\bar{a}}{\bar{x}_+}$, then as noticed in Remark~\ref{rem:psi-abar}, the supremum in~\eqref{def:psi} is attained at the unique value $(u_0,v_0)$ and  $\psi(\lambda, \bar a)=g_{\lambda,\bar a} (u_0,v_0)<0$, see~\eqref{eq:caseabar}.  This in turn shows that $\lambda < \lambda_c(a)$, and in particular $\textsc{f}(\lambda,\bar a)=0$.
Note that the value $(u_0,v_0)$ does not match here the formula for $U^*\times V^*$ since $w_-^*+ w_+^*>1$ in that case (note that we have $t_0^+ =+\infty$ if $\bar x_+<+\infty$, see Table~\ref{table:LDP}, so in view of the formula~\eqref{def:w*} we have $w_+^*> \frac{\bar a}{\rho_+} = \frac{\bar a}{\bar x_+}$).


\medskip
\noindent
\textit{Case $a<\bar a$.~}
In that case we have $\psi(\lambda, a)>-\infty$.
Note that we can write 
\begin{equation}
\label{rewriting}
g_{\lambda,a}(u,v) = \f(\lambda) - g_{\lambda,a}^-(u)- g_{\lambda,a}^+(1-v) \,,
\end{equation}
with 
\[
g_{\lambda,a}^{-}(w_1) = w_1\f(\lambda) + w_1 \I_{-}\Big(\frac{a}{w_1}\Big)\,,\qquad
g_{\lambda,a}^+(w_2) = w_2\f(\lambda) + w_2 \I_+\Big(\frac{a}{w_2}\Big)\,.
\]
As a first (and central) step, we therefore identify the minimizers of $g_{\lambda,a}^{+}(\cdot)$, resp.\ $g_{\lambda,a}^-(\cdot)$.

\begin{lemma}
\label{lem:g+}
We have
$\inf_{w\geq 0} g_{\lambda,a}^+(w) = a \Lambda_+^{-1}(\f(\lambda))$,
and the infimum is attained on $W^*_+$, with
\[
W^*_+:=
\begin{cases}
\{w_+^*\}  & \text{ if } \f(\lambda)<\Lambda_+(t_0^+)\,, \\
[0,\frac{a}{\rho_+}] & \text{ if } \f(\lambda)=\Lambda_+(t_0^+) \,, \\
\{0\} & \text{ if } \f(\lambda)>\Lambda_+(t_0^+) \,,
\end{cases}
\quad \text{ with } w_+^* := \frac{a}{\Lambda_+'\circ \Lambda_+^{-1} (\f(\lambda))} \in \Big[\frac{a}{\rho_+},+\infty \Big) \,.
\]
\end{lemma}

\begin{proof}
Recalling~ \eqref{formulasI+}-\eqref{derivI} and~\eqref{I-xI'}, we have that for any $w > \frac{a}{\bar x_+}$
\begin{equation}
\label{derivg+}
\frac{\partial}{\partial w} g_{\lambda,a}^+(w) = \f(\lambda) + \I_+\Big(\frac{a}{w}\Big) - \frac{a}{w} \I_+'\Big(\frac{a}{w}\Big)
= \f(\lambda) - \begin{cases}
\Lambda_+\circ (\Lambda_+')^{-1}\big(\frac{a}{w}\big) & \text{ if } \frac{a}{w}< \rho_+ \,, \\
\Lambda_+(t_0^+) & \text{ if } \frac{a}{w}\geq \rho_+
\,.
\end{cases}
\end{equation}
Let us now consider different cases; we refer to Table~\ref{table:LDP} for an overview, see also~\eqref{I-xI'}.

\medskip
\noindent
\textbullet\ If $\rho_+=+\infty$, then 
\[
\frac{\partial}{\partial w} g_{\lambda,a}^+(w) = \f(\lambda) -\Lambda_+\circ (\Lambda_+')^{-1} \Big( \frac{a}{w} \Big)  \qquad  \text{ for all } w>0\,.
\]

\smallskip
$\ast$ If $t_0^+ =+\infty$, then $\Lambda_+$ and $\Lambda_+'$ are increasing bijections from $(0,+\infty)$ to $(0,+\infty)$. 
From~\eqref{derivg+} we get that $\frac{\partial}{\partial w} g_{\lambda,a}^+(w)>0$ for $\frac{a}{w}< \Lambda_+'\circ \Lambda_+^{-1} (\f(\lambda))$ and $\frac{\partial}{\partial w} g_{\lambda,a}^+<0$ for $\frac{a}{w} > \Lambda_+'\circ \Lambda_+^{-1} (\f(\lambda))$.
Letting 
$w_+^*:= \frac{a}{\Lambda_+'\circ \Lambda_+^{-1} (\f(\lambda))} ,$ 
this shows that $g_{\lambda,a}^+$ is decreasing on $[0,w_+^*]$ and increasing on $[w_+^*,+\infty)$.
We therefore get that $g_{\lambda,a}^+$ has a unique minimum at $w_+^* \in (0,+\infty)$.

Using the formulas~\eqref{formulasI+} for $\I_+(\frac{a}{w_+^*})$ (and the fact that $a/w_+^* <\rho_+ =+\infty$), the minimal value of $g_{\lambda,a}^+$ is then
\begin{equation}
\label{ming+1}
\inf_{w \geq 0} g_{\lambda,a}^+(w) = w_+^* \f(\lambda) + w_+^* \Big( \frac{a}{w_+^*} (\Lambda_+')^{-1} \Big(\frac{a}{w_+^*}\Big) - \Lambda_+\circ (\Lambda_+')^{-1}\Big(\frac{a}{w_+^*}\Big) \Big)  =  a \Lambda_+^{-1}(\f(\lambda)) \,.
\end{equation}
where we have also used the formula for $w_+^*$ together with the fact that $\Lambda_+,\Lambda_+'$ are proper bijections.

\smallskip
$\ast$ If $t_0^+<+\infty$, then $\Lambda_+$ is an increasing bijections from $(0,t_0^+)$ to $(0,x\Lambda_+(t_0^+))$ and $\Lambda_+'$ is an increasing bijection from $(0,t_0^+)$ to $(0,+\infty)$ (recall $\rho_+=+\infty$).
In view of~\eqref{derivg+}, we distinguish two subcases.

If $\f(\lambda)<\Lambda_+(t_0^+)$, then $\Lambda_+\circ (\Lambda_+')^{-1}(x) = \f(\lambda)$ if and only if $x=\Lambda_+'\circ \Lambda_+^{-1} (\f(\lambda))< \Lambda_+'(t_0^+)=+\infty$: we then get that $w_+^*:= \frac{a}{\Lambda_+'\circ \Lambda_+^{-1} (\f(\lambda))} >0$ is the unique minimizer of $g_{\lambda,a}^+$.

If $\f(\lambda)\geq \Lambda_+(t_0^+)$, then $\f(\lambda) > \Lambda_+\circ (\Lambda_+')^{-1}(x)$ for any $x<+\infty$ (since $(\Lambda_+')^{-1}(x)<t_0^+$): this shows that $g_{\lambda,a}^+$ is increasing on $(0,+\infty)$ so it attains its unique minimum at $w_+^* =0$. Note that it also corresponds to the formula $w_+^*:=\frac{a}{\Lambda_+'\circ \Lambda_+^{-1} (\f(\lambda))}$ since $\Lambda_+^{-1}(\f(\lambda)) =t_0^+$ and $\Lambda'_+(t_0^+) = \rho_+=+\infty$.

The minimal value of $g_{\lambda,a}^+$ is then 
\begin{equation}
\label{ming+2}
\inf_{w \geq 0} g_{\lambda,a}^+(w) = 
\left\{
\begin{array}{ll}
 a \Lambda_+^{-1}(\f(\lambda))  &\  \text{ if } \f(\lambda) <\Lambda_+(t_0^+) \\
 \lim\limits_{w\downarrow 0} w \I_+\big(\frac{a}{w} \big)  = a t_0^+ & \  \text{ if } \f(\lambda) \geq \Lambda_+(t_0^+) 
\end{array}
\right\}
=  a \Lambda_+^{-1}(\f(\lambda))  \,.
\end{equation}
Here, we have used in the second line that $\liminf_{x\to \infty} \frac{1}{x} \I_+(x) =t_0^+$, which follows from the fact that  $\frac{1}{\Lambda'_+(y)} \I_+(\Lambda'_+(y)) = y - \frac{\Lambda_+(y)}{\Lambda'_+(y)}$, which goes to $t_0^+$ as $y\uparrow t_0^+$ (indeed, $\frac{\Lambda_+(y)}{\Lambda_+'(y)}$ goes to $0$ since we are only concerned with the case $\Lambda_+(t_0) \leq \f(\lambda)<+\infty$ and $\Lambda_+'(t_0^+) = \rho_+=+\infty$).

\medskip
\noindent
\textbullet\ If $\rho_+<+\infty$, then we again have two subcases.

\smallskip
$\ast$ If $t_0^+=+\infty$. Then $\rho_+ =\bar x_+<+\infty$, and  $g_{\lambda,a}^+(w) = +\infty$ for $w < a/\bar{x}_+$.
On the other hand, $\Lambda_+'$ an increasing bijection from $[0,\bar{x}_+]$ to $[0,\rho_+]$ and $\Lambda_+$ an increasing bijection from $(0,+\infty)$ to $(0,+\infty)$.
From~\eqref{derivg+}, we get as above that $g_{\lambda,a}^+$ is decreasing on $[0,w_+^*]$ and increasing on $[w_+^*,+\infty)$, with $w_+^*:= \frac{a}{\Lambda_+'\circ \Lambda_+^{-1} (\f(\lambda))} \in [\frac{a}{\bar{x}_+},+\infty)$.

As in~\eqref{ming+1}, we have that the minimal value of $g_{\lambda,a}^+$ is
\begin{equation}
\label{ming+3}
\inf_{w \geq 0} g_{\lambda,a}^+(w)  =  a \Lambda_+^{-1}(\f(\lambda)) \,.
\end{equation}

\smallskip
$\ast$ If $t_0^+<+\infty$. Then $\rho_+ <\bar x_+=+\infty$.
Then in view of~\eqref{derivg+}, we distinguish as above in several subcases.

First, if $\f(\lambda)<\Lambda_+(t_0^+)$, then $\Lambda_+\circ (\Lambda_+')^{-1}(x) = \f(\lambda)$ if and only if $x=\Lambda_+'\circ \Lambda_+^{-1} (\f(\lambda))< \Lambda_+'(t_0^+)=+\infty$: we then get that $w_+^*:= \frac{a}{\Lambda_+'\circ \Lambda_+^{-1} (\f(\lambda))} > \frac{a}{\rho_+}$ is the unique minimizer of $g_{\lambda,a}^+$.

Second, if $\f(\lambda) = \Lambda_+(t_0^+)$, then $g_{\lambda,a}^+$ is constant on $[0,w_+^*]$ and then increasing, where $w_+^* = \frac{a}{\rho_+} = \frac{a}{\Lambda_+'\circ \Lambda_+^{-1} (\f(\lambda))} >0$.
In that case the minimum of $g_{\lambda,a}^+$ is attained on the whole interval $[0,\frac{a}{\rho_+}]$.

Third, if $\f(\lambda) > \Lambda_+(t_0^+)$, then $g_{\lambda,a}^+$ is increasing on $(0,+\infty)$ so it attains its unique minimum at $w_+^* =0$. Note that \textit{it does not} corresponds to the formula $w_+^*:= \frac{a}{\Lambda_+'\circ \Lambda_+^{-1} (\f(\lambda))}$ since $\Lambda_+^{-1}(\f(\lambda)) =t_0^+$ and $\Lambda'_+(t_0^+) = \rho_+<+\infty$.

The minimal value of $g_{\lambda,a}^+$ is then
\begin{equation}
\label{ming+4}
\inf_{w \geq 0} g_{\lambda,a}^+(w) = 
\left\{
\begin{array}{ll}
 a \Lambda_+^{-1}(\f(\lambda))  &\  \text{ if } \f(\lambda) <\Lambda_+(t_0^+) \\
 \lim\limits_{w\downarrow 0} w \I_+\big(\frac{a}{w} \big)  = a t_0^+ & \  \text{ if } \f(\lambda) \geq \Lambda_+(t_0^+) 
\end{array}
\right\}
=  a \Lambda_+^{-1}(\f(\lambda))  \,.
\end{equation}
Here, we have used that $ \lim_{x\to \infty} \frac{1}{x} \I_+(x) = \I'_+(\rho_+) =t_0^+$, since $\I_+(x)$ is affine on $[\rho_+,+\infty)$.
\end{proof}

We can now use Lemma~\ref{lem:g+} into~\eqref{rewriting}, to conclude the proof of Lemma~\ref{lem:psi}. Recall the definition~\eqref{def:w*} of $w_+^*$, $w_-^*$.
Now, in the case where $w_+^*+w_-^* \leq 1$, then thanks to Lemma~\ref{lem:g+} (since $w\mapsto g_{\lambda,a}^\pm(w)$ is increasing for $w<w_{\pm}$), we obtain that 
\[
\sup_{0\leq u \leq v \leq 1} g_{\lambda,a}(u,v) = \f(\lambda)- a \Lambda_-^{-1}(\f(\lambda)) - a \Lambda_+^{-1}(\f(\lambda)) \,, 
\]
with the supremum attained on $U^*\times V^*$, with 
\[
U^* = \begin{cases}
\{w_-^*\} & \text{ if } \f(\lambda)<\Lambda_-(t_0^-)\,, \\
[0,\frac{a}{\rho_-}] & \text{ if } \f(\lambda)=\Lambda_-(t_0^-) \,, \\
\{0\} & \text{ if } \f(\lambda)>\Lambda_-(t_0^-) \,,
\end{cases}
\qquad
V^* = \begin{cases}
\{1-w_+^*\} & \text{ if } \f(\lambda)<\Lambda_+(t_0^+)\,, \\
[1-\frac{a}{\rho_+},1] & \text{ if } \f(\lambda)=\Lambda_+(t_0^+) \,, \\
\{1\} & \text{ if } \f(\lambda)>\Lambda_+(t_0^+) \,.
\end{cases}
\]

The last part of Lemma~\ref{lem:psi} tells that the condition $w_+^*+w_-^*\leq 1$ is ensured by having $\lambda \geq \lambda_c(a)$.
This is a consequence of the following lemma.
\begin{lemma}
\label{lem:w+w-}
If $\lambda\geq \lambda_c(a)$, then we have $w_+^*+w_-^*<1$.
\end{lemma}

\begin{proof}
First of all, this is trivial in the case where $t_0^+=t_0^-=0$ since in that case we have $w_+^*=w_-^*=0$ as noticed above.
Also, if $a=\bar a<+\infty$, we have that $\psi(\lambda,a)<0$ for any $\lambda\geq 0$ (see~\eqref{eq:caseabar}), so $\lambda_c(a)=+\infty$ and this case has to be excluded.

Let us now focus on the case where $a<\bar a$ and $t_0^++t_0^->0$, so in particular one of $\I_+$ or $\I_-$ is non-degenerate.
Let us assume that $w_+^*+w_-^* \geq 1$.
Then we show that
\begin{equation}
\label{infimumg+g-}
\inf_{(u,v)\in \mathcal{D}} \big\{ g_{\lambda,a}^-(u) + g_{\lambda,a}^+(1-v) \big\} = \inf_{u\in [0,1]} \big\{ g_{\lambda,a}^-(u) + g_{\lambda,a}^+(1-u) \big\} \,,
\end{equation}
namely the infimum in~\eqref{def:psi} is attained on the diagonal $\{(u,u), u\in [0,1]\}$.
In the end, this shows that 
\[
\psi(\lambda,a) = - \inf_{u\in [0,1]} \Big\{  u \,\I_-\Big(\frac{a}{u}\Big) + (1-u)\, \I_+\Big(\frac{a}{1-u}\Big)\Big\} <0\,,
\]
because at least one of $\I_+,\I_-$ is non-degenerate.
Since $\lambda_c(a) = \sup_{\lambda \geq 0} \{\psi(\lambda,a) <0 \}$ and $\lambda \mapsto \psi(\lambda,a)$ is continuous and non-decreasing, this proves that $\lambda<\lambda_c(a)$, which is a contradiction.

It therefore remains to show that if $w_+^*+w_-^* \geq 1$, then~\eqref{infimumg+g-} holds.
Let $(u,v) \in \mathcal{D}$ with $u<v$. Then necessarily, either $u< w_-^*$ or $1-v> w_+^*$; let us assume $u\neq w_-^*$, the other case is treated analogously.
If $u<w_-^*$, then recalling that $g_{\lambda,a}^-$ is decreasing on $[0,w_+^*]$, we have
\[
g_{\lambda,a}^-(u) + g_{\lambda,a}^+(1-v) >  g_{\lambda,a}^-(u') + g_{\lambda,a}^+(1-v) \,,
\]
with $u'=\min\{w_+^*,1-v\}$.
This shows~\eqref{infimumg+g-} and concludes the proof.
\end{proof}

As a corollary from the proof of Lemma~\ref{lem:psi}, let us extract the following lemma for future use.
It is a direct consequence from the rewriting~\eqref{rewriting} and the computation of the derivative of $g_{\lambda,a}^{\pm}$, see~\eqref{derivg+}.

\begin{lemma}
\label{lem:maximizer}
Assume that $\lambda\geq \lambda_c(a)$ and that
$\f(\lambda)< \min\{\Lambda_+(t_0^+), \Lambda_-(t_0^-)\}$, so that $w_+^*,w_-^*$ verifies $w_+^*> \frac{a}{\rho_+}$, $w_-^*>\frac{a}{\rho_-}$ and $w_+^*+w_-^*<1$. Let $(u^*,v^*) = (w_-^*,1-w+^*) $ be the unique maximizer in~\eqref{def:psi}, see Lemma~\ref{lem:psi}.
Then we have that
\[
 \f(\lambda) + \I_-\Big(\frac{a}{u^*}\Big) - \frac{a}{u^*} \I_-'\Big(\frac{a}{u^*}\Big) =0 
 \quad \text{ and } \quad
 \f(\lambda) + \I_+\Big(\frac{a}{ 1-v^*}\Big) - \frac{a}{1- v^*} \I_+'\Big(\frac{a}{1- v^*}\Big)=0 \,.
\]
\end{lemma}

\subsection{Formula for the free energy and phase transitions}

In this section, we derive a formula for $\psi(\lambda, a)$ and the critical point (Theorem~\ref{th:formuleG}), and we deduce the properties of the phase transitions and the critical curve (Propositions~\ref{th:DL_F}, \ref{prop:saturation} and \ref{prop:curve}).

\subsubsection{Formula for the free energy: proof of Theorem~\ref{th:formuleG}}

The formula~\eqref{eq:formuleG} is a direct corollary of the proof of Lemma~\ref{lem:psi}. 
For $\lambda< \lambda_c(a)$, we have that $\f(\lambda, a)=0$.
On the other hand, for $\lambda\geq \lambda_c(a)$, we have from Lemma~\ref{lem:w+w-} that $w_+^*+w_-^*<1$ so we can apply Lemma~\ref{lem:psi} to get that
\[
\psi(\lambda,a) = \f(\lambda)- a \Lambda_-^{-1}(\f(\lambda)) - a \Lambda_+^{-1}(\f(\lambda)) \,.
\]
This concludes the proof, since $\f(\lambda,a) = \psi(\lambda,a)$ for $\lambda\geq \lambda_c(a)$.
The formula~\eqref{eq:acritique} for $a_c(\lambda)$ directly follows from~\eqref{eq:formuleG}, since the critical line is characterized by $\psi(\lambda,a)=0$.
\qed

%

\subsubsection{Localization phase transition: proof of Proposition~\ref{th:DL_F}}

Let $a >0$ be such that $\lambda_c(a)<+\infty$, \textit{i.e.}\ $a<\bar a$; we have $\textsc{f}(\lambda_c(a),a) = 0$.
Looking at the expression~\eqref{eq:formuleG}, we can compute the derivative of $\f(\lambda,a)$ at a given $\lambda > \lambda_c(a)$.
Recalling that we defined $\lambda_+$ such that $\f(\lambda_+)=\Lambda_+(t_0^+)$, we have
\[
\frac{\partial}{\partial \lambda} \Lambda_+^{-1}(\f(\lambda)) = 
\begin{cases}
 \frac{\f'(\lambda)}{\Lambda_+' \circ \Lambda_-^{-1}(\f(\lambda))} & \text{ if } \lambda <\lambda_+ \,, \\
 0 & \text{ if } \lambda >\lambda_+ \,.
\end{cases}
\]
Therefore, by taking the limit $\lambda \downarrow \lambda_c(a) >0$, $\f(\cdot, a)$ has a right derivative at $\lambda_c(a)$, given by:
\begin{equation}
\label{eq:Ca}
C_a:= \lim_{\lambda\downarrow \lambda_c(a)} \frac{\partial \f(\lambda,a)}{\partial \lambda} = \textsc{f}'(\lambda_c(a)) \bigg(1 - a \Big( \frac{\ind_{\{\lambda_c(a)<\lambda_-\}}}{\Lambda_-' \circ \Lambda_-^{-1}(\textsc{f}(\lambda_c(a)))} + \frac{\ind_{\{\lambda_c(a)<\lambda_+\}}}{\Lambda_+' \circ \Lambda_+^{-1}(\textsc{f}(\lambda_c(a)))} \Big) \bigg)   \,.
\end{equation}
We therefore get, as $u\downarrow 0$,
\begin{align*}
\textsc{f}(\lambda_c(a) + u, a) &= \textsc{f}(\lambda_c(a) , a) +  C_a \, u  + o(u) \sim C_a u  \,.
\end{align*}
This concludes the first part of the proof. 
We deal with the behavior of $C_a$ as $a\downarrow 0$ below (we deduce it from the behavior of $a_c(\lambda)$ as $\lambda \downarrow \lambda_c$). \qed

\subsubsection{Saturation phase transition: proof of Proposition~\ref{prop:saturation}}

Assume that $\lambda_-< \lambda_+<+\infty$ and recall the definition~\eqref{def:fexcess} of the excess free energies $\f^{\ast},\f^{\ast\ast}$. Thanks to the formulas of Theorem~\ref{th:formuleG}, we have that
\[
\f^{\ast}(\lambda,a) = a \big(t_0^- -\Lambda_-^{-1}(\f(\lambda))\big) \,, 
\qquad
\f^{\ast\ast}(\lambda,a) = a \big( t_0^+-\Lambda_+^{-1}(\f(\lambda)) \big) \,.
\]
The statement simply follows once one observes that, by a Taylor expansion, since  $t_0^-=\Lambda_-^{-1}(\f(\lambda_-))$
\begin{equation}
\label{eq:derivt0}
t_0^- - \Lambda_-^{-1}(\f(\lambda_--u)) = (1+o(1)) \frac{\f'(\lambda_-)}{\Lambda_-'(\Lambda_-^{-1}(\f(\lambda_-)))} u =  (1+o(1)) \frac{\f'(\lambda_-)}{\rho_-} u \quad \text{ as } u\downarrow 0\,.
\end{equation}
This gives the behavior of $\f^{\ast}(\lambda_--u,a)$ as $u\downarrow 0$.

When $\lambda_-=\lambda_+$, we have
$\f^{\ast}(\lambda,a) = \f^{\ast\ast}(\lambda,a) = a (t_0^-\Lambda_-^{-1}(\f(\lambda))) +  a (t_0^+ -\Lambda_+^{-1}(\f(\lambda)))$
so the result follows in an identical manner.\qed

\subsubsection{Critical curve: proof of Proposition~\ref{prop:curve}}
There are two parts in the statement.

\smallskip
\noindent
{\it Behavior as $\lambda\downarrow \lambda_c$.}
We use the formula~\eqref{eq:acritique}, together with the fact that under Assumption~\ref{hyp:Cramer} we have $\Lambda_{\pm}(x) \sim \frac12 \sigma^2 x^2$ as $x\downarrow 0$, where $\sigma^2 := \bE[X_1^2]$ is the variance of $X_1$; hence $\Lambda_{\pm}^{-1}(x)\sim \sqrt{ 2 x/\sigma^2}$.
Using that $\f(\lambda)\downarrow 0$ as $\lambda \downarrow \lambda_c$, we get that
\begin{equation}
\label{eq:comportac}
a_c(\lambda) \sim \frac{\sigma}{2\sqrt{2}} \f(\lambda)^{1/2}  \,.
\end{equation}
The conclusion follows by applying~\eqref{eq:criticF5}, using also the definition of $c_3$.

\smallskip
\noindent
{\it Behavior as $\lambda\uparrow \lambda_-,\lambda_+$.}
Assume $\lambda_-<\lambda_+ <\infty$.
Similarly to the proof of Proposition~\ref{prop:saturation}, recalling the definition~\eqref{def:acexcess} of $a_c^{\ast}(\lambda)$ we have from the formulas of Theorem~\ref{th:formuleG} that
\[
a_c^*(\lambda) = \frac{\f(\lambda)}{(\Lambda_+^{-1}(\f(\lambda)) +  \Lambda_-^{-1}(\f(\lambda))) (\Lambda_+^{-1}(\f(\lambda)) + t_0^- )} \big( t_0^- - \Lambda_-^{-1}(\f(\lambda))\big) \,.
\]
The conclusion follows exactly as above, using~\eqref{eq:derivt0}; similarly for $a_c^{\ast\ast}(\lambda)$.
The case where $\lambda_-=\lambda_+<+\infty$ is treated identically.\qed

\subsubsection*{Behavior of the constant $C_a$ in~\eqref{eq:Ca} as $a\downarrow 0$}
Let us now conclude the proof of Proposition~\ref{th:DL_F}, regarding the behavior of the constant $C_a$ as $a\downarrow 0$.

First of all, notice that $\lambda_c(a)\downarrow \lambda_c$ as $a\downarrow 0$, so $\f(\lambda_c(a))\downarrow 0$ as $a\downarrow 0$.
In particular, we get that $\lambda_c(a)<\lambda_{+},\lambda_-$ for $a$ sufficiently small so we will use the formula~\eqref{eq:Ca} without the indicator functions.

From the asymptotic $a_c(\lambda) \sim \frac{\sigma^2}{2\sqrt{2}} \sqrt{c_3} (\lambda-\lambda_c)$ as $\lambda\downarrow \lambda_c$, we get that
$\lambda_c(a) -\lambda_c \sim  \frac{2\sqrt{2}}{\sigma^2 \sqrt{c_3}} a$ as $a\downarrow 0$.
Therefore, using~\eqref{eq:comportac}, we get that 
\[
\f(\lambda_c(a)) \sim \frac{8 a^2}{\sigma^2} \qquad \text{ as } a\downarrow 0 \,.
\]
Now, under Cram\'er's condition, we have that $\Lambda_{\pm}'\circ \Lambda_{\pm}^{-1}(x) \sim \Lambda_{\pm}'(\sqrt{ 2 x/\sigma^2} ) \sim \sqrt{2\sigma^2 x}$ as $x\downarrow 0$,  so that from the behavior of $\f(\lambda_c(a))$ that we just obtained, we get
\[
\lim_{a\downarrow 0} \frac{a}{\Lambda_{\pm}'\circ \Lambda_{\pm}^{-1}(\f(\lambda_c(a))) } = \frac14 \,.
\] 
In view of~\eqref{eq:Ca}, this gives that $C_a \sim \frac12 \f'(\lambda_c(a))$ as $a\downarrow 0$.

Now, using that $\f'(\lambda_c(a))\sim  2 c_3 \sigma^2 (\lambda_c(a)-\lambda_c)$ thanks to~\eqref{eq:criticF5} (and convexity) and the fact that $\lambda_c(a) -\lambda_c \sim  \frac{2\sqrt{2}}{\sigma^2 \sqrt{c_3}} a$ as seen above, we end up with $C_a \sim 2\sqrt{2 c_3} \, a$ as $a\downarrow 0$, which is what is claimed in Proposition~\ref{th:DL_F}.

\section{Proof of Theorems~\ref{th:freeenergy} and~\ref{th:leftrightpoints}}
\label{sec:proofs1}

Recall the definition~\eqref{def:psi} of $\psi(\lambda, a)$.
Note that when $a=0$, we have $\psi(\lambda,0)= \f(\lambda)$ since $\I_{\pm}(0)=0$ (the supremum in~\eqref{def:psi} is attained for $u=0,v=1$), we therefore focus on the case $a>0$.

\subsection{Some preliminary notation}
For an event $A$, we denote
\begin{equation} \label{eq:ZofA}
Z_{N,\lambda}^a (A) = \E \left[ \lambda^{H_N(S)} \un_{\Omega_N^+(S)} \un_{\{S_N = 0\}} \un_A \right] \,, 
\end{equation}
so that $\PP_{N,\lambda}^a(A) = \frac{Z_{N,\lambda}^a(A)}{Z_{N,\lambda}^a}$.
We also define, using the same notation as in~\cite{LT15},
\begin{equation}
\label{def:barcheckZ}
\bar Z_{N,\lambda}^a  := Z_{N,\lambda}^a\big( H_n^a(S)=0 \big)\,, \qquad
\check Z_{N,\lambda}^a := Z_{N,\lambda}^a\big( H_n^a(S) \geq 1 \big)\,,
\end{equation}
so that $Z_{N,\lambda}^a = \bar Z_{N,\lambda}^a+ \check Z_{N,\lambda}^a$.
(The notation is chosen so that the superscripts $\bar{~}$ and $\check{~}$ mimic the shape of random walk trajectories in both cases.)
Let us also define
\begin{equation}
\label{def:Zlr}
\check Z_{\ell,r} := Z_{N,\lambda}^a\big( L_N = \ell, R_N =r\big) \,,
\end{equation}
so in particular $\check Z_{N,\lambda}^a = \sum_{0\leq \ell \leq r \leq N} \check Z_{\ell,r}$.

Let us stress that from Lemma~\ref{lem:depth}, we have that $\check Z_{N,\lambda}^a$ whenever $a>\bar a$. For this reason, we will focus on the case $a\leq \bar a$.

\subsection{Proof of Theorem~\ref{th:freeenergy}}
 
We have that $Z_{N,\lambda}^a =\bar Z_{N,\lambda}^a + \check Z_{N,\lambda}^a$, so in particular $Z_{N,\lambda}^a $
\[
\max \{\bar Z_{N,\lambda}^a , \check Z_{N,\lambda}^a\} \leq Z_{N,\lambda}^a \leq 2 \max \{\bar Z_{N,\lambda}^a , \check Z_{N,\lambda}^a\} \,,
\]
and therefore
\[
\f(\lambda, a) = \max \Big\{ \lim_{N\to\infty} \frac1N \log \bar Z_{N,\lambda}^a,   \lim_{N\to\infty} \frac1N \log \check Z_{N,\lambda}^a\Big\}.
\]

We now prove the following two lemmas.
\begin{lemma}
\label{lem:barZ}
For any $\lambda>0$ and $a>0$, we have
\[
\lim_{N\to\infty} \frac1N \log \bar Z_{N,\lambda}^a =0 \,.
\]
For later purposes, we also give a more precise statement under some further assumption: 
if one has $\lim_{N\to\infty} \frac{1}{N}a_N = 0$, then
\begin{equation}
\label{asympZbar}
\bar Z_{N,\lambda}^a = \bP \Big( \min_{0\leq n\leq N} S_n > -\lfloor aN \rfloor, S_N = 0 \Big) \sim \bP(S_N=0) \sim \frac{f_{\alpha}(0)}{a_N} \qquad \text{ as $N\to\infty$}\,,
\end{equation}
where $f_{\alpha}$ is the density of the limiting $\alpha$-stable law.
\end{lemma}

\begin{lemma}
\label{lem:checkZ}
For any $\lambda > 0$ and $a< \bar a$, we have 
\[
\lim_{N\to\infty} \frac1N \log \check Z_{N,\lambda}^a = \psi(\lambda, a)  \,.
\]
If $a=\bar a <+\infty$, we have $\limsup_{N\to\infty} \frac1N \log \check Z_{N,\lambda}^a \leq  \psi(\lambda, \bar a) <0$.
\end{lemma}

\noindent
These two lemmas readily conclude the proof.
Note in particular that when $\lambda\leq \lambda_c$ or $a=\bar a$, we have $\psi(\lambda, a) \leq 0$, so $\f(\lambda,a)=0$.
\qed

\begin{proof}[Proof of Lemma~\ref{lem:barZ}]
First of all, let us prove a general bound:
\[
1\geq \bar Z_{N,\lambda}^a  \geq \bP( S_1> 0, \ldots, S_{N-1} >0, S_N =0 ) \sim \frac{c_0}{N a_N} \qquad \text{ as $N\to\infty$}.
\]
The lower bound is obvious and the last asymptotic come from Lemma~\ref{lem:kappa}. 
This readily shows that $\lim_{N\to\infty}\frac1N \log \bar Z_{N,\lambda}^a =0$.

As far as the more precise asymptotic~\eqref{asympZbar} is concerned, note that the last part is simply the local central limit theorem, see \emph{e.g.}\ \cite[Ch.~9 \S50]{GK54}.
Let us therefore focus on the first asymptotic, for which  we need to show that
\begin{equation}
\label{eq:LLTnegligible}
\bP(S_N=0  , \min_{0\leq n\leq N} S_n \leq -\lfloor aN \rfloor ) = o(1/a_N) \qquad\text{ as $N\to\infty$}.
\end{equation}
Decomposing according to the instant $T_{a}:= \min\{ n, S_n \leq - \lfloor aN \rfloor\}$, we have that the above probability is equal to 
\[
\bP\big(S_N=0  , T_a\leq N \big) = \bP\big(S_N=0  , T_a\leq N/2 \big) + \bP\big(S_N=0  ,  N/2 < T_a \le N\big) \,.
\]
We treat only the first term; the second one is treated identically.
We have
\[
\begin{split}
\bP\big(S_N=0  , T_a\leq N/2 \big)& = \sum_{k=1}^{N/2} \sum_{x\geq \lfloor aN \rfloor} \bP(T_a=k, S_k=-x) \bP(S_{N-k} =x) \\
&\leq \frac{C}{a_N} \sum_{k=1}^{N/2} \sum_{x\geq \lfloor aN \rfloor} \bP(T_a=k, S_k=-x) = \frac{C}{a_N} \bP\big( \min_{1\leq k\leq N/2}S_k  \leq - \lfloor aN \rfloor \big) \,,
\end{split}
\]
where we have used the local central limit theorem to get that $\bP(S_{n-k} =x)\leq C/a_{N-k} \leq C'/a_N$ uniformly for $k\leq N/2$.
Now, since $(a_N^{-1} S_{\lfloor tN \rfloor})_{t\in [0,1]}$ converges in distribution to an $\alpha$-stable Lévy process (for the Skorokhod topology), we get that 
\[
\lim_{N\to\infty} \bP\Big( \min_{0\leq k\leq N/2} S_k \leq -\lfloor aN \rfloor  \Big) =0 \,, 
\]
provided that $\lim_{N\to\infty} \frac{\lfloor aN \rfloor}{a_N} =+\infty$.
This concludes the proof of~\eqref{eq:LLTnegligible}, since we assumed that $\lim_{N\to\infty} \frac1N a_N = 0$.
\end{proof}

\begin{proof}[Proof of Lemma~\ref{lem:checkZ}]
As far as $\check Z_{N,\lambda}^a$ is concerned, we have
\begin{equation}
\label{eq:checkZsum}
\check Z_{N,\lambda}^a = \sum_{0< \ell \leq r < N} \check Z_{\ell,r} \,,
\end{equation}
and in particular $\max_{0< \ell \leq r < N}\{\check Z_{\ell,r}\} \leq \check Z_{N,\lambda}^a \leq N^2 \max_{0< \ell \leq r < N}\{\check Z_{\ell,r}\}$,
so we focus on $\check Z_{\ell,r}$.

For $n\in \bbN$ and $x\in \bbN$, we define
\begin{equation}
\label{eq:defq}
\begin{split}
Q_+(n,x) &= \bP\big( S_1>0, \ldots, S_{n-1}>0 , S_n=x\big) \,, \\
Q_-(n,x) &= \bP\big( S_1<0, \ldots, S_n<0 , S_{n-1}=-x\big)\,,
\end{split}
\end{equation}
so we can write
\begin{equation}
\label{eq:decompcheckZ}
\check Z_{\ell,r} = Q_-(\ell, \lfloor aN \rfloor) \,Z_{r-\ell,\lambda} \, Q_+(N-r, \lfloor aN \rfloor) \,,
\end{equation}
where $Z_{r-\ell,\lambda}=Z_{r-\ell,\lambda}^0$ is the partition function of the original wetting model (from Section~\ref{sec:wetting}), of length $r-\ell$. 

\smallskip
\noindent
{\it Upper bound.}
As far as the upper bound is concerned, we can use~\eqref{eq:boundLDP} to obtain
\[
Q_+(n,x) \leq \bP( S_n\geq x) \leq  e^{- n \I_+( \frac{x}{n})}\,,
\qquad 
Q_-(n,x) \leq \bP( S_n\geq x) \leq  e^{- n \I_-( \frac{x}{n})} \,.
\]
Together with~\eqref{eq:identityZNlambda}, we therefore get that, for any $0<\ell\leq r <N$,
\begin{equation}
\label{eq:uppercheckZ}
\begin{split}
\check Z_{\ell,r}& \leq  \exp\bigg( - \ell \I_-\Big( \frac{\lfloor aN \rfloor }{\ell} \Big) +  (r-\ell)\f(\lambda) - (N-r) \I_+\Big( \frac{\lfloor aN \rfloor }{N-r}\Big)\bigg) \\
&\leq  \exp\bigg(- N u_{\ell} \I_- \Big(\frac{a}{u_{\ell}}\Big) +  N(v_r-u_\ell)\f(\lambda) - N(1-v_r) \I_+\Big( \frac{a}{1-v_r}\Big) \bigg)  \leq  e^{ N\psi(\lambda, a) }\,,
\end{split}
\end{equation}
where we have set $u_{\ell} = \frac{\ell}{N}$ and $v_{r} = \frac{r}{N}$ and used that $\I_{\pm}$ is non-decreasing, with $\lfloor aN \rfloor \geq aN$.
The last inequality comes from the fact that $u_{\ell}, v_r$ is a subset of $0\leq u\leq v\leq 1$, together with the definition~\eqref{def:psi} of $\psi(\lambda,a)$.
We have therefore proven that
\[
\check Z_{N,\lambda}^a \leq N^2 e^{ N\psi(\lambda, a) }\,,
\]
so $\limsup_{N\to\infty} \frac1N \log \check Z_{N,\lambda}^a \leq \psi(\lambda, a)$.


\smallskip
\noindent
{\it Lower bound.} 
For the lower bound, we only consider the case $a<\bar a$.
Let us fix $(u,v)$ such that $\frac{a}{\bar x_-}< u < v <1 -\frac{a}{\bar x_+}$.
We define $\ell_N:= \lfloor u N \rfloor$ and $r_N:= \lfloor v N\rfloor$. Since $ \check Z_{N,\lambda}^a \geq \check Z_{\ell_N, r_N}$, we simply need to show that 
\begin{equation*}
\label{eq:liminfcheckZ}
\liminf_{N\to\infty} \frac1N \log \check Z_{\ell_N, r_N} \geq g_{\lambda,a}(u,v) :=  (v-u) \f(\lambda) - u \,\I_-\Big(\frac{a}{u}\Big) - (1-v)\, \I_+\Big(\frac{a}{1-v}\Big)\,,
\end{equation*}
to ensure that 
\[
\liminf_{N\to\infty} \frac1N \log \check Z_{N,\lambda}^a \geq \sup_{\frac{a}{\bar x_-}< u < v <1 -\frac{a}{\bar x_+}} g_{\lambda,a}(u,v) = \psi(\lambda, a) \,, 
\]
which concludes the lower bound.

Using the same decomposition as in~\eqref{eq:decompcheckZ} and the definition~\eqref{def:freeenergy} of the free energy of the standard wetting model, we therefore simply have to prove that
\begin{equation}
\label{eq:Q-LDP}
\begin{split}
\liminf_{N\to\infty} \frac1N \log Q_-(\ell_N, \lfloor aN \rfloor) &\geq -  u\I_-\Big( \frac{a}{u} \Big) \,,\\
\liminf_{N\to\infty} \frac1N \log Q_+(N-r_N, \lfloor aN \rfloor) & \geq -  (1-v)\I_+\Big( \frac{a}{1-v} \Big) \,.
\end{split}
\end{equation}

We therefore rely on the following lemma, whose proof is similar to that of Lemma~\ref{lem:LDPequal} and postponed to Appendix~\ref{sec:LDPequal}.
\begin{lemma}
\label{lem:Q-LDP}
Suppose that Assumption~\ref{hyp:1} holds.
For any $0\leq x < \bar x_+$, for any sequence of integers $(x_n)_{n\geq 0}$ such that $\lim_{n\to\infty} \frac1n x_n =x$, we have
\[
\liminf_{n\to\infty} \frac1n \log Q_+(n, x_n) \geq - \I_+(x) \,.
\]
A similar statement holds for $Q_-$.
\end{lemma}

This lemma readily proves~\eqref{eq:Q-LDP}, since $\frac{a}{u}<\bar x_-$ and $\frac{a}{1-v}<\bar x_+$.
This concludes the proof of  Lemma~\ref{lem:checkZ}.
\end{proof}

\subsection{Proof of Theorem~\ref{th:leftrightpoints}}

First of all, let us note that in the supercritical case $\lambda > \lambda_c(a)$, we have $\psi(\lambda, a)>0$. Hence, thanks to Lemmas~\ref{lem:barZ}-\ref{lem:checkZ} we have $\lim_{N\to\infty} \bar Z_{N,\lambda}^a/\check Z_{N,\lambda}^a =0 $.
We therefore obtain that
\[
\lim_{N\to\infty} \bP_{N,\lambda}^a \big(H_N^a(S) >0 \big) =
\lim_{N\to\infty} \frac{\check Z_{N,\lambda}^a}{\bar Z_{N,\lambda}^a+\check Z_{N,\lambda}^a} =1 \,.
\]

Let us now show something slightly more general than needed: for any $\lambda> 0$ and $a<\bar a$, for any fixed $\gep >0$ we have
\begin{equation}
\label{eq:lemleftright}
\lim_{N\to\infty}\bP_{N,\lambda}^a\bigg( \mathrm{dist}\Big( \tfrac1N ( L_N, R_N ), \argmax \psi \Big) >\gep  \;\Big|\; H_N^a(S) >0 \bigg) =0 \,,
\end{equation}
where we used the shorthand notation $\argmax \psi = \{ (u,v) \in \mathcal{D}, g_{\lambda,a}(u,v) = \psi(\lambda,a)\}$ for the set of maximizers of $\psi (\lambda,a)$ in~\eqref{def:psi}; note that the supremum is attained, see Lemma~\ref{lem:psi} if $w_+^*+w_-^*\leq 1$ or~\eqref{infimumg+g-} if $w_+^*+w_-^*\geq 1$ (recall that $g_{\lambda,a}^-,g_{\lambda,a}^+$ are continuous).
Together with Lemma~\ref{lem:psi} (and Lemma~\ref{lem:w+w-}), this concludes the proof of~\eqref{eq:leftright} in the supercritical case $\lambda>\lambda_c(a)$.

The proof of~\eqref{eq:lemleftright} is easy, since the probability is the ratio
\[
\frac{1}{\check Z_{N,\lambda}^a}Z_{N,\lambda}^a\bigg( \mathrm{dist}\Big( \tfrac1N ( L_N, R_N ), \argmin \psi \Big) >\gep  \bigg) = \frac{1}{\check Z_{N,\lambda}^a}\sum_{\mathrm{dist}( ( \frac{\ell}{N}, \frac{r}{N} ), \argmin \psi ) >\gep} \check Z_{\ell, r} \,.
\]
Using the same upper bound as in~\eqref{eq:uppercheckZ}, we get that the sum is bounded by $ N^2 e^{N \psi_{\gep}(\lambda, a)}$,
with 
\[
\psi_{\gep}(\lambda, a) := \suptwo{(u,v)\in \mathcal{D}}{\mathrm{dist}( (u,v), \argmin \psi ) >\gep}  g_{\lambda,a}(u,v) \,.
\]
Together with to Lemma~\ref{lem:checkZ}, we get that
\[
\limsup_{N\to\infty} \frac{1}{N} \log \bP_{N,\lambda}^a\bigg(  \mathrm{dist}\Big( \tfrac1N ( L_N, R_N ), \argmin \psi\Big) >\gep \;\Big|\; H_N^a(S) >0 \bigg)
\leq \psi_{\gep}(\lambda, a) - \psi(\lambda, a) <0 \,,
\]
where we have used the continuity of $(u,v) \mapsto g_{\lambda,a}(u,v)$ to obtain the strict inequality at the end.
This concludes the proof of~\eqref{eq:lemleftright} and of \eqref{eq:leftright} in Theorem~\ref{th:leftrightpoints}.

\smallskip
As far as the subcritical case $\lambda < \lambda_c(a)$ is concerned, let us distinguish two possibilities.
First, if either one of $\I_+,\I_-$ is not identically equal to~$0$. Then $\lambda \mapsto\psi(\lambda, a)$ is strictly negative on $[0,\lambda_c(0)]$ where $\f(\lambda)=0$ and then strictly increasing: we therefore get that $\psi(\lambda, a)<0$ for any $\lambda <\lambda_c(a)$ since $\psi(\lambda_c(a), a)=0$ by continuity. 
From Lemmas~\ref{lem:barZ}-\ref{lem:checkZ}, we get that
$\lim_{N\to\infty} \check Z_{N,\lambda}^a / \bar Z_{N,\lambda}^a =0$, which readily implies that
\[
\lim_{N\to\infty} \bP_{N,\lambda}^a \big(H_N^a(S)=0 \big) =
\lim_{N\to\infty} \frac{\bar Z_{N,\lambda}^a}{\bar Z_{N,\lambda}^a+\check Z_{N,\lambda}^a} =1 \,.
\]

Second, if we have that $\I_+,\I_-$ are both identically equal to $0$. 
Then we have $\psi(\lambda, a)= \f(\lambda)$ for any $\lambda >0$, hence $\lambda_c(a)=\lambda_c$ for any $a>0$.
We show below that for any $\lambda<\lambda_c =\lambda_c(a)$, there is a constant $C>0$ such that
\begin{equation}
\label{eq:checkZbound}
\check Z_{N,\lambda}^a \leq C  \bP\big( S_N=0, \min_{1\leq k \leq N} S_k \leq -\lfloor aN \rfloor \big) \,,
\end{equation}
so we get that $\check Z_{N,\lambda}^a = o(1/a_N)$ thanks to~\eqref{eq:LLTnegligible}. 
Recalling Lemma~\ref{lem:barZ}-\eqref{asympZbar}, we get that $\lim_{N\to\infty} \check Z_{N,\lambda}^a/\bar Z_{N,\lambda}^a =0$, which allows us to conclude as above.

To prove~\eqref{eq:checkZbound}, we use that for any $\lambda<\lambda_c$, there is a constant $C_{\lambda}$ such that $Z_{n,\lambda}\sim C_{\lambda} K(n)$ as $n\to\infty$, see \cite[Thm.~2.2]{Giac07}, where $K(n)= \kappa^{-1} f_n^+(0)$ is defined in~\eqref{def:K1}.
Therefore, using also Lemma~\ref{lem:kappa} and the local central limit theorem to get that $f_n^+(0) \leq C/a_n \leq  \bP(S_n=0)$, 
we get that $Z_{n,\lambda} \leq C' \bP(S_n=0)$ for all $n\geq 0$, for some constant $C'$.
All together, we have that $\check Z_{N,\lambda}^a$ is bounded by a constant times
\[
 \sum_{1\leq \ell\leq r\leq N} Q_-(\ell , \lfloor a N\rfloor) \bP(S_{r-\ell} =0) Q_+(n-r,\lfloor a N\rfloor) 
\leq  \bP\big( S_N=0, \min_{1\leq k \leq N} S_k \leq -\lfloor aN \rfloor \big) \,,
\]
which is exactly what is claimed in~\eqref{eq:checkZbound}.
\qed

\section{Sharp asymptotic behavior of the partition function, fluctuations of $L_N$, $R_N$}
\label{sec:proofs2}

In this section, we prove Theorem~\ref{th:Gaussian}, \textit{i.e.}\ we provide, in Cram\'er's region, an exact asymptotic behavior of $\check Z_{N,\lambda}$ and a local central limit theorem for $L_N,R_N$.

\subsection{Local large deviation with a positivity constraint}

Let us start with the proof of a technical estimate which is key in our proof.
It is a precise estimate of $Q_{\pm}(n,x_n)$ inside Cram\'er's region; recall that $Q_{\pm}(n,x_n)$ is defined in~\eqref{eq:defq}.

\begin{proposition}
\label{prop:Qnx}
Let $x\in (0,\rho_+)$ and let $(\delta_n)_{n\geq 0}$ be a sequence such that $\lim_{n\to\infty}\delta_n =0$.
Then, there exists a constant $p_x \in (0,1)$ such that, uniformly for non-negative sequences $(x_n)_{n\geq 0}$ with $|\frac1n x_n - x|\leq \delta_n$, we have
\[
Q_+(n,x_n) \sim \frac{p_x \sqrt{\I_+''(x)}}{\sqrt{2\pi  n } } \, e^{ - n \I_+\big(\frac{x_n}{n}\big) }\,, \quad \text{ as } n\to\infty \,.
\]
A similar statement holds for $Q_-(n,x_n)$.
\end{proposition}

This result is a version of a local large deviation estimate with the additional positivity constraint.
Our proof is quite standard, using a change of measure argument, but we also have to deal with the positivity constraint, which makes things slightly more technical.

\begin{proof}
For $t \in [0,t_0^+)$, we define the tilted law of $X_1$:
\begin{equation}
\label{def:tiltedPt}
\frac{\dd \bP_t}{\dd \bP} (x) = \frac{e^{tx}}{\bE[e^{tX_1}]}  = e^{tx-\Lambda_+(x)} \,.
\end{equation}
With a slight abuse of notation, we will also write $\bP_t$ for the law of i.i.d.\ copies.
Let us denote, for $t<t_0^+$,
\begin{equation}
m(t) := \bE_t[X_1] = \Lambda'_+(t)\,,
\qquad
\sigma^2(t) := \mathbf{Var}_t(X_1) = \Lambda''(t)\,.
\end{equation}
Notice that $m=\Lambda'$ is strictly increasing and continuous from $[0,t_0^+)$ to $[0,\rho_+)$, so we may define $t_{x} := m^{-1}(x)$ for any $x\in [0,\rho_+)$, in such a way that $\bE_{t_x}(X_1)=x$.
Note that we have $\Lambda''(t_x) = \I_+(x)$, see Section~\ref{sec:LD}; we also stress that for any $x \in [0,\rho_+)$, the supremum in $\sup_{t\geq 0} \{tx - \Lambda_+(t)\}$ is attained at $t=t_x$, so $\I_+(x)= xt_x - \Lambda_+(t_x)$.

For simplicity, let us denote $t_n = t_{x_n/n}$ and $\sigma_n:=\sigma^2(t_n)$ so in particular $x_nt_n - n\Lambda_{+}(t_n) = \I_+(\frac{x_n}{n})$.
Let also
\begin{equation}
\label{def:Ak}
A_k^+ :=\{S_1>0, \ldots, S_k>0\} \qquad \text{ for } k\geq 1\,.
\end{equation}
Then, using the definition~\eqref{def:tiltedPt} of $\bP_t$, we may rewrite
\begin{align*}
Q_+(n,x_n) := \bP\big(  A_n^+ , S_n=x_n  \big)
  &= \bE_{t_n} \Big[ e^{n \Lambda_+(t_n)- t_n S_n } \ind_{\{ A_n^+ , S_n=x_n \}}\Big] \\
  &= e^{-n \I_+(\frac{x_n}{n})} \bP_{t_n}\big(A_n^+ , S_n=x_n  \big) \,.
\end{align*}

We now let $(k_n)_{n\geq 1}$ be a sequence of integers such that $\lim_{n\to\infty} k_n =+\infty$ and  $\lim_{n\to\infty}\frac1n k_n =0$.
We also take $k_n$ such that $\lim_{n\to\infty} k_n \delta_n^2 =0$, which ensures in particular that $\lim_{n\to\infty} k_n (t_n-t_x)^2 =0$ (as will be used below).
Then, we need to control
\begin{equation}
\label{eq:decompQnx}
\bP_{t_n}\big(A_n^+ , S_n=x_n  \big) = \sum_{y\geq 0} \bP_{t_n}(A_{k_n}^+, S_{k_n}=y) \bP_{t_n}\big( S_{n-k_n} =x_n-y , S_1>-y , \ldots, S_{n-k_n} >-y \big) \,.
\end{equation}
We split our proof in an upper and a lower bound.

\smallskip
\noindent
\textit{Upper bound on~\eqref{eq:decompQnx}. }
As an upper bound, we remove the condition $S_i>-y$ for $1\leq i\leq n-k_n$ in the last probability, to obtain the following:
\[
 \bP_{t_n}\big(A_n^+ , S_n=x_n  \big) \leq \sum_{y\geq 0} \bP_{t_n}(A_{k_n}^+, S_{k_n}=y) \bP_{t_n}\big( S_{n-k_n} =x_n-y \big)  \,.
\]
At this point, we need a local central limit theorem for a family of probability distributions: we refer to Theorem~8.7.1A in \cite{Bor13}, that we now state for completeness.

\begin{theorem}[Thm.~8.7.3A, \cite{Bor13}] 
\label{th:limloc}
Let $(\bP_{t})_{t \in [0,t_0]}$ be a family of distributions indexed by a parameter $t\in [0,t_0]$; denote also $\bP_t$ the law of i.i.d.\ random variables $(X_i)_{i\geq 1}$ with law $\bP_t$.
We assume that for any $t\in [0,t_0]$, under $\bP_t$ the $X_i$'s are $\mathbb Z$-valued and $S_n=\sum_{i=1}^n X_i$ is aperiodic.
Let $m(t):= \bE_t(X_1)$ and $\sigma^2(t):= \mathbf{V}\mathrm{ar}_t(X_1)$.
Assume that
\[
 0< \inf_{t\in [0,t_0]} \sigma^2(t)\leq \sup_{t\in [0,t_0]} \sigma^2(t) <+\infty\,,
 \quad \text{ and } \quad  \sup_{t\in [0,t_0]}  \sup_{\theta \in [\gep, \frac\pi2]} |\bE_t[e^{i\theta X}]| <1 \text{ for any }\gep>0 \,.
\]
(The latter corresponds to having some ``uniform aperiodicity''.)
Then we have a local limit theorem which is \emph{uniform} in $t$:
\[
\sup_{t\in [0,t_0]}\sup_{x\in \mathbb Z} \bigg| \sqrt{n} \bP_{t}(S_n - n m(t) = x) - g_{\sigma(t)} \Big( \frac{x}{\sqrt{n}}\Big)  \bigg|  \xrightarrow{n\to\infty} 0 \,,
\]
where $g_{\sigma}(x) = \frac{1}{\sigma\sqrt{2\pi}} e^{-\frac{x^2}{2\sigma^2}}$ is the density of the Gaussian $\mathcal{N}(0,\sigma^2)$.
\end{theorem}

We may apply this theorem for parameters $t \in [\gep, t_0^+-\gep]$ for some $\gep>0$ fixed, since then we will have that $\sigma^2(t) \in [\sigma^2(\gep), \sigma^2(t_0^+-\gep)]$ with $\sigma^2(\gep)>0$ and $\sigma^2(t_0^+-\gep)<+\infty$; the second condition on the uniform aperiodicity follows similarly.  
Since $m(t_n) = \frac1n x_n$ and $\lim_{n\to\infty} t_n=t_x \in (0,t_0^+)$, we can apply this result to obtain that, uniformly on $y\in \mathbb Z$,
\[
\bP_{t_n}\big( S_{n-k_n} = x_n-y \big)  \leq \frac{1}{\sqrt{2\pi (n-k_n)} \sigma_n^2} e^{ -\frac{(k_n m(t_n) -y)^2}{(n-k_n)^2 \sigma_n^2}} + o\Big(\frac{1}{\sqrt{n}} \Big) \leq \frac{1+o(1)}{ \sqrt{2\pi n} \sigma(x)}\,,
\]
where the $o(1)$ is uniform in $|\frac1n x_n -x|\leq \delta_n$.

We therefore obtain that 
\[
\bP_{t_n}\big(A_n^+ , S_n=x_n  \big) \leq \frac{1+o(1)}{ \sqrt{2\pi n} \sigma(x)} \, \bP_{t_n} \big( A_{k_n}^+ \big) \,.
\]
We now use the following lemma, which concludes the upper bound; we postpone its proof to the end of the section.
\begin{lemma}
 \label{lem_prob}
Let $(t_n)_{n\geq 0}$ be a sequence such that $\lim_{n\to\infty} t_n =t_x \in (0,t_0^+)$.
Then, for any sequence $(k_n)_{n\geq 1}$ such that $\lim_{n\to\infty} k_n =+\infty$ and $\lim_{n\to\infty}  k_n (t_n-t_x)^2 =0$, we have
\begin{equation}
      \lim_{n\to\infty}\bP_{t_n}\left(S_1>0, \ldots, S_{k_n}>0\right) = \bP_{t_x}( S_i>0 \ \text{ for all }i\geq 1) =: p_x \in (0,1) \,.
\end{equation}
The  rate of convergence depends only on $k_n(t_n-t_x)^2$.
\end{lemma}

\smallskip
\noindent
\textit{Lower bound on~\eqref{eq:decompQnx}.}
For a lower bound, we restrict the sum to $y \in [\frac12 m(t_n) k_n, \frac32 m(t_n) k_n]$. Let us introduce, for $y\geq 0$,
\begin{equation}
\label{eq:Bn}
B_n(y) := \bP_{t_n}\big( S_{n-k_n} =x_n-y , S_1>-y , \ldots, S_{n-k_n} >-y \big) \,,
\end{equation}
so we obtain the lower bound
\[
\bP_{t_n}\big(A_n^+ , S_n=x_n  \big) \geq  \bP\Big(A_{k_n}^+, S_{k_n} \in [\tfrac12 m(t_n) k_n, \tfrac32 m(t_n) k_n] \Big) \inf_{y\in  [\frac12 m(t_n) k_n, \frac32 m(t_n) k_n]} B_n(y)\,.
\]
For the first probability, we write
\[
\bP\Big(A_{k_n}^+, S_n \in [\tfrac12 m(t_n) k_n, \tfrac32 m(t_n) k_n] \Big) \geq  \bP\big( A_{k_n}^+ \big) - \bP\Big(S_{k_n} \notin [\tfrac12 m(t_n) k_n, \tfrac32 m(t_n) k_n] \Big) \,.
\]
The first term converges to $p_x$ thanks to Lemma~\ref{lem_prob}, and it remains to see that the second one goes to $0$.
But since $\bE_{t_n}[X_i]= \frac1n x_n = m(t_n)$, we get by Chebyshev's inequality that 
\[
\bP\Big(S_{k_n} \notin [\tfrac12 m(t_n) k_n, \tfrac32 m(t_n) k_n] \Big) \leq \frac{k_n \sigma^2_n}{ \frac14m(t_n) k_n^2} \, \xrightarrow{n\to\infty} 0 
\]
since $\lim_{n\to\infty} k_n =+\infty$,
using also that $\lim_{n\to\infty} \sigma_n^2= \sigma^2(x)$ and $\lim_{n\to\infty} m(t_n) = x >0$.

It therefore remains to show that 
\begin{equation}
\inf_{y\in  [\frac12 m(t_n) k_n, \frac32 m(t_n) k_n]}B_n(y) \geq  \frac{1+o(1)}{\sqrt{2\pi n} \sigma(x)} \,.
\end{equation}
Recalling the definition~\eqref{eq:Bn}, we write
$B_n(y)\geq B_n^{(1)}(y) - B_n^{(2)}(y)$, with
\[
B_n^{(1)}(y) := \bP_{t_n}\big( S_{n-k_n} =x_n-y \big) \,,
\quad
 B_n^{(2)}(y):= \bP_{t_n}\big( S_{n-k_n} =x_n-y , \min_{1\leq i\leq n-k_n} S_i \leq -\tfrac12 m(t_n) k_n\big) \,.
\]
The first term is again controlled thanks to Theorem~\ref{th:limloc}: we have that 
\[
B_n^{(1)}(y) \geq \frac{1}{\sqrt{2\pi (n-k_n)} \sigma_n^2} e^{ -\frac{(k_n m(t_n) -y)^2}{(n-k_n)^2 \sigma_n^2}} - o\Big(\frac{1}{\sqrt{n}} \Big) \geq  \frac{1+o(1)}{ \sqrt{2\pi n} \sigma^2(x)} \,,
\]
where all inequalities are uniform in $y \in [\frac12 m(t_n) k_n, \frac32 m(t_n) k_n]$.
It only remains to show that $B_n^{(2)}(y)=o(1/\sqrt{n})$ uniformly in $y \in [\frac12 m(t_n) k_n, \frac32 m(t_n) k_n]$.
Let $T_{k_n} := \min\{i\geq 1, S_i \leq -\frac{1}{2} m(t_n) k_n \}$, so we can write that
\[
B_n^{(2)}(y) =\sum_{j=1}^{n-k_n} \sum_{z\leq - \frac12 m(t_n) k_n} \bP_{t_n}\big(T_{k_n} = j, S_{j} =z \big)  \bP_{t_n}\big( S_{n-k_n-j} = x_n-y-z \big)  \,.
\]
We show just below that there is a constant $C>0$ such that 
\begin{equation}
\label{eq:cestlafin}
\sup_{1\leq j\leq n-k_n} \sup_{w\leq m(t_n) k_n} \bP_{t_n}\big( S_{n-k_n-j} = x_n - w \big)  \leq \frac{C}{\sqrt{n}} \,.
\end{equation}
From the above display, we therefore deduce that
\[
\sup_{y \in [\frac12 m(t_n) k_n, \frac32 m(t_n) k_n]} B_n^{(2)}(y) \leq \frac{C}{\sqrt{n}} \bP_{t_n}\big( T_{k_n} \leq n-k_n \big) = \frac{C}{\sqrt{n}} \bP_{t_n}\big( \min_{1\leq i \leq n-k_n}   S_i \leq -\tfrac12m(t_n) k_n \big)  \,,
\]
with the last probability going to $0$ as $n\to\infty$, since $\bE_{t_n}[X_i] =  m(t_n) \geq \frac12 x >0$ (at least for $n$ sufficiently large).

We are left with showing~\eqref{eq:cestlafin}.
First of all, if $j\leq n/2-k_n$ (so $n-k_n-j \geq n/2$), the local limit theorem from Theorem~\ref{th:limloc} gives that
\[
\bP_{t_n}\big( S_{n-k_n-j} = w' \big) \leq \frac{C}{\sqrt{n}} \,,
\]
uniformly in $w' \in \mathbb Z$.
If on the other hand $j\leq n/2-k_n$, we bound
\[
\sup_{k\leq n/2} \sup_{w\leq m(t_n) k_n} \bP_{t_n}\big( S_{k} = x_n -w \big)
\leq \sup_{k\geq n/2} \bP_{t_n}\big( S_{k} \geq  x_n - m(t_n) k_n \big) \leq \sup_{k\leq n/2} \frac{k \sigma_n^2 }{(x_n - m(t_n) (k_n+k))^2} \,,
\]
where we have used Chebyshev's inequality for the last bound.
Since we have $m(t_n) k \leq x_n/2$, $\lim_{n\to\infty} m(t_n) k_n =0$, $\lim_{n\to\infty} \frac1n x_n =x>0$ and also $\lim_{n\to\infty} \sigma_n^2= \sigma^2(x) <+\infty$, we get that  this is bounded by $\frac{4\sigma^2(x)}{n x}$ for $n$ large enough. 
This concludes the proof of~\eqref{eq:cestlafin}, hence of the upper bound.
\end{proof}

\begin{proof}[Proof of Lemma~\ref{lem_prob}]
Recall that we denoted $A_{k_n}^+:=\{S_1>0, \ldots, S_{k_n}>0 \}$.
Recalling also the definition~\eqref{def:tiltedPt} of $\bP_{t_n}$, we have
\begin{align*}
\bP_{t_n}\big(A_{k_n}^+\big)  &= \bE_{t_x} \Big[\ind_{A_{k_n}^+} e^{(t_n - t_x) S_{k_n}} \Big] \, e^{ (\Lambda(t_x)- \Lambda(t_n))k_n}.
\end{align*}
Letting $\gep>0$ be fixed, we can apply H\"older inequality, to get 
\begin{equation*}
\begin{split}
\bP_{t_n}\big(A_{k_n}^+\big) & \leq \bP_{t_x} \big(A_{k_n}^+\big)^{1-\gep}  \bE_{t_x} \Big[ e^{ \gep^{-1}(t_n - t_x) S_{k_n}} \Big]^{\gep} e^{(\Lambda(t_x)-\Lambda(t_n))k_n} \\
& \leq \bP_{t_x} \big(A_{k_n}^+\big)^{1-\gep} \exp\Big( k_n \Big( \gep \big( \Lambda\big(t_x + \gep^{-1}(t_n-t_x) \big) - \Lambda(t_x)  \big) - \big(\Lambda(t_n) - \Lambda(t_x) \big)\Big) \Big)\,.
\end{split}
\end{equation*}
Now, by a Taylor expansion in $t_n-t_x$, we have that
\[
\gep \big( \Lambda\big(t_x + \gep^{-1}(t_n-t_x) \big) - \Lambda(t_x)  \big)+\Lambda(t_x) - \Lambda(t_n) = (1+o(1)) \frac12 (\gep^{-1}-1) \Lambda''(t_x) (t_n-t_x)^2 \,.
\]
Hence, since $\lim_{n\to\infty} k_n(t_n-t_x)^2 =0$, we get that
\[
\limsup_{n\to\infty} \bP_{t_n}\big(A_{k_n}^+\big)  \leq \limsup_{n\to\infty} \bP_{t_x} \big(A_{k_n}^+\big)^{1-\gep} = (p_x)^{1-\gep}\,,
\]
with $p_x:= \bP_{t_x}(S_i>0 \text{ for all }i\geq 1) \in (0,1)$ (note that $p_x>0$ since $\bE_{t_x}[X_1]=x>0$).
We stress that the rate of convergence depends on $k_n(t_n-t_x)^2$.

On the other hand, we also have a lower bound on $\bP_{t_n}(A_{k_n}^+)$ using H\"older's inequality the other way around:
\begin{align*}
\bP_{t_x} \big(A_{k_n}^+\big) 
& = \bE_{t_n} \Big[\ind_{A_{k_n}^+} e^{(t_x-t_n) S_{k_n}} \Big] \, e^{ (\Lambda(t_n)- \Lambda(t_x))k_n} \\
& \leq  \bP_{t_n} \big(A_{k_n}^+\big)^{1-\gep}  \bE_{t_n} \Big[ e^{ \gep^{-1}(t_x - t_n) S_{k_n}} \Big]^{\gep} e^{(\Lambda(t_n)-\Lambda(t_x))k_n} \,. 
\end{align*}
Similarly as above, we get that
\[
p_x = \lim_{n\to\infty} \bP_{t_x} \big(A_{k_n}^+\big) \leq \liminf_{n\to\infty} \bP_{t_n} \big(A_{k_n}^+\big)^{1-\gep} \,.
\]
Since $\gep>0$ is arbitrary, this concludes the proof.
\end{proof}

\subsection{Asymptotic behavior of $\check Z_{\ell,r}$}

From Proposition~\ref{prop:Qnx}, we are able to obtain the following result on the partition function $\check Z_{\ell,r} := Z_{N,\lambda}^{a} (L_N=\ell, R_N =r)$.

\begin{proposition}
\label{prop:checkZlr}
Let $\lambda>\lambda_c$, let $\frac{a}{\rho_+}<u < v< 1- \frac{a}{\rho_+}$ and let $(\delta_N)_{N\geq 0}$ be a vanishing sequence.
Then uniformly for $(\ell, r)$ with and $|\frac1N \ell -u| \leq \delta_N$ and $|\frac1N r -v| \leq \delta_N$, we have
\[
\check Z_{\ell,r} \sim \frac{p_{x_u} p_{x_{1-v}} \sqrt{\I_-''(x_u) \I_+''(x_{1-v})}}{2\pi N m_{\lambda}\sqrt{u(1-v)}}
\exp\bigg( N \Big( \big(\tfrac{r}{N}-\tfrac{\ell}{N}\big) \f(\lambda) - \tfrac{\ell}{N} \I_-\big( \tfrac{\lfloor aN \rfloor}{\ell} \big)- \big(1-\tfrac{r}{N}\big) \I_-\big( \tfrac{\lfloor aN \rfloor}{N-r} \big) \Big)  \bigg) \,,
\]
with $x_s = a/s$ and $m_{\lambda}$ from~\eqref{eq:asympZ}.
\end{proposition}

\begin{proof}
We simply need to use~\eqref{eq:decompcheckZ}, that is:
\[
\check Z_{\ell,r} =  Q_{-}(\ell, \lfloor aN \rfloor) Z_{r-\ell, \lambda} Q_+(N-r, \lfloor a N\rfloor) \,.
\]
Then, using Proposition~\ref{prop:Qnx} we get that
\[
\begin{split}
Q_{-}(\ell, \lfloor aN \rfloor)& \sim  \frac{p_{x_u}  \sqrt{\I_-''(x_u) }}{\sqrt{2\pi N u}}  \exp\Big( - \ell \I_-\big( \tfrac{\lfloor aN \rfloor}{\ell} \big) \Big) \,,\\ 
Q_{+}(N-r, \lfloor aN \rfloor) & \sim  \frac{p_{x_{1-v}}  \sqrt{\I_+''(x_{1-v}) }}{\sqrt{2\pi N (1-v)}}  \exp\Big( - (N-r) \I_+\big( \tfrac{\lfloor aN \rfloor}{N-r} \big) \Big) \,.
\end{split}
\]
Combined with the asymptotic~\eqref{eq:asympZ} for $Z_{r-\ell,\lambda}$, this concludes the proof
\end{proof}

\subsection{Proof Theorem~\ref{th:Gaussian}, part (i)}

We prove part (i) of Theorem~\ref{th:Gaussian}, namely the sharp asymptotic~\eqref{eq:asympZN} for $Z_{N,\lambda}^a$.
Let $\lambda\geq \lambda_c(a)>\lambda_c$.
Recall that we assume that Cram\'er's condition holds, so in particular $X_i$ have a finite variance: we can choose the normalizing sequence $a_N = \sigma \sqrt{N}$, and we get that $\bar Z_{N,\lambda}^a \sim \frac{1}{\sigma\sqrt{2\pi N}}$ from Lemma~\ref{lem:barZ}.

We can therefore focus on $\check Z_{N,\lambda}^a$.
Moreover, from~\eqref{eq:lemleftright}, we get that there is a sequence $(\gep_n)_{n\geq 0}$ with $\lim_{n\to\infty} \gep_n =0$, such that we have
\[
\lim_{n\to\infty} \frac{1}{\check Z_{N,\lambda}^a}  Z_{N,\lambda}^a\Big( \Big| \frac1N ( L_N, R_N )- (u^*,v^*)\Big| <\gep_n  \Big) =1 \,,
\]
where we recall that $(u^*,v^*)$ is (in Cram\'er's region) the unique maximizer of $\psi(\lambda,a)$ in~\eqref{def:psi}, see~\eqref{eq:formulauv*}.
In other words, we have that
\[
\check Z_{N,\lambda}^a \sim \sum_{ |\frac1N\ell-u^*| \leq \gep_n} \sum_{|\frac1N r -v^*| \leq \gep_n}  \check Z_{\ell,r} \,.
\]
Therefore, from Proposition~\ref{prop:checkZlr} and since $\frac{a}{\rho_+}<u^*<v^*<1-\frac{a}{\rho_+}$ (see~\eqref{eq:formulauv*}), we get that there is some explicit constant $C_{\lambda, a}$, such that
\begin{equation}
\label{eq:decomp}
\check Z_{N,\lambda}^a \sim  \frac{C_{\lambda,a}}{2\pi N} \sum_{ |\frac1N\ell-u^*| \leq \gep_n} \sum_{|\frac1N r -v^*| \leq \gep_n}  \exp\Big( N  g_{\lambda, \frac{\lfloor aN \rfloor}{N} } \Big(\frac{\ell}{N} , \frac{r}{N}\Big) \Big) \,,
\end{equation}
where we used the definition~\eqref{def:g} of $g_{\lambda,a}$ to simplify notation.
Let us stress that the constant $C_{\lambda,a}$ is explicit, $C_{\lambda,a}:= \frac{1}{m_{\lambda}} p_{x_{u^*}} p_{x_{\bar v^*}} \frac{a^2 u^* \bar v^*}{\sigma_1 \sigma_2}$,
with $\sigma_1,\sigma_2$ defined below in~\eqref{def:sigmas} (and $\bar v^*:=1-v^*$).

Then, we will use Taylor expansions inside $g_{\lambda, \frac{\lfloor aN\rfloor}{N}}$.
Let us denote, for $\ell,r$ in the range of the sum in~\eqref{eq:decomp}, 
\[
\gep_{a}:= 1- \frac{\lfloor aN \rfloor}{aN} = \frac{\{aN\}}{aN}= O(N^{-1})\,,
\quad
\gep_{\ell}:= \frac{\ell}{u^*N} - 1 = O(\gep_N)\,,
\quad
\gep_r := \frac{N-r}{\bar v^*N} -1 =O(\gep_N)\,,
 \]
with $\bar v^*=1-v^*$.
Then, thanks to a Taylor expansion of $\I_-$ around $\frac{a}{u^*}$, we get
\[
\begin{split}
\frac{\ell}{N} \I_-&\Big(\frac{\lfloor aN \rfloor}{\ell}\Big)  = (1+\gep_\ell) u^* \I_-\Big(\frac{a}{u^*} \frac{1-\gep_a}{1+\gep_{\ell}}\Big) \\
& =(1+\gep_\ell) u^* \I_-\Big(\frac{a}{u^*}\Big)
- a (\gep_\ell +\gep_a) \I_-'\Big(\frac{a}{u^*} \Big) +  \frac{a^2}{2 u^*} \frac{(\gep_\ell +\gep_a)^2}{1+\gep_{\ell}}\I_-''\Big(\frac{a}{u^*}\Big) + (\gep_\ell +\gep_a)^3 O(1)  \,,
\end{split}
\]
with the $O(1)$ uniform in $|\frac{1}{N}\ell-u^*|\leq \gep_N$.
Hence we can write $\frac{\ell}{N} \I_-\big(\frac{\lfloor aN \rfloor}{\ell}\big)$ as
\[
u^* \I_-\Big(\frac{a}{u^*}\Big) +\gep_{\ell} u^* \Big( \I_-\Big(\frac{a}{u^*}\Big) - \frac{a}{u^*} \I_-'\Big(\frac{a}{u^*}\Big) \Big) +  \frac{a^2}{2 u^*} \gep_\ell^2 \I_-''\Big(\frac{a}{u^*}\Big) - a\gep_a \I_-\Big(\frac{a}{u^*}\Big) + O(\gep_\ell^3) + o\Big( \frac{1}{N} \Big) \,.
\]
Using Lemma~\ref{lem:maximizer} and denoting $\sigma_1^2 := \frac{(u^*)^3}{a^2 \I_-''(a/u^*)}$, this can be rewritten as
\[
\frac{\ell}{N} \I_-\Big(\frac{\lfloor aN \rfloor}{\ell}\Big)  = u^* \I_-\Big(\frac{a}{u^*}\Big) - \gep_{\ell} u^* \f(\lambda) + \frac{(u^*\gep_\ell)^2}{2\sigma_1^2} - a\gep_a \I_-\Big(\frac{a}{u^*}\Big) + O(\gep_\ell^3) + o\Big( \frac{1}{N} \Big) \,.
\]
Similarly, denoting $\sigma_2^2 :=\frac{(\bar v^*)^3}{a^2 \I_+''(a/\bar v^*)}$ (and also $\bar v^*=1-v^*$ to simplify notation), we have
\[
\frac{N-r}{N} \I_+\Big(\frac{\lfloor aN \rfloor}{N-r}\Big) =  \bar v^* \I_+\Big(\frac{a}{\bar v^*}\Big)  - \gep_{r} \bar v^* \f(\lambda) + \frac{(\bar v^*\gep_r)^2}{2\sigma_2^2} - a\gep_a \I_+\Big(\frac{a}{\bar v^*}\Big) +O(\gep_{r}^3) + o\Big( \frac{1}{N} \Big) \,.
\]
All together, recalling the definition~\eqref{def:psi} of $g_{\lambda,a}$, we obtain
\begin{multline*}
g_{\lambda, \frac{\lfloor aN \rfloor}{N} } \Big(\frac{\ell}{N} , \frac{r}{N}\Big) 
= \f(\lambda) \Big(\frac{r}{N}-\frac{\ell}{N} + \gep_{\ell} u^* +  \gep_{r} \bar v^*\Big) 
- u^* \I_-\Big(\frac{a}{u^*}\Big) - \bar v^* \I_+\Big(\frac{a}{\bar v^*}\Big) \\
-  \frac{(u^*\gep_\ell)^2}{2\sigma_1^2} -  \frac{(\bar v^*\gep_r)^2}{2\sigma_2^2}
+ c_0 a\gep_a + O(\gep_\ell^3+\gep_{r}^3) + o\Big( \frac{1}{N} \Big)\,,
\end{multline*}
with $c_0 := \I_-\big(\frac{a}{u^*}\big)+ \I_+\big(\frac{a}{\bar v^*}\big)$.
Recalling the definition of $\gep_{\ell},\gep_r$, we have that $\frac{\ell}{N} - \gep_{\ell} u^* =u^*$ and $\frac{r}{N}+  \gep_{r} \bar v^*$; also, $a\gep_a = \frac1N \{aN\}$.
Hence, since we have $\psi(\lambda,a)= (v^*-u^*) \f(\lambda) - u^* \I_-\big(\frac{a}{u^*}\big) - \bar v^* \I_+\big(\frac{a}{\bar v^*}\big)$, we end up with
\begin{equation}
\label{eq:DLglambda}
g_{\lambda, \frac{\lfloor aN \rfloor}{N} } \Big(\frac{\ell}{N} , \frac{r}{N}\Big)  = \psi(\lambda,a) - \frac{(u^*\gep_\ell)^2}{2\sigma_1^2}  - \frac{ (\bar v^* \gep_r)^2  }{2\sigma_2^2}  + c_0 \frac{1}{N} \{aN\}+  (\gep_\ell^3+\gep_r^3)O(1) + o\Big(\frac1N\Big)\,,
\end{equation}
where the constants are
\begin{equation}
\label{def:sigmas}
\sigma_1^2 := \frac{(u^*)^3}{a^2 \I_-''(a/u^*)},\qquad  \sigma_2^2 :=\frac{(\bar v^*)^3}{a^2 \I_+''(a/\bar v^*)} ,\qquad
 c_0 := \I_-\Big(\frac{a}{u^*}\Big)+ \I_+\Big(\frac{a}{\bar v^*}\Big)\,.
\end{equation}
Note that, in~\eqref{eq:DLglambda}, the $O(1)$ and $o(\frac1N)$ are uniform (depending only on $\gep_N$).
Going back to~\eqref{eq:decomp}, we obtain that
\[
\check Z_{N,\lambda}^a =(1+o(1))  \frac{C_{\lambda,a}}{2\pi N} e^{N\psi(\lambda, a) + c_0 \{aN\}} \sum_{ |\frac1N\ell-u^*| \leq \gep_N}  e^{  -  (1+o(1))  \frac{(u^*\gep_\ell)^2 N}{2\sigma_1^2} } \sum_{|\frac1N r -v^*| \leq \gep_N} e^{-(1+o(1)) \frac{(\bar v^*\gep_r)^2 N }{2\sigma_2^2} } \,.
\]
Recalling that $u^*\gep_{\ell} = \frac{\ell}{N}-1$, we get that
\[
\sum_{ |\frac1N\ell-u^*| \leq \gep_N} \frac{1}{\sqrt{N}} e^{  -  (1+o(1))  \frac{(u^*\gep_\ell)^2 }{2\sigma_1^2} N}
 = \sum_{ |j|\leq N\gep_N} \frac{1}{\sqrt{N}} e^{  -  (1+o(1))  \frac{ j^2}{2\sigma_1^2 N} } \xrightarrow{N\to\infty} \int_{\mathbb R} e^{- \frac{s^2}{2\sigma_1^2} } \dd s  = \sqrt{2\pi}\sigma_1\,,
\]
by a Riemann sum approximation. A similar convergence holds for the other sum.
We therefore end up with 
\[
\check Z_{N,\lambda}^a \sim \sigma_1\sigma_2 C_{\lambda,a} e^{N\psi(\lambda, a) + c_0 \{aN\}}\,.
\]

Since we have seen above that $\bar Z_{N,\lambda}^a = O(1/\sqrt{N})$, we get that whenever $\psi(\lambda,a) \geq 0$, \textit{i.e.} when $\lambda\geq \lambda_c(a)$, which in particular entails that $\f(\lambda,a)=\psi(\lambda, a)$, we get
\begin{equation}
\label{eq:asympZfinal}
Z_{N,\lambda}^a = \bar Z_{N,lambda}^a+\check Z_{N,\lambda}^a = (1+o(1)) \check Z_{N,\lambda}^a \sim \sigma_1\sigma_12 C_{\lambda,a} e^{N\psi(\lambda, a) + c_0 \{aN\}} \,.
\end{equation}
This concludes the proof of~\eqref{eq:asympZN}.
Note also that in view of the definition of $C_{\lambda, a}$ above, we have $c_1 := \frac{1}{m_{\lambda}} a^2 u^* \bar v^*  p_{a/u^*}p_{a/\bar v^*}$ in~\eqref{eq:asympZN}.
\qed

\subsection{Proof Theorem~\ref{th:Gaussian}, part (ii)}

Let us now turn to the proof of the local central limit theorem, that is~\eqref{eq:TCLlocalZ}.
Notice that we have
\[
\bP_{N,\lambda}^a\big( L_N = \ell, R_N =r \big)
= \frac{1}{Z_{N,\lambda}} \check Z_{\ell,r} \,.
\]
From the first part of Theorem~\ref{th:Gaussian} (see in particular~\eqref{eq:asympZfinal}) and thanks to Proposition~\ref{prop:checkZlr}, we obtain, similarly to~\eqref{eq:decomp},
 \[
 \bP_{N,\lambda}^a\big( L_N = \ell, R_N =r \big)
 =  (1+o(1)) \frac{1}{2\pi N \sigma_1 \sigma_2} \exp\Big( N  g_{\lambda, \frac{\lfloor aN \rfloor}{N} } \Big(\frac{\ell}{N} , \frac{r}{N}\Big) - N\psi(\lambda, a) -c_0 \{aN\}\Big) \,,
 \]
with the $o(1)$ uniform over $|\frac1N\ell-u^*|\leq \gep_N$, $|\frac1Nr-u^*|\leq \gep_N$.
Then, applying the Taylor expansion of~\eqref{eq:DLglambda} and recalling that $\gep_\ell = \frac{\ell}{u^*N}-1$ and $\gep_{r}= \frac{N-r}{N\bar v^*} -1$, we get  
 \[
 N  g_{\lambda, \frac{\lfloor aN \rfloor}{N} } \Big(\frac{\ell}{N} , \frac{r}{N}\Big) - N\psi(\lambda, a) - c_0 \{aN\}
 = -\frac{(\ell- u^*N)^2}{2\sigma_1^2 N} (1+\gep_\ell O(1))
 -\frac{(r- v^*N)^2}{2\sigma_1^2 N} (1+\gep_r O(1)) \,,
 \]
with the $O(1)$ uniform. This concludes the proof of~\eqref{eq:TCLlocalZ}.

To obtain the convergence in distribution, we simply notice that  for any $z_1<z_2$ and $z_1'<z_2'$, we can write
 \[
 \begin{split}
 \bP_{N,\lambda}^a&\big( z_1\sqrt{N} \leq L_N -u^* N \leq z_2 \sqrt{N},  z_1'\sqrt{N}\leq  R_N -v^* \leq z_2'\sqrt{N} \big) \\
& \qquad \qquad = \sum_{z_1\sqrt{N}  \leq \ell -u^*N \leq z_2\sqrt{N}} \sum_{z_1'\sqrt{N} \leq r -v^*N \leq z_2'\sqrt{N}}
\bP_{N,\lambda}^a\big( L_N = \ell, R_N =r \big) \,.
\end{split}
 \]
Then, using the local central limit theorem from~\eqref{eq:TCLlocalZ} and a Riemann sum approximation, we get that the above probability converges to
\[
\int_{z_1}^{z_2} \frac{1}{\sqrt{2\pi} \sigma_1} e^{-\frac{s_1^{2}}{2\sigma_1^2}} \dd s_1  \int_{z_1'}^{z_2'} \frac{1}{\sqrt{2\pi} \sigma_2} e^{-\frac{s_2^{2}}{2\sigma_1^2}} \dd s_2 \,,
\]
which concludes the proof.
\qed

\subsection{Proof of Corollary~\ref{cor:HN}}

Let us prove a more general result, \textit{i.e.}\ a joint local central limit theorem for $L_N,R_N,H_N$.
First, let us observe that 
$\bP_{N,\lambda}^a(H_N=k \, \big| \,L_N=\ell, R_N=r ) = \bP_{r-\ell,\lambda}( H_{r-\ell}=k )$, where $\bP_{n,\lambda}$ is the (standard) wetting measure~\eqref{eq:PNB2}.
Then, and using the local limit theorem of Proposition~\ref{prop:Hlocal}, we get that, as $r-\ell\to\infty$
\[
\bP_{N,\lambda}^a\big(H_N=k \, \big| \,L_N=\ell, R_N=r \big) =  \frac{m_{\lambda}}{\sqrt{2\pi} \sigma_{\lambda}} e^{- \frac{ (r-\ell -m_{\lambda} k)^2}{2 \sigma_{\lambda}^2 (r-\ell)}} + o\Big( \frac{1}{\sqrt{r-\ell}} \Big) \,.
\]
Let us now set 
\[
\Delta_{\ell}:=|\ell-u^*N|,\quad \Delta_r:=|r-v^*N|, \quad \Delta_k:=|k-m_{\lambda}^{-1}(u^*-v^*) N| \,.
\]
Then, using the local central limit theorem of Theorem~\ref{th:Gaussian}-(ii), see~\eqref{eq:TCLlocalZ}, we obtain that for any $A>0$, uniformly for $\Delta_{\ell},\Delta_r,\Delta_k \leq A\sqrt{N}$ (in particular $|(r-\ell)- (v^*-u^*)N|\leq 2A\sqrt{N}$), we have
\[
\bP_{N,\lambda}^a\big(L_N=\ell ,R_N=r,H_N=k \big)
=  \frac{(1+o(1)) m_{\lambda}}{(2\pi N)^{3/2} \sqrt{v^*-u^*} \sigma_1\sigma_2 \sigma_{\lambda} } 
e^{- \frac{\Delta_\ell^2}{2 \sigma_1^2 N}}
e^{- \frac{\Delta_r^2}{2 \sigma_1^2 N}}
e^{- \frac{(\Delta_r-\Delta_n- m_{\lambda} \Delta_k)^2}{2\sigma_{\lambda}^2 (v^*-u^*)N} } \,.
\]
In other words, letting $\sigma_3:= \frac{\sqrt{v^*-u^*}}{m_{\lambda}} \sigma_{\lambda}$, we have
\begin{equation}
\label{eq:localLRH}
\bP_{N,\lambda}^a\big(L_N=\ell ,R_N=r,H_N=k \big)
= \frac{(1+o(1))}{(2\pi N)^{3/2} \sigma_1\sigma_2 \sigma_{3} } 
e^{- \frac{1}{2} Q\big(\frac{\Delta_\ell}{\sqrt{N}},\frac{\Delta_r}{\sqrt{N}},\frac{\Delta_k}{\sqrt{N}} \big) } \,,
\end{equation}
where $Q(z_1,z_2,z_3)= \frac{z_1^2}{\sigma_1^2} + \frac{z_2^2}{\sigma_2^2} + \frac{(z_3- m_{\lambda}^{-1}(z_2-z_1))^2}{\sigma_{3}^2}$ as defined below~\eqref{eq:localLRH}; note that the $o(1)$ is uniform in $\Delta_{\ell},\Delta_r,\Delta_k \leq A\sqrt{N}$.
The convergence in distribution stated in~\eqref{eq:convLRH} then follows directly from~\eqref{eq:localLRH} by a Riemann sum approximation.
\qed


\begin{appendix}

\section{A few examples of integrable wetting models}
\label{sec:examples}

To complement the study of the standard wetting model, we collect a few examples for which the free energy $\f(\lambda)$ (or the critical point) of the  admits an explicit formula.
We start with discrete examples, that have been more studied in the literature, before turning to some continuous cases.

We then compute explicitly the free energy $\f(\lambda,a)$ and the critical curve $a_c(\lambda)$ of the wetting model with elevated boundary condition.
In all examples, the underlying random walk is symmetric, so we denote $\Lambda(t):=\Lambda_+(t) =\Lambda_-(t)$.
Recall that, by Theorem~\ref{th:formuleG} we have that the free energy and the critical curve are given by
\[
\f(\lambda, a) = \Big( \f(\lambda)- 2a \Lambda^{-1}(\f(\lambda))\Big)_+ \,,\qquad a_c(\lambda) = \frac{\f(\lambda)}{2 \Lambda^{-1}(\f(\lambda))} \,.
\]

\subsection{Symmetric lazy random walk}
\label{ex:lazyRW}

Let $\gamma \in (0,\frac{1}{2})$ and consider $(X_i)_{i\geq 1}$ i.i.d.\ random variables, with symmetric distribution given by $\PP(X_i = 1) = \PP(X_i = -1) = \gamma$ and $\PP(X_i = 0) = 1 - 2\gamma>0$; the case $\gamma\uparrow\frac12$ corresponds to the simple symmetric random walk.

This model has been studied in details, see e.g.~\cite{Fish84,IY01}.
For the critical point, using Lemma~\ref{lem:ladder} we have $ \lambda_c = \kappa^{-1}=\bP(\bar H_1=0)^{-1}$. Since the steps are only $\pm1$ or $0$, the (weak) ladder height is always zero except if $X_1=-1$: we therefore get
$\lambda_c = \frac{1}{\bP(\bar H_1 =0)} = \frac{1}{1-\gamma}$.
It turns out that the free energy can also be computed explicitly, see e.g.~\cite[Eq.~(1.7)]{IY01}: 
\[
\f(\lambda) = \ln x_{\lambda}  \quad \text{ for } \lambda\geq \lambda_c = \frac{1}{1-\gamma}\,,
\]
where $x_{\lambda}$ is the positive solution of $(\lambda-1)x^2 - \lambda(\lambda-1) (1-2\gamma) x  -\lambda^2 \gamma^2=0$; one can easily check that $x_{\lambda}\geq 1$ if $\lambda\geq \lambda_c$.
As far as the critical behavior is concerned, we leave as an exercise to check that $\f(\lambda_c +u) \sim \gamma(3-4\gamma) u^2$ as $u\downarrow 0$.

Note that in the limit $\gamma\uparrow\frac12$, we find the free energy of the wetting model for the simple random walk found in~\cite[Eq.~(1.6)]{LT15} (see also~\cite[Ch.~7]{dH07}): $\f(\lambda)= \ln\big(\frac{\lambda}{2\sqrt{\lambda-1}}\big) \ind_{\{\lambda>2\}}$.
Another remarkable value is in the case $\gamma=\frac14$, we find:
$\f(\lambda) = \ln\big( \frac{\lambda}{4\sqrt{\lambda-1}} (\sqrt{\lambda-1} +\sqrt{\lambda}) \big)\ind_{\{\lambda>\frac43\}}$.

\subsubsection*{With elevated boundary conditions}
We use that $\Lambda(t) = \ln \big( 1+2\gamma (\cosh(t)-1) \big)$ for any $t\in \mathbb R_+$ so $\Lambda^{-1}(x) = \cosh^{-1}\big(1+ \frac{1}{2\gamma}(e^x-1) \big)$.
Together with the formula $\f(\lambda) = \ln x_{\lambda}$, we get that
$\Lambda^{-1}(\f(\lambda)) = \cosh^{-1}\big(1+ \frac{1}{2\gamma} (x_{\lambda}-1) \big)$, so we end up with
\[
\f(\lambda,a) = \Big( \ln x_{\lambda} - 2a \cosh^{-1}\big(1+ \tfrac{1}{2\gamma} (x_{\lambda}-1) \big) \Big)_+ \,,
\qquad a_c(\lambda) = \frac{\ln x_{\lambda}}{2 \cosh^{-1}\big(1+ \frac{1}{2\gamma} (x_{\lambda}-1) \big)} \,,
\]
where $x_{\lambda}$ is the positive solution of $(\lambda-1)x^2 - \lambda(\lambda-1)(1-2\gamma) x- \lambda^2 \gamma^2 =0$.
Note that in the case $\gamma \uparrow\frac12$, we obtain a free energy $\f(\lambda,a) = \big( \ln (\frac{\lambda}{2\sqrt{\lambda-1}}) - 2a \cosh^{-1}(\frac{\lambda}{2\sqrt{\lambda-1}})  \big)_+$.

\subsection{Symmetric geometric random walk, one-dimensional (discrete) SOS}
\label{ex:geomRW}
Let $\gamma \in (0,1)$ and consider $(X_i)_{i\geq 1}$ i.i.d.\ random variables, with symmetric distribution given by $\bP(X_i=k) = c_{\gamma}\gamma^{|k|}$, with $c_{\gamma} = \frac{1-\gamma}{1+\gamma}$.
This geometric random walk arises naturally in the context of the (discrete) Solid-On-Solid (SOS) model, which is a gradient interface model with potential $V(x) = -\log f(x) = |x|$, used is an effective model for interfaces in the Ising model, see~\cite{Vel06,IV18} for reviews.
It also appears in the context of Interacting Partially Directed Self-avoiding Walk, see~\cite{CNPT18} for a review.

As far as the critical point is concerned, we also use Lemma~\ref{lem:ladder} to get that $\lambda_c = \bP(\bar H_1=0)^{-1}$.
Here, thanks to the memoryless property of the geometric distribution, one easily gets that the ladder height $\bar H_1$ has distribution $\bP(\bar H_1 =k) = (1-\gamma)\gamma^k$ for $k\geq 0$.
We therefore get that
$\lambda_c = \frac{1}{\bP(\bar H_1 =0)} = \frac{1}{1-\gamma}$.
The computation of the Laplace transform~\eqref{def:LaplaceK} of $K(\cdot)$ (hence of the free energy) has been made in~\cite[Prop.~A.1]{LP22}: we have
\[
\f(\lambda)= \ln \Big( \frac{\lambda(\lambda-1) (1-\gamma)^2}{\lambda(1-\gamma^2)-1} \Big) \qquad \text{ for } \lambda\geq \lambda_c= \frac{1}{1-\gamma}\,.
\]
As far as the critical behavior is concerned, we find that $\f(\lambda_c +u) \sim  \frac{(1-\gamma)^2}{\gamma} u^2$ as $u\downarrow 0$.

\subsubsection*{With elevated boundary conditions} We use that  $\Lambda(t) = -\ln \big( 1 -\frac{2\gamma}{(1-\gamma)^2} (\cosh(t)-1) \big) $ for any $|t| < \ln \frac{1}{\gamma}$, so $\Lambda^{-1}(x) = \cosh^{-1}\big(1+  \frac{(1-\gamma)^2}{2\gamma}(1-e^{-x}) \big)$.
Then, using the formula above for $\f(\lambda)$, we get that $\Lambda^{-1}(\f(\lambda)) = \cosh^{-1}\big( \frac{(\lambda-1)^2 + \gamma^2 \lambda^2 }{2\gamma \lambda (\lambda-1)} \big)$, so we end up with
\[
\f(\lambda,a) =  \bigg( \ln \Big(\frac{\lambda(\lambda-1) (1-\gamma)^2}{\lambda(1-\gamma^2) -1} \Big) - 2 a \cosh^{-1}\Big( \frac{(\lambda-1)^2 + \gamma^2 \lambda^2 }{2\gamma \lambda (\lambda-1)}\Big)  \bigg)_+\,,
\]
and the value of $a_c(\lambda)$ can be read from the above.

\subsection{Symmetric Laplace random walk, one-dimensional (continuous) SOS}
\label{ex:exponential}
Let $\gamma>0$ and consider $(X_i)_{i\geq 1}$  i.i.d.\ random variables with symmetric Laplace distribution of parameter $\gamma$, that is with density $f(x) = \frac12 \gamma e^{-\gamma |x|}$; this is a symmetrized $\mathrm{Exp}(\gamma)$ distribution.
The measure~\eqref{eq:PNB} then corresponds to the ($\delta$-pinning) wetting of the continuous Solid-On-Solid (SOS) model in dimension $d=1$, see~\cite{CV00}.
As for the geometric random walk, the memoryless property of the exponential gives that $\bar H_1$ has an $\mathrm{Exp}(\gamma)$ distribution: we get that $\bar H_1$ has density $f_{\bar H_1}(x) = \gamma e^{-\gamma x} \ind_{\{x\geq 0\}}$.
Using again Lemma~\ref{lem:ladder}, we obtain
$\lambda_c = \frac{1}{f_{\bar H_1}(0)} = \frac{1}{\gamma}$.
As far as the free energy is concerned, we can also compute the Laplace transform $\mathcal{K}(\vartheta)$ explicitly (this is done in Appendix~\ref{app:Laplace} and matches the formula~\cite[Eq.~(3.6)]{dCDH11}): we obtain
\[
\f(\lambda)=  \ln\Big( \frac{\gamma^2 \lambda^2}{2\gamma \lambda -1} \Big) \qquad \text{ for } \lambda\geq \lambda_c  =\frac{1}{\gamma}\,.
\]
As far as the critical behavior is concerned, we find that $\f(\lambda_c +u) \sim  \gamma^2 u^2$ as $u\downarrow 0$.

\subsubsection*{With elevated boundary conditions}
We use that $\Lambda(t) = -\ln (1-t^2/\gamma^2)$ for any $|t| < \gamma$, so we obtain $\Lambda^{-1}(x) = \gamma \sqrt{1-e^{-x}}$.
Using the formula above for $\f(\lambda)$, we get that $\Lambda^{-1}(\f(\lambda)) = \frac{\gamma \lambda-1}{\lambda}$, so we end up with
\[
\f(\lambda,a) = \bigg( \ln \Big(\frac{\gamma^2 \lambda^2}{2\gamma\lambda-1} \Big) -   \frac{2a( \gamma\lambda-1)}{\lambda}  \bigg)_+\,,
\qquad a_c(\lambda) =\frac{ \lambda }{ 2(\gamma  \lambda-1)}  \ln \Big(\frac{\gamma^2\lambda^2}{2\gamma\lambda-1} \Big)\,.
\]

\subsection{Strictly $\alpha$-stable random walk, one-dimensional Gaussian free field}
\label{ex:stable}

Consider $(X_i)_{i\geq 1}$ i.i.d.\ random variables with strictly $\alpha$-stable distribution, $\alpha\in (0,2]$, \textit{i.e.}\ such that $n^{-1/\alpha} S_n$ has the same law as $X_1$.
In other words, we have $X_i\sim \mathcal Z$ and $a_n=n^{1/\alpha}$ in Assumption~\ref{hyp:1}; examples include standard Cauchy and Normal distributions.
In the Gaussian case, the measure~\eqref{eq:PNB} corresponds to the ($\delta$-pinning) wetting of the massless Gaussian free field in dimension $d=1$,  see~\cite{CV00}.

This class of wetting models is not completely integrable, in the sense that there is no closed formula for the free energy, but still, the critical point is explicit.
Indeed, relying on the relation~\eqref{rel:Kconti} and the strict stability, we obtain that $f_n^+(0) = f_{\alpha}(0) n^{-(1+\frac1\alpha)}$ for all $n\geq 1$, where $f_{\alpha}$ is the density of $X_1\sim \mathcal Z$; note that an explicit expression for $f_{\alpha}(0)$ can be found in~\cite[Cor.~3.1]{Nolan20}.
Thanks to Theorem~\ref{prop:pinning}, we therefore get
\[
\lambda_c = \frac{1}{\kappa} := \frac{1}{\sum_{n=1}^{\infty} \kappa_n} = \frac{1}{f_{\alpha}(0)\zeta(1+\frac1\alpha)}   \,.
\]
For instance: we get $\lambda_c= \sqrt{2\pi} /\zeta(3/2)$ if $X_i\sim\mathcal{N}(0,1)$; we get $\lambda_c= \pi/ \zeta(2)=6/\pi$ if $X_i\sim\mathcal{C}(0,1)$.

In view of the formula for $f^+_n(0)$, we obtain that $K(n)= n^{-s} \zeta(s)^{-1}$, so the renewal $\tau$ has a $\mathrm{zeta}(s)$ inter-arrival distribution, with $s:=1+\frac{1}{\alpha}$. Hence, the Laplace transform~\eqref{def:LaplaceK} is $\mathcal{K}(\vartheta) =\zeta(s)^{-1}\mathrm{Li}_s(e^{-\vartheta})$ where $\mathrm{Li}_s(z) := \sum_{n\geq 1} z^n n^{-s}$ is the so-called polylogarithmic function.
The free energy is then given by the relation~\eqref{def:freeenergy}, \textit{i.e.}\ $\mathrm{Li}_s(\f(\lambda)) = \zeta(s) (\kappa \lambda)^{-1} = (f_{\alpha}(0) \lambda)^{-1}$, which cannot be explicitly inverted.
On the other hand, after some calculation, one gets that $\f(\lambda_c+u)\sim c_{\alpha} u^{\alpha}$ as $u\downarrow 0$, with $c_{\alpha}:= (f_{\alpha}(0)/\alpha\Gamma(1/\alpha))^{\alpha}$; see also Proposition~\ref{prop:criticF}.

\begin{remark}
\label{rem:explicit}
In the continuous case, the relation~\eqref{rel:Kconti} gives that $f_n^+(0) = \frac1n f_n(0)$, where $f_n$ is the density of $S_n$. Therefore, if the density of $S_n$ at $0$ is explicit, one obtains an explicit formula for $f_n^+(0)$ and the model could also be (at least partially) integrable. 
\end{remark}

\subsubsection*{With elevated boundary conditions} Let us focus on the (standard) Gaussian random walk. 
We have that $\Lambda(t)=\frac12 t^2$, so $\Lambda^{-1}(x) = \sqrt{2 x}$.
We therefore get that
\[
\f(\lambda,a) = \Big(  \f(\lambda) -  2 a \sqrt{2 \f(\lambda)}  \Big)_+\,,
\qquad a_c(\lambda) = \frac{1}{2\sqrt{2}} \sqrt{\f(\lambda)} \,,
\]
but with no explicit expression for the free energy $\f(\lambda)$.

\subsection{Generalized Laplace distribution}
\label{ex:gLaplace}

In view of Remark~\ref{rem:explicit}, we include in the list of examples the case of (symmetric) generalized Laplace distributions, also known as \emph{variance-gamma} or \emph{Bessel function} distributions; we refer to~\cite[Sec.~4.1]{KKP01} for an overview.
The density of a symmetric generalized Laplace distribution of parameter $(\nu,\sigma)$ (we write $X\sim \textrm{g-Laplace}(\nu,\sigma)$ for short) has an explicit distribution, see \cite[\S4.1.4.2]{KKP01}, given by
$f_{\nu,\sigma}(x) = \frac{\sigma\sqrt{2}}{\Gamma(\nu) \sqrt{\pi}} \Big( \frac{|x|}{\sigma\sqrt{2}}\Big)^{\nu-1/2} \mathrm{K}_{\nu-1/2} \big(\sqrt{2} |x|/\sigma \big) $,
where $\mathrm{K}_{\nu-1/2}$ is the modified Bessel function of the third kind.
Additionally, the behavior of the density at $x=0$ is known: it goes to infinity if $\nu\leq \frac12$ and 
$f_{\nu,\sigma}(0)=\frac{1}{\sigma\sqrt{2\pi}} \frac{\Gamma(\nu-\frac12)}{\Gamma(\nu)}$ if $\nu>\frac12$.

A random variable $X\sim\textrm{g-Laplace}(\nu,\sigma)$ can be represented as a mixture of centered normal distribution with random variance $\sigma^2 W$, with $W\sim \mathrm{Gamma}(\nu,1)$; alternatively, it can be represented as the difference of two independent $\mathrm{Gamma}(\nu,\gamma)$ random variables with $\gamma = \sqrt{2}/\sigma$.
Additionally, such distributions are stable under convolution: the sum of two independent g-Laplace random variables with respective parameters $(\nu,\sigma)$, $(\nu',\sigma)$ has itself a g-Laplace distribution, with parameters $(\nu+\nu',\sigma)$.

Therefore, if $(X_i)_{i\geq 1}$ are i.i.d.\ with g-Laplace$(\nu,\sigma)$ distribution (with $\sigma=\sqrt{2}/\gamma$ if one wants to match the parametrization of Section~\ref{ex:exponential}), then $S_n\sim\textrm{g-Laplace}(n\nu,\sigma)$.
Therefore, taking $\nu>\frac12$ so that $f_n(0)<+\infty$ for all $n$ and using the relation~\eqref{rel:Kconti} together with the value for $f_{n\nu,1}(0)$ above, we get that
\[
f_n^+(0) = \frac1n f_n(0) = \frac{1}{\sigma\sqrt{2\pi}} \frac{\Gamma(n\nu-\frac12)}{\Gamma(n\nu+1)} \,. 
\]
Hence, the critical point is ``explicit'', since $\lambda_c = \kappa^{-1}$ with $\kappa = \sum_{n=1}^{\infty} f_n^+(0)$. 
It turns out that the Laplace transform (or generating function) of $f_n^+(0)$ has a simple form if $\nu=1$ (this is the example of Section~\ref{ex:exponential}), but also if $\nu=2$ (it also has a closed form if $\nu=3$):
for all $x\in [-1,1]$,
\[
\sum_{n=1}^{\infty}  x^n f_n^+(0) =
\begin{cases}
\frac{\sqrt{2}}{\sigma} (1-\sqrt{1-x}) &\quad \text{ if } \nu=1\,,\\ 
\frac{\sqrt{2}}{\sigma} (1 - \frac{1}{\sqrt{2}}\sqrt{1+\sqrt{1-x}}) &\quad \text{ if } \nu=2\,.
\end{cases}
\]
Inverting the formula of the Laplace transform   gives the expression of the free energy (and of the critical point), thanks to \eqref{def:freeenergy}. When $\nu=2$, setting $\gamma = \sqrt{2}/\sigma$, we get
\[
\f(\lambda) = - \ln \Big( \frac{4 (\gamma \lambda-1)^2 (2\gamma\lambda-1)}{(\gamma\lambda)^4}  \Big) \qquad \text{ for } \lambda\geq  \lambda_c = \frac{1}{\gamma} \frac{\sqrt{2}}{\sqrt{2} -1} \,.
\]
To conclude, let us mention that when $\nu$ is an integer, the generating function of $f_n^+(0)$ is equal to the generalized hypergeometric function ${}_{\nu+1}F_\nu( \frac{-1}{2\nu}, \frac{1}{2\nu},\ldots, \frac{2\nu-3}{2\nu};\frac{1}{\nu},\ldots,\frac{\nu-1}{\nu};x)$, up to centering (by $1$) and normalizing (by $\frac{1}{\sigma \sqrt{2\pi}}$).

\subsubsection*{With elevated boundary conditions}
We use that $\Lambda(t) = -\nu\ln( 1-  \frac12 \sigma^2 t^2)$, so that $\Lambda^{-1}(x) = \frac{\sqrt{2}}{\sigma} \sqrt{1-e^{-x/\nu}}$; for instance using the representation as a difference of independent $\Gamma(\nu,\gamma)$ random variables with $\gamma= \sqrt{2}/\sigma$.
Similarly to the Gaussian case, we obtain that the critical line is $a_c(\lambda) = \frac{\sigma}{2\sqrt{2}} \f(\lambda) (1-e^{-\f(\lambda)/\nu})^{-1/2}$, with in general no explicit expression for $\f(\lambda)$.

\section{A few results related to the (standard) wetting model}

\subsection{Proof of Lemmas~\ref{lem:ladder} and~\ref{lem:kappa}}
\label{app:renewal}

We start with the proof of Lemma~\ref{lem:ladder}, which uses a simple rewritting of $f_n^+(x)$ in terms of (weak) ladder epochs and heights.

\begin{proof}[Proof of Lemma~\ref{lem:ladder}]
Let us introduce the first (weak) ladder epoch and height 
\[
\bar T_1 := \min\{ n\geq 1 , S_n \leq 0\} \quad \text{ and } \quad \bar H_1 = -S_{\bar T_1} \,,
\]
and the following density with respect to $\mu$:
\[
f_n^+(x) = \frac{1}{\mu(\dd x)}\bP( \bar T_1= n, \bar H_1 \in \dd x)   \quad \text{ for } x\geq 0\,.
\]
Hence, summing over $n$ we get that $\kappa =\kappa(0)$ where $\kappa(x)=\sum_{n=1}^{\infty} f_n^+(x) = f_{\bar H_1}(x)$, with $f_{\bar H_1}$ the density of $\bar H_1$ with respect to $\mu$.
In particular, $\kappa =\bP(\bar H_1=0)$ in the discrete case and $\kappa = f_{\bar H_1}(0)$ in the continuous case.
This concludes the proof of Lemma~\ref{lem:ladder}.
\end{proof}

We now turn to the proof of Lemma~\ref{lem:kappa}: the result is given in~\cite{CC13} (Proposition~4.1-(4.5) for the discrete case and Theorem~5.1-(5.2) for the continuous case); we give here a simpler proof for the sake of completeness since we do not aim for the level of generality in~\cite{CC13}.

\begin{proof}[Proof of Lemma~\ref{lem:kappa}, continuous case]
In the continuous case, by using \cite[App.~A.2]{CGZ06}, we obtain that 
\begin{equation}
\label{rel:Kconti}
f_n^+(0):= \frac{1}{n} f_{n}(0)
\end{equation}
where $f_n$ is the density of $S_n$ w.r.t.\ the Lebesgue measure; see also~\cite[Eq.~(3)]{AD99} for a more general statement\footnote{Note that \cite[Eq.~(3)]{AD99} considers \emph{strong} rather than \emph{weak} ladder epochs and heights: in our context, it reads $\bP(\bar T_1=n, \bar H_1\in \dd x) = \frac{1}{n} \bP(S_n \in \dd x, \bar H_1>x)$ for any $n\geq 1$ and $x\geq 0$.}.

Then the local limit theorem for densities~\cite[Thm.~4.3.1]{IL71}
gives that if $f_n$ is bounded for some $n\geq 1$ (which is given by Assumption~\ref{hyp:1}), then
$f_{n}(0) \sim \frac{1}{a_n} f_{\alpha}(0)$.
where $f_{\alpha}$ is the density of the limiting  $\alpha$-stable distribution.
This shows Lemma~\ref{lem:kappa} in the continuous case, recalling that $\bP(\bar H_1 >0)=1$.
\end{proof}

\begin{proof}[Proof of Lemma~\ref{lem:kappa}, discrete case]
In the discrete case, we start from the relation~\cite[Eq.~(3)]{AD99}, which reads
\begin{equation}
\label{eq:AD99}
f_n^+(0) = \bP(\bar T_1 =n, \bar H_1 =0) = \frac{1}{n} \bP(S_n=0, \bar H_1 >0)  \,.
\end{equation}
(The strict inequality $>$ in place of $\geq$ is due to the fact that we are considering \emph{weak}  rather than \emph{strong} ladder heights.)
The reasoning in~\cite[Prop.~6]{AD99} cannot be reproduced identically to obtain that $f_n^+(0) \sim  \frac{1}{n} \bP(S_n=0)\bP( \bar H_1 >0)$ but we can easily adapt the proof; we simply need to update the use of Iglehart's lemma.

Our goal is to show that $\lim_{n\to\infty}\bP(\bar H_1 >0 \mid S_n =0) = \bP(\bar H_1 >0)$ or equivalently 
\begin{equation}
\label{eq:PH1cond}
\lim_{n\to\infty}\bP(\bar H_1 =0 \mid S_n =0) = \bP(\bar H_1 =0) \,.
\end{equation}
Since $\bP(S_n =0) \sim \frac{1}{a_n} f_{\alpha}(0)$ as $n\to\infty$ by the local limit theorem, see e.g.\ \cite[Thm.~4.2.1]{IL71}, combined with~\eqref{eq:AD99}, this would end the proof of Lemma~\ref{lem:kappa}.

To obtain~\eqref{eq:PH1cond}, let us start by writing
\[
\bP(\bar H_1 =0 \mid S_n =0) = \sum_{k=1}^{n} \bP(\bar T_1 = k, S_k=0 \mid S_n =0) =\sum_{k=1}^{n} \bP(\bar T_1 = k , S_k=0 ) \frac{\bP(S_{n-k}=0)}{\bP(S_n =0)} \,.
\]
To use similar notation as in~\cite{AD99}, denote $d_k := \bP(\bar T_1 = k , S_k=0) = \bP(\bar T_1 = k , \bar H_1=0) $ and $c_j:=\bP(S_j=0)$ so we need to show that $\lim_{n\to\infty} \sum_{k=1}^{n} d_k \frac{c_{n-k}}{c_n} = \sum_{k=1}^{\infty} d_k$
since it is clear that $\sum_{k=1}^{\infty}=\bP(\bar H_1=0)$.

Note that $c_{n} \sim f_{\alpha}(0)/a_n$  by the local limit theorem, with $a_n$ regularly varying with exponent $1/\alpha$.
Therefore, we have $c_{n-k}/c_n \to 1$ for any fixed $k$ and $c_{n-k}/c_n \leq C$ uniformly for $k\geq n/2$: applying the dominated convergence theorem we get that
\[
\lim_{n\to\infty} \sum_{k=1}^{n/2} d_k \frac{c_{n-k}}{c_n} = \sum_{k=1}^{\infty} d_k = \bP(\bar H_1 =0) \,.
\]
For the remaining term, using that $d_k =f_k^+(0) = \frac{1}{k} \bP( S_k=0, \bar H_1 >0)$ by~\eqref{eq:AD99}, we have $d_k\leq \frac1k c_k$ for all $k\geq 1$, so
\[
\sum_{k=n/2}^{n} d_k \frac{c_{n-k}}{c_n} \leq \sum_{k=n/2}^{n} \frac1k \frac{c_k}{c_{n}} c_{n-k} \leq \frac{C'}{n} \sum_{j=0}^{n/2} c_j \,, 
\]
where we have used that $ \frac1k \frac{c_k}{c_n} \leq C' \frac{1}{n}$ uniformly for $k\geq n/2$, again thanks to the regular variation of $c_n$.
Now, we can simply use that $c_j\to 0$ as $j\to\infty$ to get that the Ces\`aro mean $\frac1n \sum_{j=0}^{n/2} c_j$ goes to $0$.
This ends the proof of~\eqref{eq:PH1cond} and hence of Lemma~\ref{lem:kappa}.
\end{proof}

\subsection{About the critical behavior of the free energy of the wetting model}
\label{app:criticF}

Our goal here is to prove the asymptotic for the free energy that have been collected in Proposition~\ref{prop:criticF}.
First of all, let us observe that Assumption~\ref{hyp:1} is equivalent to having the following properties on the distribution of $(X_i)_{i\geq 1}$; we refer to \cite[IX.8]{FellerII}.

\smallskip
$\ast$ When $\alpha \in(0,2)$, there is some slowly varying function $\gp(\cdot)$ and constants $p,q > 0$ (take $p+q=1$ for normalization purposes) such that, as $x\to\infty$,
\begin{equation}
\label{hyp:XY}
\bP(X_i >x) \sim   p \gp(x) x^{-\alpha}  \qquad \text{and} \qquad  \bP(X_i<-x) \sim   q \gp(x) x^{-\ga} \, .
\end{equation}
This is equivalent to the fact that there exist sequences $(a_n)_{n\geq 0}$, $(b_n)_{n\geq 0}$ such that $\frac{1}{a_n}(S_n-b_n)$ converges in distribution to some $\alpha$-stable random variable. In order to have $b_n \equiv 0$, one must have $\bE[X_i]=0$ in the case $\alpha \in (1,2)$ and $p=q=\frac12$ in the case $\alpha\in (0,1]$; in the case $\alpha=1$, one also needs to have that $L(x)^{-1}\bE[X_1 \ind_{\{|X_1| \leq x\}}]$ converges as $x\to\infty$.

\smallskip
$\ast$ For $\ga=2$, then setting
\begin{equation}
\label{def:sigmax}
\sigma^2(x):=\bE \big[ \big(X_i\big)^2 \ind_{\{| X_i | \leq x\}} \big] \, ,
\end{equation}
we have that Assumption~\ref{hyp:1} holds if and only if $\bE[X_i]=0$ and $\sigma^2$ is slowly varying at $+\infty$.

\smallskip
We may define $(a_n)_{n\geq 0}$ up to asymptotic equivalence by the following relations
\begin{equation}
\label{def:an}
\bP(|X_1|>a_n)\sim \gp(a_n) a_n^{-\alpha} \sim \frac{1}{n}   \quad \text{ if } \alpha\in (0,2) \,,\qquad
\sigma^2(a_n) a_n^{-2} \sim \frac{1}{n}  \quad \text{ if } \alpha=2 \,.
\end{equation}
Then, we have that $\frac{1}{a_n S_n}$ converges in distribution to an $\alpha$-stable random variable $\mathcal{Z}$.

\subsubsection*{Asymptotic of $\mathcal K(\vartheta)$}

We now use Lemma~\ref{lem:kappa} (or~\eqref{def:K}) to obtain the asymptotic behavior of $\mathcal{K}(\vartheta)$ as $\vartheta \downarrow 0$.
The proof is standard and follows the lines of \cite[Thm.~2.1]{Giac07}, but we provide it for the sake of completeness.

\begin{lemma}
\label{lem:mathcalK}
Recall that $\mathcal{K}(\vartheta) = \sum_{n=1}^{\infty} e^{-\vartheta n} K(n)$, with $K(n) = \frac1\kappa f_n^+(0) = L(n) n^{- (1+\frac{1}{\alpha})}$.
Then we have, as $\vartheta \downarrow 0$
\[
1- \mathcal{K}(\vartheta) \sim 
\begin{cases}
  \vartheta \, m_{1/\vartheta} & \quad \text{ if }\alpha \in(0,1] \,,\\
 \alpha \Gamma(\frac{\alpha-1}{\alpha}) \vartheta K(1/\vartheta) & \quad \text{ if } \alpha \in (1,2] \,,
\end{cases}
\]
where $m_x = \sum\limits_{n=1}^{x} n K(n)$ converges if $\alpha<1$ and is slowly varying if $\alpha=1$.
\end{lemma}

\begin{proof}
Since $\sum_{n\geq 1} K(n)=1$, we simply write 
\[
1- \mathcal{K}(\vartheta) = \sum_{n=1}^{\infty} (1-e^{-\vartheta n}) K(n) = \sum_{n=1}^{\infty} (1-e^{-\vartheta n}) L(n) n^{- (1+\frac1\alpha)} \,.
\]

If $\sum_{n=1}^{\infty} n K(n)<+\infty$, which contains the case $\alpha<1$ in view of the fact that $K(n) =L(n) n^{- (1+\frac{1}{\alpha})}$, we get by dominated convergence that 
\[
\lim_{\vartheta \downarrow 0}\sum_{n=1}^{\infty} \frac{1-e^{-\vartheta n}}{\vartheta} K(n) =\sum_{n=1}^{\infty} n K(n) = \kappa^{-1} \sum_{n=1}^{\infty} n f_n^+(0) = \lim_{\vartheta \downarrow 0} \, m_{1/\vartheta} \,.
\]

If $\sum_{n=1}^{\infty} n f_n^+(0) =+\infty$ and $\alpha=1$, then $m_{x} := \sum\limits_{n=1}^{x}L(n)n^{-1}$ is slowly varying at $+\infty$ and verifies $m_x/L(x) \to\infty$, see~\cite[Prop.~1.5.9a]{BGT87}.
Then, we get that
\[
1- \mathcal{K}(\vartheta)\geq \sum_{n=1}^{1/\vartheta} (1-e^{-\vartheta n}) K(n) \geq  \vartheta \sum_{n=1}^{1/\vartheta} L(n)n^{-1} - c \vartheta^2\sum_{n=1}^{1/\vartheta} L(n) = (1+o(1)) \vartheta \, m_{1/\vartheta} \,,
\]
where we have used that $\sum_{n=1}^{1/\vartheta} L(n) \sim  \vartheta^{-1} L(1/\vartheta)$, with $L(1/\vartheta) = o(m_{1/\vartheta})$.
On the other hand,
\[
1- \mathcal{K}(\vartheta)\leq \sum_{n=1}^{1/\vartheta} (1-e^{-\vartheta n}) K(n)  + \sum_{n>1/\vartheta} K(n) 
\leq \vartheta \sum_{n=1}^{1/\vartheta} L(n)n^{-1} + \sum_{n>1/\vartheta} L(n)n^{-2} = (1+o(1)) \vartheta \, m_{1/\vartheta}  \,,
\]
where we have used that $\sum_{n>1/\vartheta} L(n)n^{-2} \sim   L(1/\vartheta) \vartheta$ with $L(1/\vartheta) = o(m_{1/\vartheta})$.

It remains to treat the case $\alpha\in (1,2]$. This time, we use a Riemann sum approximation (see also~\cite[Thm.~1.7.1 and Cor.~8.1.7]{BGT87}) to get that
\[
1- \mathcal{K}(\vartheta) = L(1/\vartheta) \vartheta^{\frac{1}{\alpha}} \sum_{n=1}^{\infty}  \vartheta \frac{1-e^{-\vartheta n}}{ (\vartheta n)^{1+\frac1\alpha}}  \frac{L(n)}{L(1/\vartheta )}
 \sim L(1/\vartheta) \vartheta^{\frac{1}{\alpha}}  \int_0^{\infty} \frac{1-e^{-u}}{u^{1+\frac{1}{\alpha}}} \dd u \,,
\]
with $\int_0^{\infty} \frac{1-e^{-u}}{u^{1+\frac{1}{\alpha}}} \dd u = \alpha \Gamma(1-\frac1\alpha)$ by some integration by parts.
\end{proof}

\subsubsection*{Asymptotic of the free energy: proof of Proposition~\ref{prop:criticF}}

Now we are ready to prove Proposition~\ref{prop:criticF}.
From~\eqref{def:freeenergy}, we get that for $\lambda>\lambda_c := 1/\kappa$,
\[
1- \mathcal{K}(\f(\lambda)) = 1 - (\kappa\lambda)^{-1} = \frac{\lambda -\lambda_c}{\lambda} \,.
\]
Since we know that $\f(\lambda)\downarrow 0$ as $\lambda\downarrow \lambda_c$, we may use Lemma~\ref{lem:mathcalK} to obtain the following.
To simplify notation, we write $\f_u := \f(\lambda_c+u)$, which goes to $0$ as  $u\downarrow 0$.

\smallskip
$\ast$ If $\sum_{n=1}^{\infty} nK(n) <+\infty$, or equivalently (by Lemma~\ref{lem:kappa}) if $\sum_{n=1}^{\infty} \frac{1}{a_n}<+\infty$ that is if $(S_n)_{n\geq 0}$ is transient, then, as $u\downarrow 0$,
\[
\f_u \sum_{n=1}^{\infty} nK(n)  \sim  \frac{u}{\lambda_c} = \kappa u\,.
\]
Since $K(n) = \kappa^{-1} f_n^+(0)$ and $\lambda_c = \kappa^{-1}$, we get~\eqref{eq:criticF2}.

\smallskip
$\ast$ If $\alpha \in (1,2]$, then as $u\downarrow 0$,
\[
 \alpha \Gamma(\tfrac{\alpha-1}{\alpha})\, \f_u \, K(1/\f_u) \sim \kappa^{-1} c_0 \alpha \Gamma(\tfrac{\alpha-1}{\alpha}) \frac{1}{a_{1/\f_u}} \sim  \kappa u\,,
\]
where we also have used Lemma~\ref{lem:kappa} to get that $nK(n) \sim \kappa^{-1} c_0/a_n$ as $n\to\infty$.
Using the relation~\eqref{def:an} that defines $(a_n)_{n\geq 0}$, we get that:
\[
\begin{split}
\f_u \sim \bP(|X_1|>\tfrac{c_2}{u}) &  \quad \text{ if } \alpha \in (1,2)\,, \\
\f_u \sim \sigma^{2}(\tfrac{c_2}{u})  \big(\tfrac{c_2}{u}\big)^{-2} &  \quad \text{ if } \alpha =2 \,, \\
\end{split}
\]
with $c_2:=c_0\alpha \Gamma(\frac{\alpha-1}{\alpha})/\kappa^2$.
This gives~\eqref{eq:criticF4}-\eqref{eq:criticF5}, using that $c_0 = \frac{1}{\sqrt{2\pi}} \bP(\bar H_1>0)$ in the case $\alpha=2$, so that $c_2=\sqrt{2} \bP(\bar H_1>0)/\kappa^2$.

\smallskip
$\ast$ If $\alpha =1$ and $\sum_{n=1}^{\infty} nK(n) = \sum_{n=1}^{\infty} L(n)n^{-1}  =+\infty$, we obtain that
$\f_u m_{1/\f_u} \sim \kappa u$, and we have to invert this relation.
Let us introduce $v_t$ such that $L(v_t) v_t^{-1} \sim t^{-1}$ and $w_t := t m_{v_t}$: then by~\cite[Lem.~4.3]{Ber19a} we get that $m_{w_t}\sim m_{v_t}$ as $t\to\infty$, so that we have $m_{w_t} v_t^{-1} \sim t^{-1}$ as $t\to\infty$.
Using this relation, we get that
\[
1/\f_u  \sim w_{1/\kappa u} \qquad \text{ so } \qquad  \f_u \sim \frac{\kappa u}{ m_{v_{1/u}}} \,,
\]
using also that $m$ is slowly varying and that $v$ is regularly varying.
Now, we can use Lemma~\ref{lem:kappa} to get that
\[
m_{v_{1/u}} \sim \kappa^{-1} c_0 \sum_{n=1}^{v_{1/u}} \frac{1}{a_n}  \sim \kappa^{-1} c_0 \int_1^{v_{1/u}} \frac{1}{a_t} \dd t \sim \kappa^{-1} c_0 \int_1^{a_{v_{1/u}}} \frac{1}{s\gp(s)} \dd s \,,
\]
where we have used a change of variable $s=a_t$ (so $t \sim s/\gp(s)$, see~\eqref{def:an}), in the spirit of~\cite[p.~36]{Ber19a}.
Now, notice that by Lemma~\ref{lem:kappa} we have $a_n \sim \kappa c_0 /nK(n) \sim \kappa c_0 n/L(n)$, so by definition of $v_t$ we get that $a_{v_{1/u}} \sim \kappa c_0/u$. Since $u\mapsto \int_1^{1/u} \frac{1}{s\gp(s)} \dd s$ is slowly varying, we get that
\[
m_{v_{1/u}} \sim \kappa^{-1} c_0 \int_1^{1/u} \frac{1}{s\gp(s)} \dd s \sim \kappa^{-1} c_0 \int_1^{1/u} \frac{ \dd s}{s^2 \bP(|X_1|>s)} \,.
\]
This concludes the proof of~\eqref{eq:criticF3}.\qed

%
%

\subsection{About the SOS (or Laplace) wetting model}
\label{app:Laplace}

In this section, we derive the free energy for the wetting model of Section~\ref{ex:exponential}.
We only deal with the exponential random walk (or SOS model) since the proof does not appear clearly in the literature, but it could also be adapted to the lazy and geometric random walks of Sections~\ref{ex:lazyRW}-\ref{ex:geomRW}, only with more tedious calculations.

%
%

First of all, let us notice that, by the memoryless property of the exponential random variable, we have that $\bar T_1$ and $\bar H_1$ are independent; and $\bar H_1$ follows an exponential distribution with $f_{\bar H_1}(0)=\gamma$.
Therefore, we have that
\[
f_n^+(0)  = \gamma \bP(\bar T_1=n) \,,\qquad K(n) = \bP(\bar T_1 =n) \,.
\]
Let us compute the Laplace transform $\mathcal{K}(\vartheta) = \sum_{n=1}^{\infty} e^{-\vartheta n} K(n)$.
For $\vartheta> 0$ fixed, we introduce the martingale $(M_n)_{n\geq 0}$ defined by 
\[
M_n := e^{-\theta S_n-\vartheta n} \,,
\]
where $\theta = \gamma \sqrt{1-e^{-\vartheta}} <\gamma$ is chosen so that\footnote{To compute $\bE[e^{-\theta X_1}]$, we can use the representation $X_1=Y-Y'$ with $Y,Y'$ independent $ \mathrm{Exp}(\gamma)$ random variables, so $\bE[e^{uY}]= \frac{1}{1-u/\gamma}$ for any $|u|<\gamma$.} 
$\bE[e^{-\theta X_1}] = \frac{1}{1-\theta^2/\gamma^2} = e^{-\vartheta}$.

Applying the stopping time theorem, together with dominated convergence (since $S_{n\wedge \bar T_1} \geq S_{\bar T_1} $, with $-S_{\bar T_1}=\bar H_1$ an exponential distribution), we obtain that
\[
1 =\lim_{n\to\infty}\bE[M_{n\wedge \bar T_1}] = \bE\big[ e^{-\theta S_{\bar T_1}-\vartheta \bar T_1} \big] = \bE\big[e^{\theta \bar H_1} \big] \bE\big[ e^{-\vartheta \bar T_1} \big]\,, 
\]
where we have used again that $\bar H_1=-S_{\bar T_1}$ is independent of $\bar T_1$ in the last identity. 
We therefore end up with
\[
\mathcal{K}(\vartheta) = \bE\big[e^{\theta \bar H_1} \big]^{-1} = 1- \frac{\theta}{\gamma} =1- \sqrt{1-e^{-\vartheta}} \,.
\]
The relation $\mathcal{K}( \f(\lambda)) = (\kappa\lambda)^{-1}$ from Theorem~\ref{prop:pinning} (with $\kappa=\gamma$) therefore gives
\[
\f(\lambda) = \ln\Big( \frac{(\gamma \lambda)^2}{2\gamma \lambda -1} \Big)   \,.
\]

\section{Local large deviations: proof of Lemmas~\ref{lem:LDPequal} and~\ref{lem:Q-LDP}}
\label{sec:LDPequal}

\begin{proof}[Proof of Lemma~\ref{lem:LDPequal}]
Let us focus on the upper local large deviation; the lower counterpart is proven in an identical manner. We need to show that 
\[
\liminf_{n\to\infty} \frac1n \log \bP(S_n = x_n) \geq - \I_+(x) \,,
\]
the other inequality being given by~\eqref{eq:boundLDP} and the fact that $\I_+$ is continuous at $x$.

First of all, if $t_0^+ =0$, \textit{i.e.}\ if $\Lambda_+(t)=+\infty$ for all $t>0$, then $\I_+(x)=0$ for all $x\geq 0$.
In that case, we get that $\limsup_{k\to\infty} \frac{1}{k} \log \bP(X_1=k) = 0$ (one can easily check that otherwise $t_0^+ >0$), so we can choose some $k_0$ such that $\bP(X_1=k_0)\geq e^{- \varepsilon k_0}$ and $k_0 \geq  3x$.
Then, setting $n_0=\lfloor x_n /k_0\rfloor$, we have that 
\begin{equation}
\label{eq:Sn=xn-lower1}
\bP\big( S_n =x_n \big) 
\geq \bP\big( X_i =k_0 \text{ for all } 1\leq i\leq n_0 \big)
 \bP\big( S_{n-n_0} = x_n - k_0 n_0 \big)\,.
\end{equation}
Notice that $n_0 \leq n/2$ (provided that $n$ is large enough) and that $x_n - k_0 n_0 \in \{0,\ldots, k_0\}$.
Using the local central limit theorem, see e.g.~\cite[Ch.~9, \S50]{GK54}, we therefore get that 
\[
\bP\big( S_n = x_n \big) 
\geq  \big( e^{- \varepsilon k_0} \big)^{n_0}  \frac{c}{a_n} \geq \frac{c}{a_n} e^{- \varepsilon x_n } \,, 
\]
so $\liminf_{n\to\infty} \frac1n \log \bP(S_n =x_n) \geq - \gep x$, recalling that $\lim_{n\to\infty} \frac1n x_n =x$. Since $\gep$ is arbitrary, this concludes the proof in the case $t_0^+=0$.

Let us now turn to the case where $t_0^+>0$. Then, fixing $\gep$ small enough so that $\frac{x}{1-\gep} < \bar x_+$, we have that
\[
\begin{split}
\bP\big( S_{(1-\gep)n} \in  [(1-\gep^2) x_n, x_n ] \big) 
 &= \bP\big( S_{(1-\gep)n} \geq (1-\gep^2) x_n\big)
-\bP\big( S_{(1-\gep)n} > x_n \big) \\
& = e^{- (1+o(1))\I_+( (1+\gep) x) (1-\gep)n } - e^{- (1+o(1))\I_+(\frac{x}{1-\gep}) (1-\gep)n} \,,
\end{split} 
\]
where we have used that $\lim_{n\to\infty} \frac1n x_n =x$ and that $\I_+$ is continuous at $ (1+\gep)x$ and $x/(1-\gep)$.
Then the second term is negligible compared to the first one since $\I_+(\frac{x}{1-\gep}) > \I_+((1+\gep)x)$, because~$\I_+$ is increasing on $[0,\bar x_+)$. We therefore get that
\[
\liminf_{n\to\infty} \frac1n \log \bP\big( S_{(1-\gep)n} \in  [(1-\gep^2) x_n, x_n ] \big)  =  - (1-\gep) \I_+((1+\gep) x) \,.
\]
We now write
\[
\begin{split}
\bP\big( S_n =x_n \big) 
& \geq \sum_{y=(1-\gep^2) x_n}^{x_n} \bP\big( S_{(1-\gep)n} =y \big) \bP(S_{\gep n} = x_n- y) \\
&\geq \bP\big( S_{(1-\gep)n} \in  [(1-\gep^2) x_n, x_n ) \big) 
\times \inf_{z\in [0,\gep^2 x_n]} \bP(S_{\gep n} = z) \,. 
\end{split}
\]
It remains to get a lower bound on the last term.
Let $k'_0:=\inf\{k\geq 1, \bP(X_1=k)>0\}$. 
Then, similarly as in~\eqref{eq:Sn=xn-lower1}, setting $n_0=n_0(z) = \lfloor z/k'_0\rfloor$, we get that 
\[
\bP(S_{\gep n} = z) \geq \bP\big(X_i = k'_0 \text{ for all } 1\leq i\leq n_0\big) \bP(S_{\gep n - n_0} = z- k'_0 n_0) \,.
\]
Again, $n_0 \leq z \leq \gep^2 x_n$, so $\gep n - n_0 \geq \gep n \left(1-\gep \frac{x_n}{n}\right)$. For $n$ big enough, we have $\frac{x_n}{n} \le 2x$, then if $\gep \le \frac{1}{4x}$ we have $\gep n - n_0 \ge \gep \frac{n}{2} $. Since $z-k'_0 n_0 \in \{0,\ldots, k'_0\}$ we can use the local central limit theorem to bound the last term by a constant times $a_{\gep n}^{-1}$, uniformly in $z\in [0,\gep^2 x_n]$. We therefore end up with
\[
\inf_{z\in [0,\gep^2 x_n]} \bP(S_{\gep n} = z) \geq  \frac{c}{a_{\gep n}} \bP(X_1=k'_0)^{\gep^2 n}  \,.
\]
All together, we get that 
\[
\liminf_{n\to\infty} \frac1n \log \bP\big( S_n =x_n \big) \geq - (1-\gep)\I_{+}\big( (1+\gep) x \big) + \gep^2 \log \bP(X_1=k'_0)\,.
\]
Since $\gep$ is arbitrary and $\I_+$ is continuous on $[0,\bar x_+]$, this concludes the proof.
\end{proof}

\begin{proof}[Proof of Lemma~\ref{lem:Q-LDP}]
The proof follows the same lines as that of Lemma~\ref{lem:LDPequal}.

First of all, let us treat the case $t_0^+ =0$. 
Then, as in~\eqref{eq:Sn=xn-lower1}, with the same definition for $k_0$ and with $n_0 = \lfloor \frac{x_n}{k_0} \rfloor$,  we have that
\[
Q_+(n, x_n) \geq \bP\big( X_i =k_0 \big)^{n_0}
 \bP\big( S_{n-n_0} = x_n - k_0 n_0 , S_i > 0  \text{ for all } 1\leq i\leq n-n_0-1\big)\,.
\]
Now, let $k_1>0$ be such that $\bP(S_1>0, \ldots, S_{k_1-1}>0, S_{k_1} = y) >0$ for all $y\in \{0,\ldots, k_0\}$, which exists by aperiodicity and let $c = \inf_{0\leq y\leq k_0} \bP(S_1>0, \ldots, S_{k_1-1}>0, S_{k_1} = y) >0$, which depends only on $k_0$.
Recalling that $x_n-k_0n_0 \in \{0,\ldots, k_0\}$, we therefore get that
\[
 \bP\big( S_{n-n_0} = x_n - k_0 n_0 , S_i > 0  \text{ for all } 1\leq i\leq n-n_0-1\big) \geq c f_{n-n_0-k_1}^+(0)  \,,
\]
with $f_n^+(0)= \bP(S_1>0, \ldots, S_{n-1}>0, S_n=0)$ defined in~\eqref{def:kappan}.
Now, thanks to Lemma~\ref{lem:kappa}, and since $n-n_0-k_1\geq n/4$ for $n$ large enough, we get that
\[
Q_+(n, x_n) \geq \frac{c'}{n a_{n}} e^{- \gep k_0 n_0} \geq \frac{c'}{n a_{n}} e^{- \gep x_n}\,.
\]
We therefore end up with $\liminf_{n\to\infty} \frac1n \log Q_+(n, \lfloor xn \rfloor) \geq - \gep x$, which gives the result since $\gep$ is arbitrary, recalling that $\I_+(x)=0$.

\smallskip
Let us now turn to the case where $t_0^+>0$.
We focus first on the case $\lim_{n\to\infty} \frac1n x_n = x \in (0,\bar x_+)$.
Denoting again $k'_0=\min\{k\geq 1, \bP(X_1=k)>0\}$, we have that
\[
Q_+(n, x_n)
\geq \bP(X_1 =k'_0)^{\gep n} \bP\big( S_{(1-\gep)n} = x_n - \gep n k'_0 , S_{i} > - \gep n k'_0 \text{ for all } i\leq (1-\gep) n \big) \,.
\]
Then, we write
\begin{equation}
\label{eq:twotermsQnx}
\begin{split}
\bP\big( S_{(1-\gep)n} &= x_n - \gep n k'_0 , S_{i} > - \gep n k'_0 \text{ for all } i\leq (1-\gep n) \big) \\
& = \bP\big( S_{(1-\gep)n} = x_n - \gep n k'_0 \big) - \bP\big( S_{(1-\gep)n} = x_n - \gep n k'_0 , \min_{1\leq i\leq (1-\gep)n} S_{i} \leq  - \gep n k'_0\big) \,.
\end{split}
\end{equation}
For the first term, we can use Lemma~\ref{lem:LDPequal} to get that it is $\exp(- (1+o(1)) (1-\gep)n\I_+(\frac{x- \gep k'_0}{1-\gep}) n)$.
For the other term, decomposing over the first time where $S_i\leq -\gep n k'_0$, we easily get by the strong Markov property that
\[
\begin{split}
\bP\big( S_{(1-\gep)n} = x_n - \gep n k'_0 , \min_{1\leq i\leq (1-\gep)n} S_{i} \leq  - \gep n k'_0\big)
& \leq \sup_{ 1\leq j\leq (1-\gep) n, y\geq 0}  \bP\big( S_{j} = x_n + y \big) \,.
\end{split}
\]
Then, using~\eqref{eq:boundLDP} and the fact that $\I_+$ is non-decreasing, we get that this is bounded by
\[
  \sup_{1\leq j\leq (1-\gep)n} e^{- j \I_+ ( \frac{x_n}{j})} \leq  e^{- (1-\gep) n \I_+ ( \frac{x_n}{(1-\gep)_n} )} \leq e^{- (1-\gep) n \I_+ ( \frac{x-\gep^2}{1-\gep} )} \,,
\]
where the first identity holds by convexity of $\I_+$, and the second one provided that $n$ is large enough so that $\frac1n x_n \geq  x -\gep^2$.

Since $(1-\gep) \I_+ ( \frac{x-\gep^2}{1-\gep} ) > (1-\gep) \I_+ ( \frac{x-\gep k'_0}{1-\gep}) $ because $\I_+$ is increasing, we get that the second term in~\eqref{eq:twotermsQnx} is negligible, so that
\[
\lim_{n\to\infty} \frac1n \log \bP\big( S_{(1-\gep)n}  = x_n - \gep n k'_0 ,\min_{1\leq i\leq (1-\gep)n} S_{i} \leq  - \gep n k'_0 \big) =  (1-\gep)\I_+\Big(\frac{x- \gep k'_0}{1-\gep}\Big)  \,.
\]
All together, this gives that
\[
\liminf_{n\to\infty} \frac 1n \log Q_+(n,x_n)
\geq \gep \log \bP(X_1 =k'_0) - (1-\gep)\I_+\Big(\frac{x- \gep k'_0}{1-\gep}\Big) \,,
\]
which concludes the proof since $\gep$ is arbitrary and $\I_+$ is continuous at $x <\bar x_+$.

For the case $\lim_{n\to\infty} \frac1n x_n =0$, we fix $\gep>0$ and we use the following lower bound (omitting the integer parts for simplicity):
\[
Q_+(n,x_n) \geq \bP(S_{n/2}=\gep n, S_n=x_n, S_1>0, \ldots, S_{n-1}>0 ) \geq Q_+(n/2, \gep n) Q_-(n/2,\gep n -x_n)\,,
\] 
where we have used Markov's property and the duality property
\begin{multline*}
\bP(S_{n/2}>x_n, \ldots, S_{n-1}>x_n , S_n=x_n  \mid S_{n/2} =\gep n) \\
= \bP(S_1<0,\ldots, S_{n/2}  <0, S_{n/2} = \gep n-x_n) = Q_-(n/2, \gep n-x_n).
\end{multline*}
Therefore, from the case $x>0$ above we get that $\liminf_{n\to\infty} \frac1n \log Q_+(n,x_n) \geq - \I_+(\gep/2) - \I_-(\gep/2)$.
This concludes the proof, by taking $\gep\downarrow 0$, since $\I_+$, $\I_-$ are continuous at $0$.
\end{proof}

\end{appendix}

\subsection*{Acknowledgment}
The authors would like to warmly thank Francesco Caravenna for interesting discussions on the wetting model and Yvan Velenik and Giambattista Giacomin for pointing out several relevant references.
QB acknoweledges the support of grant ANR-22-CE40-0012.

\bibliographystyle{alphaabbr}
\bibliography{biblio.bib}

\end{document}